\documentclass[paper=a4, fontsize=11pt]{scrartcl} % A4 paper and 11pt font size
\usepackage[T1]{fontenc} % Use 8-bit encoding that has 256 glyphs
\usepackage{fourier} % Use the Adobe Utopia font for the document - comment this line to return to the LaTeX default
\usepackage[english]{babel} % English language/hyphenation
\usepackage{amsmath,amsfonts,amsthm} % Math packages
\usepackage{tikz-cd}
\usepackage{lipsum} % Used for inserting dummy 'Lorem ipsum' text into the template
\usepackage{shuffle}
\usepackage{sectsty} % Allows customizing section commands
\allsectionsfont{\centering \normalfont\scshape} % Make all sections centered, the default font and small caps
\usepackage{ amssymb }
\usepackage{fancyhdr} % Custom headers and footers
\usepackage{abstract}
\usepackage{comment}
\usepackage{stmaryrd}
\usepackage{cprotect}
\usepackage{hyperref}
\usepackage{mathrsfs}
\usepackage{enumitem}
\usepackage{mathtools}
\usepackage{array}
\usepackage{tabularx}
\usepackage{authblk}

\DeclareMathAlphabet{\mathcal}{OMS}{cmsy}{m}{n}

\colorlet{darkblue}{blue!90!black}
\colorlet{darkred}{red!90!black}

\def\CT{\mathcal{T}}

\newcommand{\fraks}{\mathfrak{s}}
\DeclareMathOperator*{\argmin}{arg\,min}

\DeclareMathOperator{\Span}{span}
\DeclareMathOperator*{\esssup}{ess\,sup}
\DeclareMathOperator*{\supp}{\text{supp}}
\newcommand{\eps}{\varepsilon}

\newcommand{\tzero}{|_{t=0}}

\makeatletter
\pgfdeclareshape{crosscircle}
{
	\inheritsavedanchors[from=circle] % this is nearly a circle
	\inheritanchorborder[from=circle]
	\inheritanchor[from=circle]{north}
	\inheritanchor[from=circle]{north west}
	\inheritanchor[from=circle]{north east}
	\inheritanchor[from=circle]{center}
	\inheritanchor[from=circle]{west}
	\inheritanchor[from=circle]{east}
	\inheritanchor[from=circle]{mid}
	\inheritanchor[from=circle]{mid west}
	\inheritanchor[from=circle]{mid east}
	\inheritanchor[from=circle]{base}
	\inheritanchor[from=circle]{base west}
	\inheritanchor[from=circle]{base east}
	\inheritanchor[from=circle]{south}
	\inheritanchor[from=circle]{south west}
	\inheritanchor[from=circle]{south east}
	\inheritbackgroundpath[from=circle]
	\foregroundpath{
		\centerpoint%
		\pgf@xc=\pgf@x%
		\pgf@yc=\pgf@y%
		\pgfutil@tempdima=\radius%
		\pgfmathsetlength{\pgf@xb}{\pgfkeysvalueof{/pgf/outer xsep}}%  
		\pgfmathsetlength{\pgf@yb}{\pgfkeysvalueof{/pgf/outer ysep}}%  
		\ifdim\pgf@xb<\pgf@yb%
		\advance\pgfutil@tempdima by-\pgf@yb%
		\else%
		\advance\pgfutil@tempdima by-\pgf@xb%
		\fi%
		\pgfpathmoveto{\pgfpointadd{\pgfqpoint{\pgf@xc}{\pgf@yc}}{\pgfqpoint{-0.707107\pgfutil@tempdima}{-0.707107\pgfutil@tempdima}}}
		\pgfpathlineto{\pgfpointadd{\pgfqpoint{\pgf@xc}{\pgf@yc}}{\pgfqpoint{0.707107\pgfutil@tempdima}{0.707107\pgfutil@tempdima}}}
		\pgfpathmoveto{\pgfpointadd{\pgfqpoint{\pgf@xc}{\pgf@yc}}{\pgfqpoint{-0.707107\pgfutil@tempdima}{0.707107\pgfutil@tempdima}}}
		\pgfpathlineto{\pgfpointadd{\pgfqpoint{\pgf@xc}{\pgf@yc}}{\pgfqpoint{0.707107\pgfutil@tempdima}{-0.707107\pgfutil@tempdima}}}
	}
}
\makeatother

\colorlet{symbols}{blue!90!black}
\colorlet{testcolor}{green!60!black}

\def\symbol#1{\textcolor{symbols}{#1}}
\def\1{\mathbf{\symbol{1}}}

\usetikzlibrary{shapes.misc}
\usetikzlibrary{shapes.symbols}
\usetikzlibrary{snakes}
\usetikzlibrary{decorations}
\usetikzlibrary{decorations.markings}

\def\drawx{\draw[-,solid] (-3pt,-3pt) -- (3pt,3pt);\draw[-,solid] (-3pt,3pt) -- (3pt,-3pt);}
\tikzset{
	root/.style={circle,fill=testcolor,inner sep=0pt, minimum size=2mm},
	dot/.style={circle,fill=black,inner sep=0pt, minimum size=1mm},
	var/.style={circle,fill=black!10,draw=black,inner sep=0pt, minimum size=2mm},
	dotred/.style={circle,fill=black!50,inner sep=0pt, minimum size=2mm},
	generic/.style={semithick,shorten >=1pt,shorten <=1pt},
	dist/.style={ultra thick,draw=testcolor,shorten >=1pt,shorten <=1pt},
	testfcn/.style={ultra thick,testcolor,shorten >=1pt,shorten <=1pt,<-},
	testfcnx/.style={ultra thick,testcolor,shorten >=1pt,shorten <=1pt,<-,
		postaction={decorate,decoration={markings,mark=at position 0.6 with {\drawx}}}},
	kprime/.style={semithick,shorten >=1pt,shorten <=1pt,densely dashed,->},
	kprimex/.style={semithick,shorten >=1pt,shorten <=1pt,densely dashed,->,
		postaction={decorate,decoration={markings,mark=at position 0.4 with {\drawx}}}},
	kernel/.style={semithick,shorten >=1pt,shorten <=1pt,->},
	multx/.style={shorten >=1pt,shorten <=1pt,
		postaction={decorate,decoration={markings,mark=at position 0.5 with {\drawx}}}},
	kernelx/.style={semithick,shorten >=1pt,shorten <=1pt,->,
		postaction={decorate,decoration={markings,mark=at position 0.4 with {\drawx}}}},
	kernel1/.style={->,semithick,shorten >=1pt,shorten <=1pt,postaction={decorate,decoration={markings,mark=at position 0.45 with {\draw[-] (0,-0.1) -- (0,0.1);}}}},
	kernel2/.style={->,semithick,shorten >=1pt,shorten <=1pt,postaction={decorate,decoration={markings,mark=at position 0.45 with {\draw[-] (0.05,-0.1) -- (0.05,0.1);\draw[-] (-0.05,-0.1) -- (-0.05,0.1);}}}},
	kernelBig/.style={semithick,shorten >=1pt,shorten <=1pt,decorate, decoration={zigzag,amplitude=1.5pt,segment length = 3pt,pre length=2pt,post length=2pt}},
	rho/.style={dotted,semithick,shorten >=1pt,shorten <=1pt},
	renorm/.style={shape=circle,fill=white,inner sep=1pt},
	labl/.style={shape=rectangle,fill=white,inner sep=1pt},
	xi/.style={circle,fill=symbols!10,draw=symbols,inner sep=0pt,minimum size=1.2mm},
	xix/.style={crosscircle,fill=symbols!10,draw=symbols,inner sep=0pt,minimum size=1.2mm},
	xib/.style={circle,fill=symbols!10,draw=symbols,inner sep=0pt,minimum size=1.6mm},
	xibx/.style={crosscircle,fill=symbols!10,draw=symbols,inner sep=0pt,minimum size=1.6mm},
	not/.style={circle,fill=symbols,draw=symbols,inner sep=0pt,minimum size=0.5mm},
	>=stealth,
}
\makeatletter
\def\DeclareSymbol#1#2#3{\expandafter\gdef\csname MH@symb@#1\endcsname{\tikz[baseline=#2,scale=0.15,draw=symbols]{#3}}\expandafter\gdef\csname MH@symb@#1s\endcsname{\scalebox{0.7}{\tikz[baseline=#2,scale=0.15,draw=symbols]{#3}}}}
\def\<#1>{\csname MH@symb@#1\endcsname}
\makeatother

\DeclareSymbol{Xi22}{0.5}{\draw (0,0) node[xi] {} -- (-1,1) node[not] {} -- (0,2) node[xi] {};}

\DeclareSymbol{Xi2}{-2}{\draw (0,-0.25) node[xi] {} -- (-1,1) node[xi] {};}
\DeclareSymbol{Xi3}{0}{\draw (0,0) node[xi] {} -- (-1,1) node[xi] {} -- (0,2) node[xi] {};}
\DeclareSymbol{Xi4}{2}{\draw (0,0) node[xi] {} -- (-1,1) node[xi] {} -- (0,2) node[xi] {} -- (-1,3) node[xi] {};}
\DeclareSymbol{Xi2X}{-2}{\draw (0,-0.25) node[xi] {} -- (-1,1) node[xix] {};}
\DeclareSymbol{XXi2}{-2}{\draw (0,-0.25) node[xix] {} -- (-1,1) node[xi] {};}

\DeclareSymbol{IXi2}{0}{\draw (0,-0.25) node[not] {} -- (-1,1) node[xi] {} -- (0,2) node[xi] {};}
\DeclareSymbol{IXi^2}{-1}{\draw (-1,1) node[xi] {} -- (0,0) node[not] {} -- (1,1) node[xi] {};}

\DeclareSymbol{XiX}{-2.8}{\node[xibx] {};}
\DeclareSymbol{Xi}{-2.8}{\node[xib] {};}
\DeclareSymbol{IXiX}{-1}{\draw (0,-0.25) node[not] {} -- (-1,1) node[xix] {};}

\DeclareSymbol{Xi3b}{-1}{\draw (-1,1) node[xi] {} -- (0,0) node[xi] {} -- (1,1) node[xi] {};}

\DeclareSymbol{IXi3}{2}{\draw (0,-0.25) node[not] {} -- (-1,1) node[xi] {} -- (0,2) node[xi] {} -- (-1,3) node[xi] {};}
\DeclareSymbol{IXi}{-2}{\draw (0,-0.25) node[not] {} -- (-1,1) node[xi] {};}
\DeclareSymbol{XiI}{-2}{\draw (0,-0.25) node[xi] {} -- (-1,1) node[not] {};}

\DeclareSymbol{Xi4b}{-1}{\draw(0,1.5) node[xi] {} -- (0,0); \draw (-1,1) node[xi] {} -- (0,0) node[xi] {} -- (1,1) node[xi] {};}
\DeclareSymbol{Xi4b'}{-1}{\draw(0,1.5) node[xi] {} -- (0,-0.2); \draw (-1,1) node[xi] {} -- (0,-0.2) node[not] {} -- (1,1) node[xi] {};}
\DeclareSymbol{Xi4c}{0}{\draw (0,1) -- (0.8,2.2) node[xi] {};\draw (0,-0.25) node[xi] {} -- (0,1) node[xi] {} -- (-0.8,2.2) node[xi] {};}
\DeclareSymbol{Xi4d}{-4.5}{\draw (0,-1.5) node[not] {} -- (0,0); \draw (-1,1) node[xi] {} -- (0,0) node[xi] {} -- (1,1) node[xi] {};}
\DeclareSymbol{Xi4e}{0}{\draw (0,2) node[xi] {} -- (-1,1) node[xi] {} -- (0,0) node[xi] {} -- (1,1) node[xi] {};}
\DeclareSymbol{Xi4e'}{0}{\draw (0,2) node[xi] {} -- (-1,1) node[xi] {} -- (0,-0.2) node[not] {} -- (1,1) node[xi] {};}

\newcommand{\mbC}{\mathbb{C}}

\newcommand{\mbE}{\mathbb{E}}

\newcommand{\mbG}{\mathbb{G}}
\newcommand{\mbH}{\mathbb{H}}
\newcommand{\mbI}{\mathbb{I}}

\newcommand{\mbN}{\mathbb{N}}

\newcommand{\mbR}{\mathbb{R}}

\newcommand{\mbZ}{\mathbb{Z}}

\newcommand{\RR}{\mathbb{R}}

\newcommand{\NN}{\mathbb{N}}

\newcommand{\frakg}{\mathfrak{g}}
\newcommand{\bfeta}{\pmb{\eta}}
\newcommand{\bfPi}{\pmb{\Pi}}

\newcommand{\dd}{\mathop{}\!\mathrm{d}}

\newcommand{\msD}{\mathscr{D}}

\newcommand{\msR}{\mathscr{R}}

\newcommand{\mcC}{\mathcal{C}}
\newcommand{\mcD}{\mathcal{D}}

\newcommand{\mcG}{\mathcal{G}}

\newcommand{\mcI}{\mathcal{I}}
\newcommand{\mcJ}{\mathcal{J}}
\newcommand{\mcK}{\mathcal{K}}

\newcommand{\mcM}{\mathcal{M}}
\newcommand{\mcN}{\mathcal{N}}

\newcommand{\mcP}{\mathcal{P}}
\newcommand{\mcQ}{\mathcal{Q}}
\newcommand{\mcR}{\mathcal{R}}
\newcommand{\mcS}{\mathcal{S}}
\newcommand{\mcT}{\mathcal{T}}

\newcommand{\mcV}{\mathcal{V}}
\newcommand{\mcW}{\mathcal{W}}

\newcommand{\mfB}{\mathfrak{B}}

\newcommand{\mfD}{\mathfrak{D}}

\newcommand{\mfG}{\mathfrak{G}}
\newcommand{\mfH}{\mathfrak{H}}

\newcommand{\mfK}{\mathfrak{K}}
\newcommand{\mfL}{\mathfrak{L}}

\newcommand{\mfZ}{\mathfrak{Z}}

\newcommand{\mfc}{\mathfrak{c}}

\newcommand{\mfg}{\mathfrak{g}}
\newcommand{\mfh}{\mathfrak{h}}

\newcommand{\mfm}{\mathfrak{m}}

\newcommand{\mfp}{\mathfrak{p}}

\newcommand{\mfr}{\mathfrak{r}}
\newcommand{\mfs}{\mathfrak{s}}

\newcommand{\bP}{\mathbf{P}}

\newcommand{\bR}{\mathbf{R}}

\newcommand{\dil}{\operatorname{D}}
\newcommand{\opP}{\operatorname{P}}
\newcommand{\opbfP}{\pmb{\operatorname{{P}}}}
\newcommand{\opG}{\operatorname{G}}
\newcommand{\opT}{\operatorname{T}}

\newcommand{\loc}{\text{loc}}
\newcommand{\ex}{\text{ex}}

\newcommand{\opH}{\operatorname{H}}

\newcommand{\vertiii}[1]{{\left\vert\kern-0.25ex\left\vert\kern-0.25ex\left\vert #1 
		\right\vert\kern-0.25ex\right\vert\kern-0.25ex\right\vert}}

\newtheorem{theorem}{Theorem}[section]

\newtheorem{claim}[theorem]{Claim}
\newtheorem{lemma}[theorem]{Lemma}
\newtheorem{definition}{Definition}
\newtheorem{prop}[theorem]{Proposition}
\newtheorem{assumption}[theorem]{Assumption}
\theoremstyle{remark}
\newtheorem{remark}[theorem]{Remark}

\newcommand{\id}{\mathrm{id}}

\numberwithin{equation}{section} % Number equations within sections (i.e. 1.1, 1.2, 2.1, 2.2 instead of 1, 2, 3, 4)
\numberwithin{figure}{section} % Number figures within sections (i.e. 1.1, 1.2, 2.1, 2.2 instead of 1, 2, 3, 4)
\numberwithin{table}{section} % Number tables within sections (i.e. 1.1, 1.2, 2.1, 2.2 instead of 1, 2, 3, 4)
\numberwithin{theorem}{section}

\numberwithin{definition}{section}

\newcommand{\horrule}[1]{\rule{\linewidth}{#1}} % Create horizontal rule command with 1 argument of height

\title{	
	\horrule{0.5pt} \\[0.4cm] % Thin top horizontal rule
	\large Singular SPDEs on Homogeneous Lie Groups
	\\ % The assignment title
	\horrule{0.8pt} \\[0.5cm] % Thick bottom horizontal rule
}

\author[1]{\normalsize Avi Mayorcas}
\author[2]{\normalsize Harprit Singh}
\affil[1]{Institute of Mathematics, Technische Universität Berlin, Str. des 17. Juni 136, 10587 Berlin, Germany. \newline Email: avimayorcas@gmail.com; ORCID iD: 0000-0003-4133-9740}
\affil[2]{Department of Mathematics, Imperial College London, South Kensington Campus, SW7 2AZ, UK.\newline Email: h.singh19@imperial.ac.uk; ORCID iD: 0000-0002-9991-8393}
\date{\normalsize \today}                     %% if you don't need date to appear
\begin{document}
	\maketitle % Print the title
\begin{abstract}
\noindent	The aim of this article is to extend the scope of the theory of regularity structures in order to deal with a large class of singular SPDEs of the form
	\begin{equation*}
		\partial_t u = \mfL  u+ F(u, \xi)\ ,
	\end{equation*}
	where the differential operator $\mfL$ fails to be elliptic. 
	This is achieved by interpreting the base space $\mathbb{R}^{d}$ as a non-trivial homogeneous Lie group $\mbG$ such that the differential operator $\partial_t -\mathfrak{L}$ becomes a translation invariant hypoelliptic operator on $\mbG$.  Prime examples 
	are the kinetic Fokker-Planck operator $\partial_t  -\Delta_v  - v\cdot \nabla_x$ and heat-type operators associated to sub-Laplacians.
	As an application of the developed framework, we solve a class of parabolic Anderson type equations 
	$$\partial_t u = \sum_{i} X^2_i u + u (\xi-c)$$
	on the compact quotient of an arbitrary Carnot group.
\end{abstract}

\noindent \textit{Keywords:} Regularity structures, homogeneous Lie groups, hypoelliptic operators, stochastic partial differential equations.

\noindent \textit{2020 MSC:} 60L30 (Primary); 60H17, 35H10, 35K70 (Secondary)
	
	\tableofcontents
	%
	%\newpage
	
\section{Introduction}\label{sec1}

The theory of regularity structures, \cite{hairer_14_RegStruct}, provides a framework for the study of subcritical stochastic partial differential equations (SPDEs) of the form
\begin{equation}\label{eq:gen_SPDE}
	\partial_t u-{\mfL} u=F(u, \xi)\ ,
\end{equation}
when the operator ${\mfL} $ is uniformly elliptic. This article extends the theory of regularity structures in order to solve equations where the differential operator stems from a large class of hypoelliptic operators.
This is achieved by building on the fundamental idea of Folland \cite{folland_75_subelliptic} to reinterpret the differential operator in question as a differential operator on a homogeneous Lie group and extending the theory of regularity structures to these non-commutative spaces. 

The treatment of singular SPDE via the theory of regularity structures can be divided into two parts, an analytic step and an algebraic and stochastic step.
\begin{enumerate}
	\item \hspace{0.2em} The first, analytic step, is to introduce the notions of a regularity structure, models and modelled distributions; generalised Taylor expansions of both functions and distributions. Given this set-up, sufficient analytic tools are then developed to allow for the study of abstract, fixed-point equations in the spaces of modelled distributions. Crucially, a reconstruction operator maps modelled distributions to genuine distributions (or functions) on the underlying domain. For the case of constant coefficient, hypoelliptic equations on compact quotients of Euclidean domains, this aspect of the theory was already worked out in full generality in \cite{hairer_14_RegStruct}.
	\item \hspace{0.2em} The second step, which was carried out in a case by case basis in \cite{hairer_14_RegStruct}, was then fully automated in the subsequent works \cite{chandra_hairer_16,bruned_hairer_zambotti_19_algebraic,bruned_chandra_chevyrev_hairer_21_renormalising}. 
	In \cite{bruned_hairer_zambotti_19_algebraic} the authors construct concrete (equation-dependent) regularity structures and models together with a large enough renormalisation group. 
	Then, in \cite{bruned_chandra_chevyrev_hairer_21_renormalising}, the question of how this renormalisation group acts on the SPDE in question was answered. Convergence of renormalised models for an extremely large class of noises,  known as the BPHZ theorem, was proved in \cite{chandra_hairer_16}.
\end{enumerate}
In much of the above work translation invariance on $\mbR^d$, of both the equation and the driving noise, plays a major role. In this paper we will instead work with equations that are translation invariant with respect to a homogeneous Lie group $\mbG$ (Euclidean spaces being special case).

{The bulk of this paper is dedicated to implementing the analytic aspects of the theory in the case when the underlying space is a general (non-Abelian) homogeneous Lie group. The second step is then carried out for a specific class of equations. While the theory developed herein remains, somewhat surprisingly, close to that of \cite{hairer_14_RegStruct}, there are a number of key conceptual and technical obstacles to overcome. The main cause of these deviations are the non-commutative structure of the base space and the fact that the notion of Taylor expansions, a crucial element of the theory, heavily depends on the underlying group structure. We discuss these issues in more detail in Section~\ref{sec:structure_and_contribution} along with some wider context in Section~\ref{sec:open_problems}.  For the readers convenience the article aims at being self-contained and thus arguments that adapt directly from \cite{hairer_14_RegStruct} are clearly marked as such.}

The central motivation for this extension to homogeneous Lie groups is to allow for the study of singular SPDE with linear part given by a  hypoelliptic operator, which may fail to be uniformly parabolic. Two motivating examples are the heat operator associated to the sub-Laplacian on the Heisenberg group and the kinetic Fokker-Planck operator on $\mbR \times \mbR^{2n}$,
\begin{equation}\label{eq:intro_example}
	u(t,x,y,z)\mapsto \partial_t u - (\Delta_x +\Delta_y + (x^2+y^2)\Delta_z) u  \text{ and }  u(t,x,v)\mapsto \partial_t u - \Delta_v u - v\cdot \nabla_x u\ .
\end{equation}
We will carry these two operators and their associated homogeneous Lie groups, throughout the paper as working examples of the theory. In the final sections we develop a full solution theory for Anderson type equations associated to sub-Laplacians on general stratified Lie groups (Carnot groups), of which the Heisenberg group is a well-studied example.

More generally, however, the analytic results of this paper apply to any differential operator, $\bar{\mfL}$ on $\mbR^{d-1}$, such that the following combination of results by Folland applies to $\mfL:=\partial_t- \bar{\mfL}$, with respect to some homogeneous Lie group structure on $\mbR^{d}$.
\begin{theorem}[{\cite[Thm.~2.1 \& Cor.~2.8]{folland_75_subelliptic}}]\label{th:folland_translation}
	Let $\mbG$ be a homogeneous Lie group of homogeneous dimension $|\mfs|$ and $\mfL$ be a left-translation invariant (with respect to $\mbG$), homogeneous differential operator of degree $\beta\in (0,|\mfs|)$ on $\mbG$, such that $\mfL$ and its adjoint $\mfL^\ast$ are both hypoelliptic. Then there exists a unique homogeneous distribution $K$, of order $\beta-|\mfs|$ such that for any distribution, $\varphi\in \mcD'(\mbG)$
	\begin{equation*}
		\mfL (\varphi \ast K) = (\mfL \varphi)\ast K = \varphi.
	\end{equation*}
	where the convolution is with respect to the given Lie structure.
\end{theorem}
Referring to Section~\ref{sec:analysis_background} for a more thorough discussion of homogeneous Lie groups and their properties, we recall here that a differential operator $\mfD$ on $\mathbb{R}^d$ is called \textit{hypoelliptic}, if for every open subset $\Omega\subset \mathbb{R}^d$ one has that,
\begin{equation*}
\mfD u\in C^\infty(\Omega) \Rightarrow u\in C^\infty (\Omega) \ .
\end{equation*}
The following celebrated result of H\"ormander (almost) entirely classifies the family of second order hypoelliptic operators.
\begin{theorem}[{\cite[Thm.~1.1]{hormander_67}}]\label{th:hormander}
Let $r\leq d$ and $\{X_i\}_{i=0}^r$ be a collection of first order differential operators (i.e. vector fields) on $\mbR^d$ and recursively define $\mcV_1 = \text{span}\{X_i\,:\, i=0,\ldots,r\}$, $\mcV_{n+1} = \mcV_n \cup \left\{[V,W]\,: V\in \mcV_1,\, W\in \mcV_n\right\}$. If a differential operator can be written in the form,
\begin{equation}\label{eq:vec_field_operator}
	\mfD = \sum_{i=1}^r X_i^2 + X_0 +c,
\end{equation}
for some $c\in \mbR$, and there exists an $N\geq 1$ such that $\dim(\mcV_N)=d$ at every point in $\Omega\subseteq \mbR^d$ then $\mfD$ is hypoelliptic on $\Omega$. 
\end{theorem}
\begin{remark}\label{rem:hormander_sharp}
By Frobenius's theorem, \cite{frobenius_77_uber}, if the condition of Theorem~\ref{th:hormander} fails in some open set then $\mfD$ is not hypoelliptic on that set. However, the statement is not truly sharp; for example the Grushin operator $\partial^{2}_{vv} - v \partial_x$ is hypoelliptic, while the conditions of Theorem~\ref{th:hormander} are violated on sets intersecting $\{v=0\}$. On the other hand this type of exception cannot occur if the sections $\mcV_i$ are continuous, since in this case it can be shown that the map \mbox{$x\mapsto \mathbb{1}_{\{\dim \mcV_i(x)=d\}}$} is upper semi-continuous. %
\end{remark}

While H\"ormander's theorem gives an almost sharp
characterisation of second order, hypoelliptic operators it does not say much about fine properties of the fundamental solution. For example, it gives no direct route to a refined regularity theory for hypoelliptic equations. This observation highlights the contribution of Folland's theorem (Theorem~\ref{th:folland_translation}); since translation invariance allows access to many additional tools in the Euclidean setting, such as harmonic analysis and singular integral methods. Viewing the class of translation invariant operators satisfying Folland's theorem as analogous to constant coefficient, elliptic operators on $\mbR^d$ a programme was successfully carried out in the works \cite{folland_stein_74_kohn,folland_stein_74_parametrices,folland_75_subelliptic, rothschild_stein_76}, establishing a full $L^p$ regularity theory for general, smooth coefficient, hypoelliptic operators. We refer to \cite[Ch.~3]{bramanti_14_invitation} for a concise introduction to this programme. 
\subsection{Related Literature}\label{subsec:related_works}
Non-translation invariant and non-uniformly elliptic SPDEs have been well-studied in the classical,  i.e.\ non-singular, regime and we do not attempt to present this large literature here. We refer to standard texts such as \cite{daprato_zabczyk_14,lototsky_rozovsky_17,dalang_koshnevisan_mueller_nualart_xiao_09,prevot_rockner_07} for a general overview and references to more specific works contained therein. However, in the more specialised setting of semi-linear SPDEs on homogeneous Lie groups, we mention the works \cite{ tindel_viens_99_regularity,tindel_viens_02_regularity,peszat_tindel_10_heat_wave} which treat both hypoelliptic and parabolic SPDEs on some classes of Lie groups and sub-Riemannian manifolds. A solution theory for conservative SPDEs based on the kinetic Fokker--Planck equation (see Section~\ref{subsec:kernel_examples})  in the It\^o case was developed in \cite{fedrizzi_flandoli_priola_vovelle_17}. Using the theory of paracontrolled calculus this was extended to the singular regime in \cite{hao_zhang_xicheng_zhu_zhu_21_kinetic}. The kinetic Fokker--Plank operator and its associated homogeneous Lie group fall into the analytic framework developed in this paper, however, we postpone a concrete application of this theory towards a kinetic Anderson type equation to a future work. Recently, in \cite{baudoin_ouyang_tindel_wang_22_hPAM_ito}, an Anderson type equation on the Heisenberg group with white in time and coloured in space noise was studied using It\^o calculus techniques, up the to the full sub-critical regime of the noise regularity. We treat a closely related problem in the final sections of this paper. A fuller discussion of the similarities and differences between the results of \cite{baudoin_ouyang_tindel_wang_22_hPAM_ito} and those of our approach is postponed to Remark~\ref{rem:other_paper}.

Singular SPDEs with non-translation invariant, but uniformly parabolic or elliptic linear part, have also been considered, especially using the recently developed pathwise techniques of \cite{hairer_14_RegStruct,gubinelli_imkeller_perkowski_15, otto_sauer_smith_weber_21_quasilinear, Sin23}. Some of these works are discussed in more detail in Section~\ref{sec:open_problems} below. Quasilinear SPDEs have been studied using regularity structures, rough path based methods and paracontrolled distribution theory, see \cite{gerencser_hairer_19_quasilinear,otto_weber_19_quasilinear, bailleul_mouzard_22_quasilinear, Ismael}. Recently, an approach inspired by the theory of regularity structures, but technically distinct, has been developed in \cite{otto_sauer_smith_weber_21_quasilinear,linares_otto_tempelmayr_21_structure,linares_otto_22_tree_free}. SPDEs on domains with boundaries have been treated using both regularity structures and paracontrolled distribution based methods, \cite{labbe_19_anderson_3d,gerencser_hairer_19_domains,chouk_vanZuijlen_21_asymptotics,gerencser_hairer_21_boundary_renorm}. Finally, a number of works have considered parabolic, singular SPDEs on Riemannian manifolds; a paracontrolled approach using the spectral decomposition associated to the Laplace--Beltrami operator has been developed in \cite{bailleul_bernicot_16_heat,bailleul_bernicot_19_high,mouzard_22_weyl}. An approach via regularity structures has been applied to the $2$d parabolic Anderson equation on a Riemannian manifold in \cite{dahlqvist_diel_driver_19_riemannian}, while \cite{bailleul_bruned_21_locallity, Sin23, SinHomogenize, WeijunHomogenize} develop some aspects required to treat non-translation invariant, uniformly parabolic equations. Finally, the work \cite{hairer_singh_22_manifolds} gives a comprehensive extension of regularity structures to singular SPDEs on Riemannian manifolds, with only the renormalisation of suitable stochastic objects left to a case by case treatment.
\subsection{A Motivating Class of Examples}
In this paper we restrict ourselves to linear operators satisfying the criteria of Theorem~\ref{th:folland_translation} and use the Anderson equation as a main motivating example, although we stress that our main analytic results apply in the full generality treated by \cite{hairer_14_RegStruct}. The parabolic Anderson model 
\begin{equation}\label{eq:PAM}
\partial_t u= \Delta u- u\xi,\quad u\tzero=u_0,
\end{equation}
on $\mbR_+\times \mbR^d$ describes the conditional, expected density of particles, where each particle moves according to an independent Brownian motion and branches at a rate proportional to the random environment $\xi$, see \cite{konig_16_PAM_book,carmona_molchanov_94_PAM_book}. Rigorously, this description is derived by discrete approximations and it is well known that in the case $\xi$ is a spatial white noise, when $d=2$ and $d=3$ one needs to recentre the potential in order to obtain a non-trivial limit as the discretisation is removed, \cite{hairer_14_RegStruct,hairer_labbe_15_simple,hairer_labbe_18_multiplicative,gubinelli_imkeller_perkowski_15}. We point out that if the environment is also allowed to depend on time and is for example, white in time, then martingale methods can be used instead to develop a probabilistic solution theory, see \cite{walsh_86_introduction,dalang_99_extending,dalang_01_extending_corrections,chen_15_precise}.

If one replaces the Brownian motions with a general diffusion,
\begin{equation*}
\dd x _t =  \sqrt{2}\sum_{i=1}^r X_i(x_t)\dd W^i_t,
\end{equation*}
where the vector fields satisfy H\"ormander's rank condition (Theorem~\ref{th:hormander}), then, formally the conditional expected density of particles is described by the hypoelliptic Anderson equation,
\begin{equation}\label{eq:hypoelliptic_anderson}
\partial_t u -\bar{\mfL}u := \partial_t u- \sum_{i=1}^r X^2_i u  = u \xi\ .
\end{equation}
When realisations of the environment are sufficiently singular then a pathwise solution theory for \eqref{eq:hypoelliptic_anderson} is expected to require renormalisation. Since the vector fields $\{X_i\}_{i=1}^r$ satisfy H\"ormander's condition, the operator $\bar{\mfL}$ is smoothing and therefore in principle, an extension of the theory of regularity structures should be applicable to find a renormalised, pathwise, solution theory for \eqref{eq:hypoelliptic_anderson}. In Section~\ref{sec:worked_application} we apply the analytic tools developed in the paper to demonstrate such an extension to the case where $\bar{\mfL}$ is the sub-Laplacian on a compact quotient of stratified Lie group (Carnot group).

We show a result analogous to those of \cite{hairer_14_RegStruct,hairer_labbe_18_multiplicative}, finding a notion of renormalised solution to \eqref{eq:hypoelliptic_anderson}, when $\xi$ is a coloured periodic noise on a stratified Lie group. More precisely, we show that when $\xi$ is replaced with a mollified, recentred noise $\xi_{\eps}-c_\eps$, for (specific) diverging constants $\{c_\eps\}_{\eps\in (0,1)}$, solutions $u_\eps$ to \eqref{eq:hypoelliptic_anderson} converge in probability to a unique limit independent of the specific choice of mollification-scheme. We stress that this result does not cover the full subcritical regime of \eqref{eq:hypoelliptic_anderson}. The treatment of this full regime would require an analogue of the BPHZ theorem on homogeneous Lie groups, see \cite{chandra_hairer_16,hairer_steele_23_spectral}.

A notable example of a stratified Lie group is the Heisenberg group, $\mbH^n \cong \mbR^{2n}\times \mbR$, see Section~\ref{subsec:heisenberg_group} for a description. In this case the collection of left-translation invariant vector fields, which generate the associated Lie algebra are,
\begin{equation*}
A_i(x,y,z) = \partial_{x_i} + y_i \partial_z,\quad B_i(x,y,z) = \partial_{y_i} - x_i \partial_z, \quad \text{for }\,i=1,\ldots,n.
\end{equation*}
It is readily checked that the collections $\{A_i,B_i\}_{i=1}^n$ are left-translation invariant with respect to the group action described in Section~\ref{subsec:heisenberg_group} and that one has $C(x,y,z):= [A_i,B_i] = -2\partial_z$ for all $i=1,\ldots,n$. The associated sub-Laplacian is the linear differential operator,
\begin{equation*}
\bar{\mfL}u = \sum_{i=1}^n (A_i^2 +B_i^2) u,
\end{equation*}
naturally extended to a heat type operator as in \eqref{eq:intro_example} and \eqref{eq:hypoelliptic_anderson}, c.f. Section~\ref{subsec:heisenberg_heat}
for further discussion.
A phenomenon of interest for Anderson equations is that of localisation, the concentration of the solution at large times, taking arbitrarily large values on islands of arbitrarily small size, \cite{carmona_molchanov_94_PAM_book,konig_16_PAM_book}. Since one expects the geometry of the underlying domain to have effect on the emergence of this phenomenon, as also noted in \cite{baudoin_ouyang_tindel_wang_22_hPAM_ito} and since pathwise approaches to the parabolic Anderson model have proved fruitful in studying finer properties of solutions, see e.g. \cite{allez_chouk_15,hairer_labbe_15_simple,hairer_labbe_18_multiplicative,gubinelli_ugurcan_zachhuber_18,labbe_19_anderson_3d,bailleul_dang_mouzard_22_analysis_anderson}, it is our hope that the tools developed herein may prove useful in analysing similar equations on more complex domains. The solution theory exposited in Section~\ref{sec:worked_application} applies to precisely this example on a compact quotient of the Heisenberg group.
\subsection{Structure and Contributions of the Paper}\label{sec:structure_and_contribution}
While it is possible to develop a bespoke solution theory for specific hypoelliptic SPDEs, this rapidly becomes technically challenging and computationally heavy, c.f. \cite{hao_zhang_xicheng_zhu_zhu_21_kinetic}. Furthermore, for more complicated equations this bespoke approach quickly becomes impossible. Thus,  in our view, the overarching contribution of this article is to show that by taking a more abstract and general point of view one can treat a large class of such equations at once in a framework (surprisingly) close to that of \cite{hairer_14_RegStruct}. In addition, this work highlights the dependence on the algebraic structure of the underlying space and thus contributes to the understanding of solution theories in more general settings, see Section~\ref{sec:open_problems}.

We begin, in Section~\ref{sec:analysis_background}, with a recap of analysis on homogeneous Lie groups that is sufficient for our purpose. The most important result in this section is Theorem~\ref{th:taylor} which provides a precise intrinsic version of Taylor's theorem on homogeneous Lie groups and is a sharper version of existing results, see Remark~\ref{rem:taytay}.

Section~\ref{sec:r_structures} contains the bulk of the paper and establishes an extension of the analytic aspects of regularity structures to the setting of homogeneous Lie groups. In  Section~\ref{sec:polynomial_r_struct} we introduce the appropriate notion of abstract polynomials on homogeneous Lie groups which are the backbone of much of our analysis. 
In Section~\ref{sec:modelled_distributions} we discuss modelled distributions. In order to prove reconstruction we extend ideas from \cite{friz_hairer_20_introduction}, constructing a suitable replacement for the test functions, $\rho $ and $ \varphi$ employed therein. This construction avoids the need for a general theory of wavelets or Fourier analysis on homogeneous Lie groups. In a similar vein, the set-up of singular modelled distributions in Section~\ref{subsec:singular_modelled}, while conceptually similar to that of \cite[Sec.~6]{hairer_14_RegStruct} requires a number technical novelties. In particular, the singular hyperplane of \cite{hairer_14_RegStruct} is replaced with a homogeneous Lie-subgroup and thus statements and proofs are adapted to the underlying (non-commutative) structure. 
In Section~\ref{subsec:schauder} we provide a surprisingly compact proof of the Schauder estimates, building on Theorem~\ref{th:taylor} in combination with the setup of Sections~\ref{sec:polynomial_r_struct} \& \ref{sec:modelled_distributions} in order to reduce to estimates already obtained in \cite[Sec.~5]{hairer_14_RegStruct}.
The final three sections, Section~\ref{subsec:local_ops}, \ref{subsec:symmetries} \& \ref{subsec:bounds_on_models}, carry over almost mutatis mutandis from \cite{hairer_14_RegStruct} and are included for the reader's convenience.

In Section~\ref{sec:evolution} we demonstrate the application of this general theory to semi-linear evolution equations. We again take an abstract point of view and start with a general homogeneous Lie group and only postulate a ``time'' direction as a one dimensional subspace of its Lie algebra. 
Then, we provide an example fixed point theorem for multiplicative stochastic heat equations in this setting. We note that there is no difficulty in extending the general fixed point theorem of \cite[Thm.~7.8]{hairer_14_RegStruct} to our setting, we only specialise for the sake of brevity. In Section~\ref{subsec:kernel_examples} we present examples of singular kernels on homogeneous Lie groups for which our theory applies, such as the Green's function of heat operators on the Heisenberg group and Kolmogorov type operators associated to inertial systems.

In Section~\ref{sec:PAM_structure} we explicitly describe a regularity structure for treating linear, Anderson-type equations on homogeneous Lie groups. For this specific class of equations it is notable that one may reuse the algebraic structures developed in \cite{bruned_hairer_zambotti_19_algebraic}. However, an algebraic question left open by this work is the extension of the results in \cite{bruned_hairer_zambotti_19_algebraic} to homogeneous Lie groups in full generality.

{Finally, in Section~\ref{sec:worked_application}, we employ the theory developed in the rest of the paper to demonstrate a solution theory for Anderson-type equations on a compact quotient $S$ of a stratified Lie group $\mathbb{G}$. Since we do not limit ourselves to quotients by normal subgroups, the space $S$ may fail be a group and therefore it is not possible to directly regularise the noise in the usual manner. This last obstacle is circumvented by making use of the fact that $S$ is a $\mathbb{G}$-homogeneous space allowing us to define an intrinsic regularisation of the noise.}
\subsection{Open Problems and Wider Context}\label{sec:open_problems}
As discussed above, translation invariant (and for example) parabolic operators on Euclidean domains serve as a starting point from which
non-translation invariant parabolic problems on Euclidean domains as well as parabolic problems on Riemannian manifolds can be studied.
A somewhat parallel progression holds starting from translation invariant operators on homogeneous Lie groups,  moving to general Hörmander operators as well as heat type equations on sub-Riemannian manifolds, see \cite{bramanti_14_invitation}. This parallel progression in PDE analysis leads one to ask, how far the theory of regularity structures can be extended in these directions. The table below gives a schematic presentation of the progress so far, presents some open problems and places this work within the context of the study of subcritical parabolic-type SPDEs. The two rows describe the two parallel progressions outlined above in the context of regularity structures.
\vspace{0.5em}
\small \begin{center}
\begin{tabularx}{\textwidth} { 
		| >{\raggedright\arraybackslash}X 
		| >{\raggedright\arraybackslash}X 
		| >{\raggedright\arraybackslash}X | }
	\hline
	\textbf{Translation invariant operators on $\mathbb{R}^d$}& \textbf{Non-translation invariant operators on $\mathbb{R}^d$} & \textbf{On Riemannian manifolds}\\
	\hline
	The works  \cite{hairer_14_RegStruct,bruned_hairer_zambotti_19_algebraic,bruned_chandra_chevyrev_hairer_21_renormalising,chandra_hairer_16} present a general and complete picture. & Analytic step in \cite{hairer_14_RegStruct}; aspects of renormalisation addressed in \cite{bailleul_bruned_21_locallity, Sin23}.  & In \cite{dahlqvist_diel_driver_19_riemannian} 2d-PAM is treated. A general account in \cite{hairer_singh_22_manifolds}.\\ 
	\hline  \hline
	\textbf{Translation invariant operators on homogeneous Lie groups} & \textbf{General H\"ormander operators} & \textbf{On sub-Riemannian manifolds} \\ 
	\hline
	Analytic theory covered in this work. Renormalisation and convergence by hand. & Open problem A & Open problem B \\ 
	\hline
\end{tabularx}
\end{center}
\vspace{1em}
\normalsize

Each problem in the second row is closely related to its equivalent problem in the first row, we discuss the posed open problems in a little more detail.
\begin{itemize}
\item \hspace{0.2em} Open Problem A is expected to be more involved than its counterpart in the Euclidean setting; going from translation invariant operators to non-translation invariant operators by local approximations. In the case of general H\"ormander operators, the underlying Lie group structure would in general vary from point to point. An interesting application would be the study of general hypoelliptic Anderson models \eqref{eq:hypoelliptic_anderson} where the underling particles perform a general diffusion with generator of the form \eqref{eq:vec_field_operator}, c.f. Theorem~\ref{th:hormander}.
\item \hspace{0.2em} Open Problem B is motivated by the fact that analogous to the tangent space being an appropriate local approximations of a Riemannian manifold, the non-holonomic tangent space is an appropriate local approximations of a sub-Riemannian manifold and each fibre of the non-holonomic tangent space is a stratified Lie group, c.f.  \cite{agrachev_barilari_boscain_20_subRiemannian}. Note that in view of Remark \ref{rem:hormander_sharp} the difficulties in Problem B  are expected to be somewhat complementary to those of Problem A. 
\end{itemize}
\paragraph{Notation:} Given $\alpha \in \mbR$ we let $[\alpha]:= \max\{r\in \mbZ\,:\, r\leq \alpha\}$ denote the integer part of $\alpha$. We often write $\lesssim$ to mean that an inequality holds up to multiplication by a constant which may change from line to line but is uniform over any stated quantities. In general we reserve the notation $D$ for the usual derivative on Euclidean space, and $X,\,Y$ for elements of a Lie algebra thought of as vector fields on the associated Lie group, see Section~\ref{sec:analysis_background}. We will use $|\,\cdot\,|$ to denote the size of various quantities; elements of a Lie group, the Haar measure of subsets of the group, the homogeneity of members of the structure space in a regularity structure, traces of linear maps and the absolute value function on $\mbR$ etc. The meaning will usually be clear from the context but in cases where it is not we will make sure to specify in the text. For integrals we will use the standard notation, $\dd x$, for integration against the Haar measure on a given homogeneous, Lie group, see Proposition~\ref{prop:polynomial_and_haar}. Throughout most of the article we take an intrinsic point of view and do not equip the homogeneous Lie group with an explicit chart. 
\paragraph{Acknowledgements:}
The authors wish to thank Martin Hairer, Rhys Steele, Ilya Chevyrev and Ajay Chandra for helpful and insightful conversations during the preparation of this manuscript.

AM wishes to thank the INI and DPMMS for their support and hospitality which was supported by Simons Foundation (award ID 316017) and by Engineering and Physical Sciences Research Council (EPSRC) grant number EP/R014604/1 as well as DFG Research Unit FOR2402 for ongoing support. 

HS acknowledges funding by the Imperial College London President's PhD Scholarship. He wishes to thank Josef Teichmann and M\'at\'e Gerencs\'er for their hospitality during his visit to ETH Z\"urich and TU Wien.
\section{Analysis on Homogeneous Lie Groups}\label{sec:analysis_background}
We collect some basic facts on homogeneous Lie groups and smooth function on them. Let $\mathfrak{g}$ be a Lie algebra and $\mbG$ be the unique, up to Lie algebra isomorphism, corresponding simply connected Lie group. 
We write $[\cdot, \cdot ]$ for its Lie bracket and inductively set, $\mathfrak{g}_{(1)}=\frakg$ and $\frakg_{(n)}=[\frakg_{(n-1)}, \frakg]$. Recalling the exponential map $\exp:\mfg \rightarrow \mbG$, we have the Baker--Campbell--Hausdorff formula 
\begin{equation}\label{eq:BCH}
\exp(X)\exp(Y)= \exp(\opH (X,Y))\ ,
\end{equation}
where $\opH$ is given by $\opH (X,Y)=X+Y+\frac{1}{2}[X,Y]+...$ with the remaining terms consisting of higher order iterated commutators of $X$ and $Y$; crucially $\opH$ is universal, i.e.\ it does not depend on the underlying Lie algebra.
\begin{prop}[{\cite[Prop.~1.2]{folland_stein_82_hardy}}]\label{prop:polynomial_and_haar}
Assume that $\frakg$ is nilpotent, i.e. $\frakg_{(n)}={0}$ for some $n\in \NN$. Then, the exponential map is a diffeomorphism and
\begin{itemize}
	\item through this identification of $\frakg$ with $\mbG$, the map
	\begin{equation*}
		\mbG\times \mbG \ni (x,y) \mapsto xy \in \mbG
	\end{equation*}
	becomes a polynomial map (between vector spaces).
	\item the pull-back to $\mbG$ of the Lebesgue measure on $\mfg$ is a bi-invariant Haar measure on $\mbG$.\footnote{Recall that a bi-invariant Haar measure on a locally compact Hausdorff group is any Borel measure which is invariant under push-forward by both left and right translates of the underlying group. See the proof of \cite[Prop.~1.2 (c)]{folland_stein_82_hardy} for further details.}
\end{itemize}
\end{prop}
\begin{definition}
A dilation on $\frakg$ is a group of algebra automorphisms $\{\dil_r\}_{r>0}$ of the form \mbox{$\dil_r X= \exp (\log r \cdot \fraks) X $} with $\fraks : \frakg\to \frakg$ being a diagonalisable linear operator with $1$ as its smallest eigenvalue. 
\end{definition}
\begin{remark}
The requirement that the smallest eigenvalue of $\mfs$ be $1$ is purely cosmetic. Otherwise, denoting the smallest eigenvalue by $\fraks_1$, one can work with the new operator $\tilde{\fraks}=\frac{1}{\fraks_1}\fraks$.
\end{remark}
\begin{remark}\label{rem:eigenspaces}
For $a\in \RR$, denote by $W_a\subset \frakg$ the eigenspace of $\fraks$ with eigenvalue $a$. Thus for $X\in W_a$, $Y\in W_b$ one has the condition
\begin{equation*}
	\dil_r [X,Y] =  [\dil_r X, \dil_r Y] = r^{a+b} [X,Y]
\end{equation*}
which in particular implies $[W_a, W_b]\subset W_{a+b}$ and since $W_a=\{0\}$ for $a<1$ 
\begin{equation*}
	\frakg_{(j)}\subset \bigoplus_{a \geq j} W_a\ .
\end{equation*}
In particular, if $\mfg$ admits a family of dilations, it is nilpotent. The converse is not necessarily true. 
\end{remark}
\begin{definition}
A homogeneous Lie group $\mbG$ is a simply connected Lie group where its Lie algebra $\frakg$ is nilpotent and endowed with a family of dilations $\{\dil_r\}_{r>0}$. For $r>0$, we define the group automorphism $$x\mapsto r\cdot x:= \exp \circ \dil_r\circ \exp^{-1} x \ .$$ 

A homogeneous norm on $\mbG$ is a continuous function $|\cdot|: \mbG\to [0,\infty)$ satisfying the following properties for all $x\in \mbG$, $r\geq 0$
\begin{enumerate}
	\item $|x|=0$ if and only if $x=e$, the neutral element, 
	\item $|x|=|x^{-1}| \ ,$
	\item $|r\cdot x|= r|x| \ . $
\end{enumerate}
\end{definition}
The homogeneous norm naturally induces a topology generated by the open sets and in turn a Borel $\sigma$-algebra. From now on, we will always assume that $\mbG$ is equipped with this topology and $\sigma$-algebra. Furthermore, given a homogeneous group $\mbG$ we denote by $\{X_{j}\}_{j=1}^d\in \mathfrak{g}$ a basis of eigenvectors of $\fraks$ with eigenvalues $1=\fraks_1\leq \fraks_2 \leq ... \leq \fraks_d$ and such that 
\begin{equation}\label{eq:eigenvectors}
\fraks X_j = \fraks_j X_j \ .
\end{equation}
Given a measurable subset $E\subset \mbG$ we write $|E|$ for its Haar measure which we assume to be normalized such that the set $B_1=:\{x\in \mbG\ : \  | x| \leq 1\}$ has measure $1$, c.f.\ Proposition~\ref{prop:polynomial_and_haar}. In integrals we use the standard notation $\dd x$. 
We define $|\fraks|:=\text{trace} (\fraks)$ as the homogeneous dimension of $\mbG$, since  for any measurable subset $E\subset \mbG$ and $r> 0$, one has
\begin{equation*}
|r\cdot E|= r^{|\fraks|} |E| \ . 
\end{equation*}
We also define balls of radius $r> 0$,
\begin{equation*}\label{eq:ball_def}
B_r(x):= \{y\in \mbG\ : \  | x^{-1}y| < r\} \ .
\end{equation*}
The topology induced by these balls agrees with the topology of $\mbG$ as a Lie group, see \cite[Sec.~3.1.6]{fischer_ruzhansky_18_quantisation}. Note that due to the non-commutativity of $\mbG$, in general $|x^{-1}y|\neq |yx^{-1}|$. We consistently work with the following choice of a semi-metric on the group, 
$$\dd_{\mbG}(x,y):= |x^{-1}y|=|y^{-1}x|\ .$$
For $\mfK\subset \mbG$ we write $\bar{\mfK}:= \{z \in \mbG\,:\, \dd_{\mbG}(z,\mfK) := \inf_{y\in \mfK} |y^{-1}z| \leq 1\}$ for the $1$-fattening of $\mfK$.
\begin{remark}\label{rem:triangle_inequality}
A function $f:\mbG\rightarrow \mbR$ is called homogeneous of degree $\lambda\in \mbR$, if $f(r\cdot x)=r^\lambda f(x)$ for all $x\in \mbG$. One can show, c.f. \cite[Prop.~ 1.5 \& Prop.~ 1.6]{folland_stein_82_hardy}, that for any homogeneous norm on $\mbG$ there exists $\mu>0$ such that 
\begin{itemize}
	\item $|xy|\leq \mu(|x|+|y|)$ for any $x,y \in \mbG$
	\item $\left| |xy|-|x|\right| \leq \mu |y|$ for any $x,y \in \mbG$ such that $|y|\leq \frac{1}{2} |x|$ \ .
\end{itemize}
Furthermore all homogeneous norms are mutually equivalent and we may always choose a homogeneous norm that is smooth away from $e\in \mbG$, \cite[Sec.~3.1.6]{fischer_ruzhansky_18_quantisation}.
\end{remark}
\subsection{Derivatives and Polynomials}
We identify $\frakg$ with the left invariant vector fields $\frakg_L$ on $\mbG$ and write $\frakg_R$ for the right invariant vector fields. We write $X_i$ for the the basis elements as in \eqref{eq:eigenvectors} seen as elements of $\frakg_L$ and $Y_i$ for the basis of $\frakg_R$ satisfying $Y_i |_e= X_i |_e$ . Thus we can write 
$$X_jf(y)= \partial_t f(y\exp(tX_j))|_{t=0}\quad \text{ and } \quad  Y_jf(y)= \partial_t f(\exp(tY_j)y)|_{t=0} \ $$
for any smooth function $f\in C^\infty(\mbG)$.

A map $P: \mbG \to \mathbb{R}$ is called a polynomial if $P\circ \exp : \mfg \rightarrow \mbR$ is a polynomial on $\mfg$.\footnote{Recall that the space of polynomial functions on $\frakg$ is canonically isomorphic to $\bigoplus_{n} (\frakg^\ast)^{\otimes_s n}$ where $\otimes_s$ denotes the symmetric tensor product.} Let $\zeta_i$ be the basis dual to the basis $X_i$ of $\frakg$. We set 
\begin{equation}\label{eq:coordinate_def}
\eta_j \coloneqq \zeta_j\circ \exp^{-1}:\mbG\to \mbR,
\end{equation}
Note that $\eta= (\eta_1,...,\eta_d)$ forms a global coordinate system \footnote{Occasionally we use the corresponding notation $\frac{\partial}{\partial\eta_i}$.} and furthermore any polynomial map on $\mbG$ can be written in terms of coefficients $a_I \in \mbR$ as 
$$P= \sum_I a_I \eta^I$$
with the sum running over a finite subset of $\NN^d$ and where for a multi-index $I=(i_1,...,i_d)\in \NN^d$ we write
$\eta^I= \eta_1^{i_1}\cdot...\cdot \eta_d^{i_d}$. 
Define $d(I)=\sum_j \fraks_j i_j$ and $|I| = \sum_j i_j$, we call $\max \{ d(I)\  : \ a_I \neq 0\}$ the homogeneous degree and $\max\{ |I| \  : \ a_I \neq 0\}$ the isotropic degree of $P$. For $a> 0$, we denote by $\mathcal{P}_a$ the space of polynomials of homogeneous degree strictly less than $a$ and define $\triangle= \{d(I)\in \mathbb{R} \ : \ I\in \mathbb{N}^d \}$. \footnote{We point out a possibly counter-intuitive quirk of our definition; for $k\in \triangle$  the set $\mcP_k$ does \textit{not} contain polynomials of degree $k$ but only those of degree less than $k$. 
}
We can rewrite the group law on $\mbG$ explicitly in terms of $\eta= (\eta_1,...,\eta_d)$.

\begin{prop}\label{prop:factorisation_formula}
For $j\in \{1,\ldots,d\}$and multi-indices $I,\,J$ s.t. $d(I)+d(J) = \mfs_j$, there exist constants $C^{I,J}_j\in \mbR$ such that the following formula holds,
$$\eta_j (xy)=\eta_j (x)+\eta_j (y) + \sum_{I,J\neq 0,\ d(I)+d(J)=\fraks_j} C_j^{I,J}\eta^I(x)\eta^J(y)\ ,$$
for all $x,\,y \in \mbG$, where $\eta_j$ is given by the expression \eqref{eq:coordinate_def}.
\end{prop}
\begin{proof} \, Applying the Baker--Campbell--Hausdorff formula \eqref{eq:BCH} there exist constants $C^{I,J}_j$ which depend only on the choice of basis for $\mfg$ made in \eqref{eq:eigenvectors} but are independent of $x,\, \& y \in \mbG$ such that
$$\eta_j (xy)=\eta_j (x)+\eta_j (y) + \sum_{|I|+|J|\geq 2} C^{I,J}_j \eta^I(x)\eta^J(y)\ , \qquad \text{for all  } x,\, y\in \mbG.$$
By setting either $x=e$ or $y=e$, we find that $C^{I,J}_j=0$ if $I=0$ or $J=0$.
Since furthermore,
$\eta_j((rx)(ry))= r^{\fraks_j}\eta_j(xy)$ the claim follows.
\end{proof}
\begin{remark}
Proposition~\ref{prop:factorisation_formula} implies that for $\fraks_j<2$ one has $\eta_j (xy)=\eta_j (x)+\eta_j (y)$ while for $\fraks_j=2$ one has $\eta_j (xy)=\eta_j (x)+\eta_j (y) + \sum_{\fraks_k=\fraks_l=1} C_j^{k,l}\eta_k(x)\eta_l(y) $ .
\end{remark}
\begin{remark}
It is also noteworthy to realise that Proposition~\ref{prop:factorisation_formula} implies that $\mathcal{P}_a$ is invariant under right and left-translations. (This is not true if one replaces homogeneous degree by isotropic degree, except if $\mbG$ is abelian or $a=0$).
\end{remark}
\begin{remark}\label{rem:factorising_constants}
If follows from Proposition~\ref{prop:factorisation_formula} that one can write
$$\eta^K(xy)=\eta^K (x)+\eta^K (y) + \sum_{I,J\neq 0,\ d(I)+d(J)=d(K)} C_K^{I,J}\eta^I(x)\eta^J(y)\ ,$$
where the constants $C_K^{I,J}$ can be written in terms of the constants $C_j^{I,J}$.
\end{remark}
\begin{lemma}
For $i,\,j \in \{1,\ldots,d\}$,
\begin{equation}\label{eq:derivative_of_polynomial}
	X_i \eta_j = \delta_{i,j} + \sum_{I\neq 0,\  d(I)= \fraks_j-\fraks_i} C_j^{I, e_i} \eta^I \ ,
\end{equation}
where $e_i$ denotes the multi-index $(0,\ldots,0,1,0,\ldots,0)$ with the $1$ being in the $i$-th slot.
\end{lemma}
\begin{proof}\,
This follows directly from Proposition~\ref{prop:factorisation_formula}, applying $X_i$ to the function
$$y\mapsto \eta_j(xy) \ , $$ 
evaluating at $y=0$ and using the fact that $X_i|_0= \frac{\partial}{\partial\eta_i}|_0$.
\end{proof}
\begin{prop}[{\cite[Prop.~1.26]{folland_stein_82_hardy}}]
One has $X_j= \sum P_{j,k} \left(\frac{\partial}{\partial \eta_k}\right)$ where
$$P_{j,k}=\left\{
\begin{array}{ll}
	1  & \mbox{if } k=j \\
	0 & \mbox{ if  } \fraks_k\leq \fraks_j, k\neq j \ 
\end{array}
\right. $$
and $P_{j,k}$ is a homogeneous polynomial of degree $\fraks_k-\fraks_j$ if $\fraks_k>\fraks_j$. The analogous statement holds for the vector fields $Y_j$,
\end{prop}
For a multi-index $I=(i_1,...,i_d)\in \mbN^d$ we introduce the notation $X^I = X_1^{i_1}... X_d^{i_d}$. Note that the order of the composition matters since $\mfg$ is not in general Abelian. It is a well known fact that any left invariant differential operator on $\mbG$ can (uniquely) be written as a linear combination of $\{X^I \}_{I \in \mathbb{N}^d}$. The next proposition follows as a direct consequence.
\begin{prop}[{\cite[Prop.~1.30]{folland_stein_82_hardy}}]\label{prop:uniqueness_taylor}
The following maps from $\mathcal{P}_a\to \mathbb{R}^{\dim \mathcal{P}_a}$ are linear isomorphisms.
\begin{enumerate}
	\item $\,\, P \mapsto \left\{ \left(\frac{\partial}{ \partial_\eta}\right)^I P(e)\right\}_{d(I)< a}$ ,
	\item $\,\, P \mapsto \left\{ X^I P(e)\right\}_{d(I)< a}$ ,
	\item $\,\, P \mapsto \left\{ Y^I P(e)\right\}_{d(I)< a}$ ,
\end{enumerate}
The same holds replacing $e\in \mbG$ with any other point $x\in \mbG$.
\end{prop}
\begin{definition}\label{def:taylor}
For a smooth function $f:\mbG\to \mathbb{R}$, a point $x\in \mbG$ and $a\in (0,\infty)$, we define the left Taylor polynomial of homogeneous degree (less than) $a$ of $f$ at $x$ to be the unique polynomial $\opP^a_x[f]\in \mathcal{P}_a$ such that $X^I \opP^a_x[f](e)= X^I f(x)$ for all $I$ such that $d(I)<a$. The right Taylor Polynomial can be defined similarly, replacing $X^I$ by $Y^I$.
\end{definition}
\begin{remark}\label{rem:derivative_of_taylor}
One observes that for $I\in \mathbb{N}^d$ and $a>d(I)$ one has that
$$ X^I \opP^a_x[f]= \opP^{a-d(I)}_x[X^If] $$
for all $f\in C^\infty(\mbG)$. Indeed, this follows from the fact that for any $d(J)< a-d(I)$ one has $ X^J (X^I \opP^a_x[f])(e)=  X^J(X^I f)(x)$, since $X^J X^I$ can be written as a linear combination of $\{X^K\}_{d(K)< a}$.
\end{remark}
\begin{theorem}[Taylor's Theorem]\label{th:taylor}
For each $a> 0$ and every $f\in C^\infty(\mbG)$ it holds that
$$ f(xy)-\opP^a_x[f](y)= \sum_{|I|\leq [a]+1, d(I)\geq a} \int_{\mbG} X^{I}f(xz) Q^I(y, \dd z)  \ ,$$
where for each multi-index $I$ and $y\in \mbG$ the measure $Q^I(y, \,\cdot\,)$ is supported on $B_{ \beta^{[a]+1} |y|}(e)$ for some $\beta>0$ depending only on $\mbG$ and satisfies $\int_{\mbG} |Q^I(y, \dd z)| \lesssim |y|^{d(I)}$.
\end{theorem}
\begin{remark}\label{rem:taytay} Note that there is a more common version of Taylor's Theorem on homogeneous Lie groups, c.f. \cite[Thm.~1.37]{folland_stein_82_hardy}, which states that the remainder satisfies the estimate
\begin{equation}\label{eq:common_form_taylor}
	|f(xy)-\opP^a_x[f](y)|\lesssim_a \sum_{|I|\leq [a]+1, d(I)\geq a} |y|^{d(I)} \sup_{|z|\leq \beta^{[a]+1} |y|} |X^{I}f(xz) | \ 
\end{equation}
and which follows as a trivial consequence of Theorem~\ref{th:taylor}. For us the explicit formula of the reminder is needed in order establish Schauder estimates in Section~\ref{subsec:schauder}.
\end{remark}

\begin{remark}
The analogous claim for the right Taylor polynomial, $(\opP_R)^a_x[f]$, holds. In this case the analogue of \eqref{eq:common_form_taylor} reads 
\begin{equation}\label{eq:right_taylor}
	|f(yx)-(\opP_R)^a_x[f](y)|\lesssim_a \sum_{|I|\leq [a]+1, d(I)\geq a} |y|^{d(I)} \sup_{|z|\leq \beta^{[a]+1} |y|} |Y^{I}f(zx) | \ .
\end{equation}
\end{remark}
\begin{remark}\label{rem:rescaled_taylor}
\, 	If we define $\tilde{\opP}^a_x[f](y):= \opP^a_x[f](x^{-1}y)$, then it follows that 
\begin{equation*}
	f(y) - \tilde{\opP}^a_x[f](y)= \sum_{|I|\leq [a]+1, d(I)\geq a} \int_{\mbG} X^{I}f(xz) Q^I(x^{-1}y, dz) 
\end{equation*}
and in particular that
\begin{equation*}
	|f(y) - \tilde{\opP}^a_x[f](y)| \lesssim_{a}\sum_{\substack{|I|\leq [ a ]+1\\d(I)\geq a }} |x^{-1} y|^{d(I)} \sup_{|z| \leq C|x^{-1} y|} |X^I f(xz)|  \ .  
\end{equation*}
In the particular case when $f$ is compactly supported, one can rewrite this for the rescaled function $f^{\lambda}(x)= \frac{1}{\lambda^{|\fraks|}} f (\lambda\cdot x)$ as
\begin{align}\label{eq:rescaled_taylor}
	|f^\lambda (y) - \tilde{\opP}^a_x[f^\lambda](y)| & \lesssim_{a} \sum_{\substack{|I|\leq [ a ]+1\\d(I)\geq  a }}\lambda^{-d(I)-|\fraks | }|x^{-1} y|^{d(I)} \sup_{z\in \mbG} |X^I f(z)| \\	
	& \lesssim_{a,f } \sum_{\delta\geq 0}\lambda^{-(a+\delta)-|\fraks|}|x^{-1} y|^{a+\delta} \notag  \ ,
\end{align}
where the sum over $\delta$ runs over some finite subset of $[0,\infty)$.
\end{remark}
\begin{remark}
Given $f\in C^\infty(\mbG)$, if we define $F\in C^\infty(\mbG\times \mbG)$ by setting $F(y,z)=f(z^{-1}y)$, we find that 
$$\tilde{\opP}^a_x[F(\,\cdot\,,z)](y)=\tilde{\opP}^a_{z^{-1}x}[f](z^{-1}y) \ .$$
\end{remark}
\begin{remark}\label{rem:expanding_polynomials}
For any $p \in \mcP_{\gamma}$ one has
$$\tilde{\opP}_x^{\gamma} [p](y)= p(y) \ .$$
\end{remark}
\begin{remark}\label{rem:decompose_exp_map}
In the proof of Theorem~\ref{th:taylor}, given below, we use the following elementary observations
\begin{itemize}
	\item The map
	\begin{align*}
		\Phi: \mathbb{R}^d &\to \mbG, \\
		(t_1,...,t_d)&\mapsto \exp(t_1 X_1)\cdot ...\cdot \exp(t_d X_d)
	\end{align*}
	is a diffeomorphism. To see this note that it is clearly a local diffeomorphism since one has
	\begin{align*}
		D\Phi|_0: T\RR^d|_0\sim \RR^d &\to T\mbG|_0\sim \mathfrak{g}\\
		(t_1,...,t_d) &\mapsto \sum_i t_i X_i.
	\end{align*}
	Then since, 
	$$\Phi(r^{\fraks_1} t_1,...,r^{\fraks_d}t_d)=  r\cdot \Phi(t_1,...,t_d)$$ 
	it is seen to be a global diffeomorphism.
	\item Setting for $t\in \mathbb{R}^d$,
	$$|t|_\fraks:= \sum_{i=1}^d |t_i|^{1/\fraks_i} \ ,$$
	one finds that there exists some $\beta=\beta(\mbG)>0$ such that
	\begin{equation}\label{eq:Phi_Two_Sided_Scaling}
		\frac{1}{\beta} |t|_\fraks  \lesssim |\Phi(t)|\lesssim \beta |t|_\fraks
	\end{equation}
	uniformly over $t \in \mathbb{R}^d$.
	\item By repeated use of the commutator, for any $I, J\in \mathbb{N}^d$ and $i\in \NN$ there exist coefficients $\lambda_{i,J, I}$ such that
	\begin{equation}\label{eq:Vector_Field_Composition}
		X_i X^J= \sum_{I} \lambda_{i,J, I} X^{I}
	\end{equation}
	and one has $\lambda_{i,J, I}=0$ whenever $d(J)+\fraks_i\neq d(I)$ .
	\item If $a\in \triangle$, then $m_a:=\max\{|I| \ : \ d(I)\leq a\}= [a]$ where $[a]$ denotes the integer part of $a$. This follows from the fact that $d(I) \leq a$ implies $m_a\leq [a]$ and since $([a],0,...,0)$ is in the set over which the maximum is taken.
\end{itemize}
\end{remark}
\begin{proof}[Proof of Theorem~\ref{th:taylor}]\,
Define for $i\in \{1,...,d\}$ the following measure on $\RR^d$ with support on $B^\fraks_{|t|_\fraks}(0)$
$$\tilde{Q}^{e_i}(t, ds)=\prod_{j<i} \delta_{t_j} (ds_j) \cdot \mathbf{1}_{[0,t_i]} (ds_i) \prod_{j>i} \delta_{0} (ds_j)$$
and define
${Q}^{e_i}(y, \,\cdot\,)= \Phi_* \tilde{Q}^{e_i}(\Phi^{-1}(y), \,\cdot\,) $ to be the push-forward measure. First we show that,
\begin{equation} \label{eq:mean_value}
	f(xy)- f(x)=\sum_{i=1}^n \int_{\mbG} (X_i f)(xz) {Q}^{e_i}(y, dz) \ .
\end{equation}
Indeed one can write for $y=y_1y_2\ldots y_d$, where $y_i= \exp(t_i X_i)$ and $t=(t_1,...,t_d)= \Phi^{-1}(y)$
\begin{align*}
	f(xy)- f(x)&= \sum_{i=1}^d f(xy_1\ldots y_{i-1} y_{i})- f(xy_1...y_{i-1})\\
	&= \sum_{i=1}^d \int_0^{t_i} \partial_s f(xy_1 \ldots y_{i-1} \exp(s X_i))|_{s=s'} ds'\\
	&= \sum_{i=1}^d \int_0^{t_i} (X_i f)(xy_1\ldots y_{i-1} \exp(s X_i)) ds\\
	&= \sum_{i=1}^d \int_{\mathbb{R}^d} (X_i f)(x\Phi(s)) \tilde{Q}^{e_i}(t, ds) \\
	&= \sum_{i=1}^d \int_{\mbG} (X_i f)(xz) {Q}^{e_i}(y, dz) \ .
\end{align*}
Using the fact that $\Phi$ is a diffeomorphism one easily checks that ${Q}^{e_i}(y, dz)$ is supported on $B_{ \beta |y|}(0)$ where $\beta$ is as in \eqref{eq:Phi_Two_Sided_Scaling}, and satisfies $\int_\mbG |Q^{e_i}|(y, dz) \lesssim |y|^{d_i}$, thus we have proved the theorem in the special case $a\in (0,1]$, also known as mean-value theorem.
We turn to the proof in the general case. 
\begin{claim}
	Set $g(y)= f (xy)- \opP^a_x[f](y)$, we note that $X^Jg(e)=0$ for $d(J)< a$ while $X^Jg(y)=X^Jf(xy)$ for $d(J) \geq a$. 
	We shall prove that for any multi-index $J$ it holds that one can write
	$$X^Jg(y)= \sum_{|I|\leq [a]+1, d(I)\geq a} \int_{\mbG} X^{I}f(xz) Q^{I,J}(y, dz) \ ,$$
	where the measures $Q^{I,J}(x,\cdot)$ satisfy the following properties
	\begin{itemize}
		\item $Q^{I,J}(x,\cdot)$ is supported on $B_{ \beta^{m(J)} |y|}(e)\subset \mathbb{G}$ where $m(J)= \min\{ m\in \mathbb{N}:\  a-m<d(J) \}\ ,$
		\item $\int_{\mbG}|Q^{I,J}|(x,dy)\lesssim |x|^{d(I)-d(J)}$ .
	\end{itemize}
\end{claim}
\noindent We shall prove by induction on $m\in \{1,..., \lceil a \rceil \}$, that for multi-indices $J$ satisfying $a-m\leq d(J)< a$ the claim holds. 
\begin{itemize}
	\item The case $m=1$ follows from \eqref{eq:Vector_Field_Composition} and \eqref{eq:mean_value},
	\begin{align*}
		X^Jg(y)= X^J g(y)-X^J g(e) &= \sum_{i=1}^d \int_{\mbG} (X_i X^Jg)(z) {Q}^{e_i}(y, dz)  \\
		&= \sum_{i=1}^d \int_{\mbG} \sum_{I \ : \ d(I)
			= d(J)+\fraks_i} \lambda_{i,J, I} (X^{I} g)(z) {Q}^{e_i}(y, dz) \\
		&= \sum_{i=1}^d \int_{\mbG} \sum_{I \ : \ d(I)= d(J)+\fraks_i} \lambda_{i,J, I} (X^{I} f)(xz) {Q}^{e_i}(y, dz) \\
		&=\sum_{|I|\leq [a]+1, d(I)\geq a} \int_{\mbG} X^{I}f(xz) Q^{I,J}(y, dz)
	\end{align*}
	The properties of $Q^{I,J}(y, dz)$ follow from the corresponding properties of \linebreak ${Q}^{e_i}(y, dz)$, completing the case $m=1$.
	\item Suppose the claim holds for $m-1$. Then by the same argument as above 
	\begin{align*}
		X^Jg(y) =& \sum_{i=1}^d \sum_{  d(K)= d(J)+\fraks_i} \lambda_{i,J, K} \int_{\mbG}  (X^{K} g)(xz) {Q}^{e_i}(y, dz)\\
		= & \sum_{i, K  :  d(K)= d(J)+\fraks_i< a} \lambda_{i,J, K} \int_{\mbG}  (X^{K} g)(xz) {Q}^{e_i}(y, dz) \\
		&\quad +\sum_{i, K  :  d(K)= d(J)+\fraks_i\geq a} \lambda_{i,J, K} \int_{\mbG}  (X^{K} g)(xz) {Q}^{e_i}(y, dz). 
	\end{align*}
	We can apply the induction hypothesis to the terms in the first sum, since $d(K)= d(J)+\fraks_i\geq d(J)+1\geq a-(m-1)$  and find
	$$(X^{K} g)(xz)= \sum_{|I|\leq [a]+1, d(I)\geq a} \int_{\mbG} X^{I}f(x\tilde{z}) Q^{I,K}(z, d\tilde{z})$$
	and therefore 
	\begin{align*}
		X^Jg(y) =&\sum_{i, K  :  d(K)= d(J)+\fraks_i< a} \int_{\mbG}  (X^{K} g)(xz) {Q}^{e_i}(y, dz)\\
		&+\sum_{i, K  :  d(K)= d(J)+\fraks_i\geq a} \lambda_{i,J, K} \int_{\mbG}  (X^{K} g)(xz) {Q}^{e_i}(y, dz) \\
		=&\sum_{i, K  :  d(K)= d(J)+\fraks_i< a} \lambda_{i,J, K} \sum_{|I|\leq k+1, d(I)\geq a}\int_{\mbG}   \int_{\mbG} X^{I}f(x\tilde{z}) Q^{I,K}(z, d\tilde{z}) {Q}^{e_i}(y, dz) \\
		&+ \sum_{i, K  :  d(K)= d(J)+\fraks_i\geq a} \lambda_{i,J, K} \int_{\mbG}  (X^{K} g)(xz) {Q}^{e_i}(y, dz). \\ 
		=&\sum_{i, K  :  d(K)= d(J)+\fraks_i< a} \lambda_{i,J, K} \sum_{|I|\leq k+1, d(I)\geq a}   \int_{G} X^{I}f(x\tilde{z}) \big(Q^{I,K} *{Q}^{e_i}\big)(y, d\tilde{z}) \\
		&+\sum_{i, K  :  d(K)= d(J)+\fraks_i\geq a} \lambda_{i,J, K} \int_{\mbG}  (X^{K} g)(xz) {Q}^{e_i}(y, dz) \\
		=& \sum_{|I|\leq [a]+1, d(I)\geq a}   \int_{\mbG} X^{I}f(x\tilde{z}) Q^{I,J}(y, d\tilde{z})\,
	\end{align*}
	where $Q^{I,J}$ is defined such that the last line holds true. Concerning the measures $Q^{I,J}(y, \cdot)$, we note that
	\begin{enumerate}
		\item $\supp Q^{I,J}(y, \cdot)\subset B_{ \beta^{m(J)} |y|} $ since
		$$\supp \big( Q^{I,K} *{Q}^{e_i}(y, \cdot) \big)\subset B_{ \beta^{m(K)+1} |y|}  $$ 
		as $\supp  Q^{I,K}(y, \cdot)\subset B_{ \beta^{m(K)} |y|}(e)$ and $\supp {Q}^{e_i}(y, dz)\subset B_{ \beta |y|}(e)$.
		\item $\int_{\mbG} |Q^{I,J}(y, \cdot)| \lesssim |x|^{d(I)-d(J)}$ since
		$$\int |Q^{I,K} *{Q}^{e_i}|(y, dz)\leq \int |Q^{I,K}|(y, dz) \, \cdot \, \int |{Q}^{e_i}|(y, dz)\lesssim |y|^{d(I)-d(K)}\cdot |y|^{\fraks_i}\ .$$
	\end{enumerate}
\end{itemize}
This completes the proof.
\end{proof}
\subsection{Distributions and Convolution}\label{subsec:distributions_convolutions}
We define $\mcD(\mbG):= C^\infty_c(\mbG)$ to be the space of compactly supported smooth functions on $\mbG$ equipped with the canonical LF-topology. The space of distributions on $\mbG$ is given by the dual $\mcD'(\mbG)$ of $\mcD(\mbG)$ and for $\xi \in \mcD'(\mbG)$ and $\phi \in \mcD(\mbG)$ we either write $\langle \phi,\xi\rangle$ or $\xi(\phi)$ for the canonical pairing.
Given $r \in \mbR_{\geq 0}$ we introduce the following useful space of test functions
$$\mfB^{\lambda}_{ r}(x):= \left\{ \phi\in C^\infty_c (B_\lambda(x))  \ : \ |X^I \phi | \leq \frac{1}{\lambda^{|\mfs|+d(I)}} \  \forall I : d(I)\leq r \right\} \ . $$
We will often use the shorthand $ \mfB_{r}:= \mfB^{1}_{r}(e)$ .
Given $\phi \in \mcD(\mbG)$ we extensively use the following notations,
$$\phi^\lambda_x(z):=\frac{1}{\lambda^{|\mfs|}}\phi \left(\frac{1}{\lambda}\cdot  (x^{-1}z)\right).$$
Recalling that we use the notation $\dd x$ for the Haar measure on $\mbG$ we define the norms,
\begin{equation*}
\|f\|_{L^p} := \begin{cases}
	\left(\int_{\mbG} |f(x)|^p\dd x\right)^{1/p}, & \text{ for } 1\leq p<\infty,\\
	\esssup_{x \in \mbG} |f(x)|, & p=\infty.
\end{cases}
\end{equation*}
By left and right translation invariance of the Haar measure, for $f\in L^1(\mbG)$ it holds that
$$\int_{\mbG} f(yx)\dd y = \int_{\mbG} f(xy)\dd y = \int_{\mbG}f(y)\dd y = \int_{\mbG}f(y^{-1})\dd y.$$
See the discussion after \cite[Thm.~1.1.1]{fischer_ruzhansky_18_quantisation}. 
\begin{remark}\label{rem:test_functions}
From the scaling properties of $\dil$ it follows that 
$$ \mfB^{\lambda}_{r}(x) = \left\{ \phi^{\lambda}_x \ : \phi\in \mfB^{1}_{r}(e)  \right\}.$$
\end{remark}
One defines the convolution of two functions $\phi, \psi: \mbG\to \mathbb{R}$ as
\begin{equation}\label{eq:def_convolution}
\psi*\phi(x)= \int \psi(y) \phi(y^{-1}x) dy = \int \psi(xy^{-1}) \phi(y) dy\ .
\end{equation}
Note that in general $\psi* \phi(x)\neq \phi * \psi(x)$ but still many of the usual inequalities hold, in particular the Young inequality, c.f. \cite[Prop.~1.18]{folland_stein_82_hardy}. One can write
$$\psi*\phi(x)= \int \psi(y) \phi_y(x) dy\ . $$
We extend convolution to generalised function, i.e. elements of $\mcD'(\mbG)$, by duality in the usual manner.
From now on we will use the notation $\tilde{\phi}(z):= \phi(z^{-1})$, one notes that the following identities hold
\begin{align}
\langle f, g*\phi \rangle = \langle \tilde{g} * &f, \phi \rangle = \langle   f * \tilde{\phi}, g \rangle \label{eq:throwing_convolution} \\
\widetilde{f*g} &=  \tilde{g}*\tilde{f} \label{eq:convolution_reorder}\\
\psi^\lambda*\phi^\lambda = (\psi*\phi)^\lambda\quad &\text{and}\quad (\psi \ast \phi)_x = \psi_x \ast \phi. \notag
\end{align}
Recalling the notation $X^I = X_1^{i_1}... X_d^{i_d}$ for any multi-index $I=(i_1,...,i_d)\in \mbN^d$, we introduce the analogous notation
$Y^I = Y_d^{i_d}....Y_1^{i_1}$
for the right invariant basis vectors. Note the reversal in the composition. Since for any right invariant vector field $Y$ and left invariant vector field $X$ such that $X|_e = Y|_e$ it holds that $\langle X f,g\rangle = \langle f, Y g\rangle $ for $f,\,g \in C^\infty_c(\mbG)$, it follows from the definition of the convolution that
\begin{itemize}
\item $X^I(\psi*\phi) = \psi*( X^I \phi) $,
\item $Y^I(\psi*\phi) = (Y^I\psi)* \phi $,
\item $(X^I\psi)*\phi= \psi*Y^I \phi \ $.
\end{itemize}

\begin{definition}\label{def:Holder}
For $\alpha \in \mbR$ we define the space $\mcC^{\alpha}(\mbG)$ of $\alpha$-Holder continuous distributions to be those $f\in \mcD'(\mbG)$ such that for every compact set $\mfK\subset \mbG$,
\begin{itemize}
	\item If $\alpha \leq 0$
	\begin{equation*}\label{eq:Negative_Holder_Def}
		\|f\|_{\mcC^\alpha( \mfK )}=	\sup_{x\in \mfK} \sup_{\phi \in \mfB_{ \lceil{-\alpha \rceil }}}\sup_{\lambda\in (0,1)} \frac{|\langle f,\phi^\lambda_x\rangle|}{\lambda^\alpha} < +\infty \ .
	\end{equation*}
	\item If $\alpha >0$, there exists a continuous map $\mbG\to \mathcal{P}_\alpha,\ x\mapsto \tilde{P}_x$ such that
	\begin{equation}\label{eq:Positive_Holder_Def}
		\|f\|_{\mcC^\alpha(\mfK )}=	\sup_{x\in \mfK}\sup_{\phi \in \mfB_{0}}\sup_{\lambda\in (0,1)} \frac{ |\langle f- \tilde{P}_x,\phi_x^\lambda\rangle|}{\lambda^{\alpha}} + \sup_{x\in \mfK} |\tilde{P}_x|_{\mathcal{P}_\alpha} < +\infty \ ,
	\end{equation}	
	where $|\, \cdot\, |_{\mathcal{P}_\alpha}$ denotes any norm on the finite dimensional vector space $\mathcal{P}_\alpha$.
\end{itemize}
Furthermore, $C(\mbG)$ denotes the space of continuous functions (note that only the strict inclusion $C(\mbG)\subsetneq \mcC^0(\mbG)$ holds). 
\end{definition}
\begin{remark}
We note that as in \cite{hairer_14_RegStruct}, for $\alpha \in \triangle$ the $\mcC^\alpha(\mbG)$ spaces do not agree with the usual notion of continuously $\alpha$-differentiable functions. This definition, however, is more canonical in the context of regularity structures, c.f. Theorem~\ref{th:hoelder_characterisation}. 
\end{remark}
The following proposition gives some useful properties that also mirror those of Euclidean H\"older spaces.
\begin{prop}\label{prop:holder_inclusion_compact}
Given any real number $\alpha\in \mathbb{R}$ the following holds.
\begin{enumerate}
	\item \label{it:holder_inclusion} $\mcC^\alpha(\mbG)\subset \mcC^{\beta}(\mbG)$ for every $\beta<\alpha$.
	\item\label{it:derivative} For any multiindex $I$ the map $X^I:\  \mcC^\alpha(\mbG)\to \mcC^{\alpha-d(I)} (\mbG), \quad  f\mapsto X^I f$ is well defined.
	\item \label{it:holder_compact} If $\alpha>0$, then any distribution $f \in \mcC^{\alpha}(\mbG)$ agrees\footnote{As usual, we say a distribution agrees with a continuous functions, if it lies in the image of the canonical embedding $C(\mbG)\to \mcD '(\mbG)$.} with a continuous function. 
\end{enumerate}
\end{prop}
Note that this proposition in particular implies that $C^\infty(\mbG)= \cap_{\alpha>0} \mcC^\alpha(\mbG)$.
\begin{proof}\,
Points~\ref{it:holder_inclusion} and \ref{it:derivative} follow directly. For Point \ref{it:holder_compact}, denote by $\tilde{P}_x$ the polynomial in Definition~\ref{def:Holder}, we set $F=f- \tilde{P}_{\cdot}(\cdot)$, then for any $\psi\in \mfB_{0}$ such that $\int_{\mbG} \psi = c^{-1}>0$ and smooth compactly supported function $g$ we find
\begin{align*}
	\langle F,  g\rangle = c\lim_{\lambda\to 0}  \langle F, g* \psi^\lambda \rangle = c\lim_{\lambda\to 0} \left \langle F, \int_{\mbG}  g(y)  \psi_y^\lambda dy \right\rangle = c\lim_{\lambda\to 0}  \int_{\mbG}\langle F,     \psi_y^\lambda  \rangle g(y) dy =0 \ ,
\end{align*}
where in the last step we used that since $y\to \tilde{P}_y$ is continuous we have
\begin{equation*}
	\langle F,     \psi_y^\lambda \rangle = \langle f- \tilde{P}_y(\,\cdot\,) ,     \psi_y^\lambda \rangle+\int \big(\tilde{P}_y(z) - \tilde{P}_z (z)\big)    \psi_y^\lambda(z) \, dz \to 0\ ,\end{equation*}
as $\lambda\to 0$.
Thus we find that the distribution $f$ agrees with the continuous function $x\mapsto \tilde{P}_x(x)$, concluding the proof.
\end{proof}
\begin{remark}\label{rem:taylor_extension0}
Note that by Proposition~\ref{prop:holder_inclusion_compact} for $\alpha>0$ and any $f\in \mcC^\alpha(\mbG)$ we know that $X^If$ agrees with a continuous function whenever $d(I)<\alpha$. Thus, in particular Definition~\ref{def:taylor} of Taylor polynomials $\opP^\alpha_x[f]$ extends to $\mcC^\alpha(\mbG)\supset C^\infty(\mbG)$.
\end{remark}

\subsection{Discrete Subgroups}\label{sec:discrete_subgroups}
Recall that, given a discrete subgroup $\mfG\subset \mbG$ acting on $\mbG$ (say) on the left, then the quotient space $ S:=\mbG / \mfG=\{\mfG x \ : \ x\in \mathbb{G}\}$  is a smooth manifold. Furthermore, the quotient map $\pi: \mbG\to S$ is a smooth normal covering map  , c.f. \cite[Thm.~21.29]{lee_13_smooth_manifolds} and the canonical $\mbG$ right action 
$$S\times \mbG\to S, \qquad (\mfG x, x') \mapsto \mfG x x' \ $$
makes it a homogeneous space.
We call $\mfG$ a \textit{lattice} if  $ S= \mbG / \mfG  $ is a compact space. This is equivalent to $S$ carrying a finite $\mbG$ invariant measure, which we shall denote by $d(\mfG x)$, c.f. \cite[Thm.~2.1]{raghunathan_07_discrete}. Throughout this article we assume that every homogeneous Lie group we work with carries a lattice, the following theorem, \cite[Thm.~2.12]{raghunathan_07_discrete}, gives sufficient conditions for this to be the case, which all our examples satisfy. 
\begin{theorem}
Let $\mbG$ be a simply connected nilpotent Lie group and $\mfg$ its Lie algebra. Then, $\mbG$ admits a lattice if and only if $\mfg$ admits a basis with respect to which the structure constants of $\mfg$ are rational.
\end{theorem}
Let us point out that there do exist nilpotent Lie groups that do not admit a basis with respect to which the structure constants  are rational, see \cite[Rem.~2.14]{raghunathan_07_discrete} for an example.\\

Denote by $\pi^\ast: C(S) \to C(\mbG)$ the pull-back under $\pi$. We define a convolution map
$$*_S : C^\infty(S)\times C_{c}^\infty(\mbG) \to C^\infty(S),\quad (f, \phi) \mapsto f *_S \phi$$
as the unique map, such that the following diagram commutes
\[ \begin{tikzcd}
C^\infty(S)\times C_{c}^\infty(\mbG)  \arrow{r}{ *_S } \arrow[swap]{d}{\pi^\ast \times \id} & C^\infty(S)\arrow{d}{\pi^\ast }  \\%
C^\infty(\mbG)\times C_{c}^\infty(\mbG) \arrow{r}{*} & C^\infty(\mbG)\ ,
\end{tikzcd}
\]
where convolution on the bottom row is convolution on the group as defined in \ref{eq:def_convolution}.
In order to check that this map is well defined, we use the following terminology. A function $f\in C^\infty (\mbG)$ is called left $\mfG$ periodic, if for every $x\in \mbG$ and $n\in \mfG$ it holds that 
$$f(x)= f(nx) \ .$$
Thus, we only need to check that for any $\phi \in C_{c}^\infty(\mbG)$ and any left $\mfG$ periodic function $f\in C^\infty (\mbG)$, the function  
$f* \phi$ is left $\mfG$ periodic. Indeed for any $n\in \mfG$ and $x\in \mbG$, it holds that
\begin{align*}
f* \phi (nx)&= \int_{\mbG} f(y) \phi(y^{-1}nx) dy \\
&= \int_{\mbG} f(y) \phi((n^{-1}y)^{-1}x) dy \\
&= \int_{\mbG} f(ny) \phi(y^{-1}x) dy \\
&= \int_{\mbG} f(y) \phi(y^{-1}x) dy \ .
\end{align*}	
By duality, one naturally extends the notion of convolution on $S$ to pairs, $(\xi,\zeta)\in \mcD'(S)\times \mcD'_c(\mbG)$, where $\mcD'_c(\mbG)$ denotes distributions on $\mbG$ with compact support.
\subsection{Concrete Examples}\label{subsec:group_examples}
In order to cement ideas we present two concrete examples of non-abelian, homogeneous Lie groups, both of which are identified with fixed global charts. These will be revisited in Section~\ref{subsec:kernel_examples} when we discuss the heat operator on the Heisenberg group and Kolmogorov type operators.
\subsubsection{The Heisenberg group}\label{subsec:heisenberg_group} Given $n\geq 1$ we define the Heisenberg group $\mbH^n$ as the set $\mbR^{2n}\times \mbR$ with the group law,
\begin{equation*}
(x,y,z) (x',y',z') = \left(x+x',y+y', z+z' + \sum_{i=1}^n \left(x'_iy_i - x_iy'_i\right)  \right).
\end{equation*}
\begin{remark}
Note that one may equally define the complex Heisenberg group on $\mbC^{n}\times \mbR$ equipped with the group law, 
\begin{equation*}
	(u,z)(u',z') = \left(u+u',z+z' + \Im \sum_{i=1}^{n} u_i \bar{u}'_i\right),
\end{equation*}
where $\Im  z$ denotes the imaginary part. One sees that these definitions are equivalent after identifying $\mbC^n$ with $\mbR^{2n}$.
\end{remark}
The origin $e=(0,0,0)$ is clearly the identity and for $(x,y,z)\in \mbR^{2n}\times \mbR$ one has $(x,y,z)^{-1}= (-x,-y,-z)$. The Lie algebra, $\mfh$ of $\mbH^n$ is identified with $\mbR^{2n+1}$ and spanned by the basis of left-invariant vector fields,
\begin{equation*}
A_i(x,y,z) = \partial_{x_i} + y_i \partial_z,\quad B_i(x,y,z) = \partial_{y_i} - x_i \partial_z,\quad C(x,y,z) = \partial_z
\end{equation*}
and equipped with the Lie bracket $[A,B]= AB-BA$. We observe that for any $(x,y,z)\in \mbH^n$ and $i =1,\ldots,n$,
\begin{equation}\label{eq:Heisenberg_commute}
[A_i,B_i](x,y,z) =- 2C .
\end{equation}
We say that a Lie algebra is graded if  there exist vector spaces $\{W_k\}_{k \geq 1}$ where only finitely many $W_k$ are non-zero, $\mfg = \bigoplus_{k=1}^\infty W_k$ and $[W_i,W_j]\subset W_{i+j}$. 
\begin{definition}\label{def:stratified}
A homogeneous Lie group $\mbG$ is called stratified (or a Carnot group) if its Lie algebra is graded and generated by $W_1$.
\end{definition}
Due to \eqref{eq:Heisenberg_commute}, we see that if we equip $\mbH^n$ with the dilation map,
\begin{equation*}
\lambda \cdot (x,y,z) := (\lambda x, \lambda y, \lambda^2z),
\end{equation*}
then $\mbH^n$ becomes a stratified group. Refer to \cite[Sec.~3.3.6]{bramanti_14_invitation} for a discussion of generalisations of this structure.

A simple example of a lattice on the Heisenberg group is the set of integer vectors $(a,b,c)\in \mfH^n:= \mbZ^{2n}\times \mbZ$ equipped with the group law as defined above. Note that it is not a normal subgroup and thus the quotient space $\mbH^n/\mfH^n$ is not a group. However, since it is a lattice the homogeneous space $\mbH^n/\mfH^n$ is compact, c.f. Subsection~\ref{sec:discrete_subgroups}. 
\subsubsection{Matrix Exponential Groups} \label{subsec:matrix_groups}
For $n\geq 1$, let $B$ be a rational $n\times n$ block matrix of the form,
\begin{equation*}
B = \begin{pmatrix}
	0 & B_1 & 0& \cdots &0\\
	0& 0 & B_2 &\cdots &0\\
	\vdots & \vdots & \ddots & \ddots & \vdots\\
	0 & 0&\cdots & 0 & B_{k}\\
	0 & 0 &\cdots  &0 &0
\end{pmatrix}
\end{equation*}
with each $B_i$ a $p_{i-1}\times p_{i}$ block matrix of rank $p_i$, where $n\geq p_0 \geq p_1\geq \cdots \geq p_{k}$ and $\sum_{i=0}^{k}p_i=n$. Note that this implies the zero blocks on the diagonal are all $p_i\times p_i$ square matrices. We can equip $\mbR\times \mbR^{n}$ with an associated Lie structure by defining the group law
\begin{equation*}
(t,z)(s,z') := \left(t+s,z'+\exp\left(sB^\top\right)z\right).
\end{equation*}
To define the dilation we begin by decomposing according to the structure of the block matrix $B$, writing
\begin{equation*}
\mbR^n  =  \mbR^{p_0} \times \cdots \times  \mbR^{p_{n}}.
\end{equation*}
Thus the action of $B^\top$ on $z=(z_0,..., z_k)\in  \mbR^{p_0} \times \cdots \times  \mbR^{p_{k}}$ is written $B^\top z=(B^{\top}_1 z_0,\,B^{\top}_2z_{1},...B^{\top}_kz_{k-1})$ etc. Then we set
\begin{equation*}\label{eq:matrix_exp_dilation}
\lambda\cdot  (t,z_0,\ldots,z_k):= (\lambda^2 t, \lambda z_0, \ldots,\lambda^{2k+1} z_k).
\end{equation*}
The origin $e=(0,0)$ is again the identity element and $(t,z)^{-1} = (-t,-\exp(-tB^\top)z)$. The Lie algebra is spanned by the translation invariant vector fields,
\begin{equation*}
X_i(t,z) = \partial_{z_i}\quad \text{for } i=1,\ldots,p_0 \quad \text{and}\quad Y(t,z) = \partial_t - (B z) \cdot \nabla = \partial_t - \sum_{i,j = p_0+1}^n b_{ij}z_i \partial_{z_i}.
\end{equation*}
This defines the matrix exponential group associated to $B$ and one sees that it is not stratified.

The simplest non-trivial example is to set $n=2$ and
\begin{equation*}
B=	\begin{pmatrix}
	0 & 1\\
	0&0
\end{pmatrix},
\end{equation*} 
so that using the suggestive notation $(t,v,x)\in \mbR\times \mbR\times \mbR$, the group law becomes,
\begin{equation*}
(t,v,x)(s,w,y) = (t+s,v+w,x+y+sv).
\end{equation*}
The equivalent scaling as above is to set $\lambda \cdot (t,x,v) = (\lambda^2t,\lambda v, \lambda^3 x)$. We refer to \cite{manfredini_97} for more details and Section \ref{subsec:kernel_examples} below for a discussion of natural second order hypoelliptic linear operators associated to these groups. 
\section{Regularity Structures and Models}\label{sec:r_structures}
\begin{definition}\label{def:reg_structure}
A regularity structure is a pair $\mcT=(\opT,\opG)$ consisting of the following elements:
\begin{enumerate}
	\item A graded vector space $\opT=\bigoplus_{\alpha\in A} \opT_\alpha$ where
	\begin{itemize}
		\item the index set $A\subset \mbR$ is discrete, bounded from below and contains zero,
		\item each $\opT_\alpha$ is finite dimensional with a fixed norm $|\cdot|_\alpha$. We write $\mcQ_\alpha: \opT \to \opT_\alpha$ for the canonical projection,
		\item $\opT_0$ is isomorphic to $\mathbb{R}$ with a distinguished element $\mathbf{1}\in \opT_0$, such that $|\mathbf{1}|_0=1$.
	\end{itemize}
	The space $\opT$ is called the structure space.
	\item A group $\opG$ of linear operators acting on $\opT$, such that for every $\Gamma\in \opG$ it holds that  $\Gamma|_{\opT_0}$ is the identity map and for all $\tau \in \opT_\alpha$:
	$$\Gamma \tau-\tau\in \bigoplus_{\beta<\alpha} \opT_\beta.$$
	The group $\opG$ is called the structure group of $\,\mcT$.
\end{enumerate}

A sector is a $\opG$ invariant subspace $V \subset \opT$ and if $V\neq \{0\}$ the regularity of the sector $V$ is defined as $\min\{\alpha \in A \ : \ V\cap \opT_\alpha \neq \{0\} \}$.
\end{definition}
We make use of the natural shorthands, $\opT_{>\alpha},\,\opT_{\geq \alpha},\, \opT_{<\alpha},\,\opT_{\leq \alpha}$ and the projections $\mcQ_{>\alpha}$ etc.
\begin{definition}\label{def:model}
Given a regularity structure $\mcT=(\opT,\opG)$ and $r\in \mbN$ such that $r>|\min A|$, a model for $\mcT$ is a pair $M=(\Pi, \Gamma)$, consisting of 
\begin{itemize}
	\item a realisation map $\Pi: \mathbb{R}^d\to L(T, \mcD'(\mbG)),\ x\mapsto \Pi_x$, such that for any compact set $\mfK\subset \mbG$ one has $\|\Pi\|_{\gamma;\mfK}:=\sup_{x\in \mfK }\|\Pi\|_{\gamma,x}<+\infty$ for all $\gamma>0$, where
	$$\|\Pi\|_{\gamma,x}:= \sup_{\zeta\in A\cap (-\infty,\gamma) }\sup_{\tau \in T_\zeta} \sup_{\lambda<1}  \sup_{\phi\in \mfB_r} \frac{|\langle \Pi_x \tau, \phi^\lambda_x \rangle |}{|\tau|_{\zeta} \lambda^\zeta},$$
	\item a re-expansion map $ \Gamma: \mathbb{R}^d\times \mathbb{R}^d \to \opG, \ (x,y)\mapsto \Gamma_{x,y}$, which satisfies the algebraic conditions
	$$\Pi_x\Gamma_{x,y} = \Pi_y, \qquad \Gamma_{x,y}\Gamma_{y,z}=\Gamma_{x,z}$$ 
	and the analytic condition $\|\Gamma\|_{\gamma;\mfK}:=\sup_{x,y\in \mfK:\  {|y^{-1}x|}<1} \|\Gamma\|_{x,y,\gamma}<+\infty$ for all $\gamma>0$,
	where 
	$$\|\Gamma\|_{x,y,\gamma}:=\sup_{\zeta\in A\cap (-\infty,\gamma)} \sup_{A\ni \beta<\zeta}\sup_{\tau \in T_\zeta} \frac{|\Gamma_{x,y}\tau|_\beta}{{|y^{-1}x|}^{\zeta-\beta}|\tau|_{\zeta}}.$$ 
\end{itemize}
Lastly, we denote by $\mathcal{M}_\mcT$ the space of models for $\mcT$ equipped with the semi-norms $\|M\|_{\gamma; \mfK} :=\|\Pi\|_{\gamma;\mfK}+\|\Gamma\|_{\gamma; \mfK} $.
\end{definition}
\begin{remark}
In the rest of the article, often without any further comment, we will write $$r:= \min \left\{n\in \mathbb{N}\ : \ n> |\min A|\right\}\ .$$
\end{remark}
\subsection{The Polynomial Regularity Structure}\label{sec:polynomial_r_struct}
As an important example we describe the polynomial regularity structure and canonical model which are crucial in the analysis of singular SPDEs on homogeneous Lie groups. We define the structure space $\bar{\opT}$ to be the symmetric tensor (Hopf) algebra generated by $\{\pmb{\eta}_i\}_{i=1}^d$, which we think of as abstract lifts of the monomials $\{\eta_i\}_{i=1}^d$. For a multi-index $I$, we write $\bfeta^{I}:=\bfeta_1^{i_1}\cdot...\cdot\bfeta_d^{i_d}$ as well as $\mathbf{1}:=\bfeta^{0}$. The group $\bar{\opG}$ is given by a copy of $\mbG$ acting by 
$$ g\mapsto \Big( \Gamma_{g}: \bfeta_j\mapsto \bfeta_j  + \sum_{I,J\neq 0,\ d(I)+d(J)=\fraks_j} C_j^{I,J}\eta^I(g)\bfeta^J\ +\eta_j (g)\mathbf{1}\Big) 
$$
and $\Gamma_{g} \bfeta^I= \big(\Gamma_{g} \bfeta\big)^I$ .
The canonical polynomial model is defined by setting 
$$ \Pi_{x}\bfeta_j (z)=\eta_j (x^{-1}z)$$
and
$$ \Gamma_{x,y}\bfeta_j=\bfeta_j  + \sum_{I,J\neq 0,\ d(I)+d(J)=\fraks_j} C_j^{I,J}\eta^I(y^{-1} x)\bfeta^J\ +\eta_j (y^{-1} x)\mathbf{1}\ .$$
These maps are extended multiplicatively to all of $\bar{\opT}$.
Using Prop.~ \ref{prop:factorisation_formula} in the third equality,
\begin{align*}
\Pi_{x}\Gamma_{x,y}\bfeta_j(z) &= \Pi_{x}\Big(\bfeta_j  + \sum_{I,J\neq 0,\ d(I)+d(J) =\fraks_j} C_j^{I,J}\eta^I(y^{-1} x)\bfeta^J\ +\eta_j (y^{-1} x)\Big)(z)\\
&= \eta_j(x^{-1} z)  + \sum_{I,J\neq 0,\ d(I)+d(J)=\fraks_j} C_j^{I,J}\eta^I(y^{-1} x)\eta^J(x^{-1} z)\ +\eta_j (y^{-1} x) \\
&= \eta_j(y^{-1} x x^{-1}z)\\
&= \eta_j(y^{-1} z)\\
&= \Pi_y \bfeta_j (z).
\end{align*}
\subsubsection{Derivatives and abstract Polynomials}\label{sec:derivatives_and_abstract_polynomials}
Next we lift the vector fields $X_i$ to abstract differential operators $\mathcal{X}^i$ on $\bar{\opT}$ of degree $\fraks_j$, c.f. Section~\ref{subsec:local_ops}. We set
$$\mathcal{X}^i \bfeta^j = \delta_{i,j}\mathbf{1} + \sum_{I\neq 0,\  d(I)= \fraks_j-\fraks_i} C_j^{I, e_i} \bfeta^I$$
and extend this definition to the whole regularity structure by the Leibniz rule.
The conditions to be an abstract differential operator are checked directly:
\begin{enumerate}
\item Indeed $\mathcal{X}^i: \opT_\alpha\mapsto T_{\alpha-\fraks_i}$.
\item By the Leibinz rule the second property is checked by showing that
\begin{equation}\label{eq:remains_to_check}
	\mathcal{X}^i  \Gamma_{x,y}\bfeta^j= \Gamma_{x,y}\mathcal{X}^i \bfeta^j \ .
\end{equation}
Note that $\Pi_x \mathcal{X}^i \bfeta^j =X_i \Pi_x\bfeta^j\ ,$ since
$$\Pi_x \mathcal{X}^i \bfeta_j (z) = \Pi_x  \Big(\delta_{i,j}\mathbf{1} \,+ \hspace{-0.2em}\sum_{I\neq 0,\  d(I)= \fraks_j-\fraks_i} C_j^{I, e_i} \bfeta^I\Big) = \delta_{i,j}\, +\hspace{-0.2em} \sum_{I\neq 0,\  d(I)= \fraks_j-\fraks_i} C_j^{I, e_i} \eta^I(x^{-1}z)$$
and using left invariance of the vector field $X_i$ as well as \eqref{eq:derivative_of_polynomial}
$$X_i \Pi_x\bfeta_j  (z)= (X_i \eta_{j}( x^{-1} \cdot )) (z) = (X_i \eta_{j}) (x^{-1}z)= \delta_{i,j} + \sum_{I\neq 0,\  d(I)= \fraks_j-\fraks_i} C_j^{I, e_i} \eta^I(x^{-1}z)\ .$$
Thus \eqref{eq:remains_to_check} follows from the injectivity of the maps $\Pi_x$ and the relation $\Pi_x\Gamma_{x,y}= \Pi_y$.
\end{enumerate}
\begin{remark} In addition to showing that $\mathcal{X}_i$ is an abstract differential operator, we have shown that it also realizes $X_i$ for the polynomial model.
\end{remark}
\subsubsection{Abstract Taylor Expansions}\label{sec:abstract_taylor_polynomials}
The following operator,  which sends a smooth function to its abstract Taylor expansion, will be used throughout the article.
\begin{definition}\label{def:Taylor_as_Modelled_Dist}
For $x\in \mbG$, we define the family of maps
$$\opbfP^a_x: \mcC^a (\mbG) \to \bar{\opT}$$ 
where for $a>0$ 
$\opbfP^a_x[f]$ is characterised by the fact that $\Pi_e \opbfP^a_x[f] =\opP^a_x[f]\in \mathcal{P}_a$ and for $a\leq 0$ it satisfies $\opbfP^a_x[f]=0$.
\end{definition}
\begin{remark}\label{rem:taylor_extension}
Note that the map $\opbfP^a_x$ is well defined by Remark~\ref{rem:taylor_extension0} . We shall almost exclusively use it, when its argument is a smooth function $f\in C^\infty(\mbG)\subset \mcC^a(\mbG)$.
\end{remark}
\begin{lemma}\label{lem:coefficients_of_taylor}
Let $f\in C^\infty(\mbG)$ and $a \geq 0$. Then for each multi-index $I$ the coefficient of $\pmb{\eta}^I$ in $\opbfP^a_x[f]$ is given by a linear combination (depending only on $\mbG$) of $\{X^Kf(x)\}_{d(K)\leq d(I)}$. Furthermore, for $b\geq a$ one has
\begin{equation}\label{eq:abstract_taylor_project}
	\opbfP_x^a[f]= \mcQ_{< a} \opbfP^{b}_x [f]\ .
\end{equation}		
\end{lemma}
\begin{proof}\,
We first observe that \eqref{eq:abstract_taylor_project} holds since the polynomial $\Pi_e \mcQ_{\leq a} \opbfP^{b}_x [f]$ satisfies the property required in Definition~\ref{def:taylor}. The first part of the lemma follows by combining Proposition~\ref{prop:uniqueness_taylor} with \eqref{eq:abstract_taylor_project}.	
\end{proof}
\begin{lemma}\label{lem:second_polynomial}
In the setting of the above lemma one has 
$$|X^I \Pi_e \opbfP^a_x[f] (z)|= |X^I \opP^a_x[f] (z)| \lesssim \sum_{d(I)\leq d(J)< a} \sup_{d(J')\leq d(J)} |X^{J'}f(x)||z|^{d(J)-d(I)} \ .$$
\end{lemma}
\begin{proof}\,
Since the coefficient of ${\eta}^J$ in $ {\opP}^a_x[f]$ is bounded by a universal multiple of $\sup_{d(J')\leq d(J)} |X^{J'}f(x)|$ by Lemma~\ref{lem:coefficients_of_taylor} the claim follows using the fact that 
$|X^I\Pi_e \pmb{\eta}^J(z)|\lesssim |z|^{{d(J)-d(I)}}$ .
\end{proof}
The next lemma will be useful when proving Schauder estimates in Section~\ref{subsec:schauder}, since it allows us to circumvent the issue of having non-explicit Taylor polynomials.
\begin{lemma}\label{lem:to_bound_polynomials}
Let $\delta>0$ and $P\in \bar{\opT}_{<\delta}$. If, for some $\varepsilon \in [0,1]$,
$$\left| (X^{I}\Pi_e P) (e)\right|\leq \eps^{\delta-d(I)}$$
for all  $I\in \mathbb{N}^d$ such that $\delta> d(I)$,
then it holds that $|P|_a\lesssim_{\delta,\mbG} \eps^{\delta-a}$ for all $a\leq \delta$. 
\end{lemma}
\begin{proof}\,
We observe that $\opP_x^\delta[\Pi_e P]=\Pi_e P$. Thus, the claim follows directly from Lemma~\ref{lem:coefficients_of_taylor}.
\end{proof}
\subsection{Modelled Distributions}\label{sec:modelled_distributions}
\begin{definition}\label{def:modelled_dist}
Given a regularity structure  $\mcT=(\opT,\opG)$, a model $M=(\Pi, \Gamma)$  and $\gamma\in \mbR$ we define $
\msD_M^\gamma$ as the space of all continuous maps $f:\mbG\to \opT_{<\gamma}$, such that for all $\zeta \in A\cap (-\infty,\gamma)$ the following bounds hold for every compact set $\mfK \subset \mbG$
$$\sup_{x\in \mfK} |f(x)|_\zeta <+\infty, \qquad \sup_{\substack{x,y\in \mfK,\\0<{|y^{-1}x|}\leq 1}} \frac{|f(y)-\Gamma_{x,y}f(x)|_\zeta}{{|y^{-1}x|}^{\gamma-\zeta}}<+\infty \ .$$
We define the corresponding semi-norm $\|\,\cdot\,\|_{\gamma;\mfK}$
on $\msD_M^\gamma$ by setting,
\begin{equation*}\label{eq:modelled_dist_norm}
	\|f\|_{\gamma;\mfK} := \sup_{x\in \mfK} \sup_{\zeta<\gamma}|f(x)|_\zeta + \sup_{\zeta<\gamma} \sup_{\substack{x,y\in \mfK, \\ 0<{|y^{-1}x|}\leq 1}}\frac{|f(y)-\Gamma_{x,y}f(x)|_\zeta}{{|y^{-1}x|}^{\gamma-\zeta}}.
\end{equation*}
Given two models $M = (\Pi,\Gamma),\, \bar{M}=(\bar{\Pi},\bar{\Gamma})$, two modelled distributions $f\in \msD^\gamma_M,\, \bar{f}\in \msD^\gamma_{\bar M}$ and a compact set $\mfK\subset\mbG$ we define the quantity, 
\begin{equation*}
	\begin{aligned}
		\|f;\bar{f}\|_{\gamma;\mfK} :=& \sup_{x \in \mfK}\sup_{\zeta<\gamma}|f(x)-\bar{f}(x)|_\zeta \\
		&+ \sup_{\substack{x,y\in \mfK\\ 0<{|y^{-1}x|} \leq 1}}\sup_{\zeta<\gamma} \frac{|f(y)-\bar{f}(y)-\Gamma_{x,y}f(y)+\bar{\Gamma}_{x,y}\bar{f}(x)|_\zeta}{{|y^{-1} x|}^{\gamma-\zeta}}.
	\end{aligned}
\end{equation*}
For the set of modelled distributions taking values in a sector $V\subset \opT$ we write $\msD^\gamma_M(V)$ and if the regularity of the sector is $\alpha \in A$ we often use the shorthand $\msD^\gamma_{\alpha;M}$. We will freely drop the explicit dependence on the model, image sector and its regularity when the context is clear.
\end{definition}
\subsubsection{Reconstruction}\label{subsec:Reconstruction}
\begin{theorem}[Reconstruction Theorem]\label{th:Recontruction}
Let $\mcT=(\opT,\opG)$ be a regularity structure with $\underline{\alpha} = \min A$. Then for every $\gamma>0$ and $M = (\Pi,\Gamma) \in \mathcal{M}_{\mcT}$, there exists a unique, continuous linear map $\mcR_{M}:\msD_M^\gamma \rightarrow \mcC^{\underline{\alpha}}$, called the reconstruction operator associated to $M$, such that for any compact $\mfK$ and $\lambda \in (0,1]$ 
\begin{equation}\label{eq:Reconstruction_Regularity}
	\sup_{\psi \in \mfB_{r}} |\langle \mathcal{R}_M f- \Pi_x f(x),\psi^\lambda_x\rangle|\lesssim_{\mfK} \lambda^\gamma\|f\|_{\gamma;B_{2\mu\lambda}(x)}\|\Pi\|_{\gamma;B_{2\mu\lambda}(x)} \ ,
\end{equation}
uniformly over $x\in \mfK$, where $\mu\geq 1$ was defined in Remark~\ref{rem:triangle_inequality}.
Furthermore the map $M \rightarrow \mcR_M$ is locally Lipschitz continuous in the sense that for a second model $\bar{M}=(\bar{\Pi},\bar{\Gamma})$ and $\bar{f}\in \msD_{\bar{M}}^\gamma$, for every $\lambda \in (0,1]$
\begin{equation}\label{eq:rec2}
	\begin{aligned}
		\sup_{\psi \in \mfB_{r}} |\langle \mathcal{R}_M f-\mathcal{R}_{\bar{M}}\bar{f} -\Pi_x f(x)+\bar{\Pi}_x \bar{f}(x),\psi^{\lambda}_x\rangle|
		\lesssim_{\mfK} \lambda^\gamma \bigg(&\|f;\bar{f}\|_{\gamma;B_{2\mu\lambda}(x)}\|\bar{\Pi}\|_{\gamma;B_{2\mu\lambda}(x)}\\
		&+\|f\|_{\gamma;B_{2\mu\lambda}(x)}\|\Pi-\bar{\Pi}\|_{\gamma;B_{2\mu\lambda}(x)} \bigg)\ ,
	\end{aligned}
\end{equation}
\end{theorem}
\begin{remark}
Existence of a reconstruction operator for $\gamma \leq 0$ also holds, however, uniqueness does not. In the case $\gamma=0$ the analogous bound to \eqref{eq:Reconstruction_Regularity} contains an additional logarithmic correction on the right hand side, c.f. \cite{caravenna_zambotti_20}.
\end{remark}
\begin{remark}
It follows from the definition of the reconstruction operator, that if $f$ is a modelled distributions with values in a sector $V\subset \opT$ of regularity $\alpha\geq \underline{\alpha}$, then $\mathcal{R}_Mf\in C^{\alpha}$. 
\end{remark}
In order to prove Theorem~\ref{th:Recontruction} we require two preliminary results. As in \cite[Sec.~13.4]{friz_hairer_20_introduction} we make the observation that for every $N\geq 0$ there exists a $\rho : \mbG\rightarrow \mbR$, smooth and compactly supported in $B_1(e)$ and such that
\begin{equation}\label{eq:PolynomialAnnhiliate}
\int \eta^I(x) \rho(x)\dd x = \delta_{I,0},\quad 0<d(I)\leq N,
\end{equation}
where the $\delta$ here denotes the Kronecker delta applied componentwise to the multi-index. For $\mfr>1$ we define $\rho^{(n)}(x):= \mfr^{n|\mfs|}\rho (\mfr^{n}\cdot x)$ as well as,
\begin{equation*}
\rho^{(n,m)} = \rho^{(n)}\ast \rho^{(n+1)}\ast \cdots \ast \rho^{(m)},
\end{equation*}
with the convention $\rho^{(n,n)}= \rho^{(n)}$. We then have the following result.

\begin{lemma}\label{lem:convergence_of_test_functions}
If $\,\mfr> \| \rho\|_{L^1}>1$, then there exists a smooth, compactly supported function $\varphi^{(n)}=\lim_{m\rightarrow \infty} \rho^{(n,m)}$, where the convergence takes place in $\mcD(\mbG)$ and $\supp(\varphi^{(n)})\subset B_{C \mfr^{-n}}$ for $C=\frac{\mfr}{\mfr-1}$.
\end{lemma}
\begin{proof}\, 
First, note that since $\|\rho^{(m)}\|_{L^1}=\|\rho\|_{L^1}$ and $\rho^{(m)}$ is supported on a ball of radius $\mfr^{-m}$, it follows from the mean value theorem on $\mbG$ (c.f. \eqref{eq:common_form_taylor} with $a=0$) that
\begin{equation*}
	|f\ast \rho^{(m)}(x)-f(x)|  = \left| \int_{\mbG} (f(y)-f(x))\rho^{(m)}(y^{-1}x)\dd y\right| \lesssim \max_{i=1,...,d} \|Y_i f\|_{{L^\infty}} \|\rho\|_{L^1} \mfr^{-m} \ .
\end{equation*}
Secondly, since
\begin{equation}\label{eq:derive_and_conv}
	Y^I \rho^{(n,m)}=  Y^I(\rho^{(n)}\ast \rho^{(n+1,m)})=(Y^I \rho^{(n)})\ast  \rho^{(n+1,m)}
\end{equation}
one finds by applying Young's convolution inequality $m$-times, c.f. Section~\ref{subsec:distributions_convolutions}, that
\begin{align*}\|Y_i \rho^{(n,m)} \|_{L^\infty}= \|(Y_i \rho^{(n)})\ast  \rho^{(n+1,m)} \|_{L^\infty}\leq \| Y_i \rho^{(n)}\|_\infty \|\rho\|_{L^1}^{m-n-1} &\leq \mfr^{n(\mfs_i+|\mfs|)} \|Y_i \rho\|_{L^\infty} \|\rho\|^{m-n-1}_{L^1} \\
	&\leq \mfr^{2n|\mfs|} \|Y_i \rho\|_{L^\infty} \|\rho\|^{m-n-1}_{L^1}
\end{align*}
and therefore
\begin{align*}
	\|\rho^{(n,m)}- \rho^{(n,m-1)}\|_{L^\infty} = \|\rho^{(n,m-1)}\ast \rho^{(m)}- \rho^{(n,m-1)}\|_{L^\infty} 
	&\lesssim \max_{i=1,...,d} \|Y_i \rho^{(n,m-1)}\|_{{L^\infty}}   \|\rho\|_{L^1} \mfr^{-m} \\
	&\leq  \mfr^{2n|\mfs|}\max_{i=1,...,d}\|Y_i\rho\|_{L^\infty} \|\rho\|^{m-n-2}_{L^1} \mfr^{-m}  \\
	&\leq \mfr^{2n|\mfs|} \frac{\max_{i=1,...,d}\|Y_i\rho\|_{L^\infty} }{\|\rho\|^{n+2}_{L^1} }   \left(\frac{\|\rho\|_{L^1}}{ \mfr }\right)^m 
\end{align*}
which is summable in $m$ since we assumed that $\mfr>\|\rho\|_{L^1}$. Thus we may write,
\begin{equation*}
	\rho^{(n,m)} =\rho^{(n)} +\sum_{k=0}^{m-n-1} \rho^{(n,m-k)} - \rho^{(n,m-1-k)}, 
\end{equation*}
and it follows that $\rho^{(n,m)}$ converges uniformly as $m\to +\infty$. Using \eqref{eq:derive_and_conv} we obtain convergence in $\mcD(\mbG)$. 
It remains to check the support of $\varphi^{(n)}$. For two functions $f_1,f_2$ such that $\supp(f_i)\in B_{r_i}$ one has
$\supp(f_1\ast f_2)\in B_{r_1+r_2}$, hence it follows that $ \varphi^{(n)}$ is supported in a ball of radius $\sum_{m=n}^\infty \mfr^{-m}= \frac{\mfr^{-n}}{1-\mfr^{-1}} $.
\end{proof}
It follows from the definitions that
$\varphi^{(n)}= \rho^{(n)}\ast \varphi^{(n+1)}$. We set
\begin{equation*}
\tilde{\rho}^{(m,n)} := \tilde{\rho}^{(m)}\ast \tilde{\rho}^{(m-1)}\ast \cdots \ast \tilde{\rho}^{(n)}
\end{equation*}
and using \eqref{eq:convolution_reorder} we note that $\tilde{\varphi}^{(m+1)}\ast \tilde{\rho}^{(m)}=\tilde{\varphi}^{(m)}$ and $ \tilde{\rho}^{(m,n)} \to \tilde{\varphi}^{(n)}$ in $\mcD(\mbG)$ as $m\to \infty$.
\begin{lemma}\label{lem:Pre_Reconstruct_Lemma}
Let $\mfr>1$ and $\rho$ be as in Lemma~\ref{lem:convergence_of_test_functions} Let $\alpha>0$ and $\xi_n :\mbG\rightarrow \mbR$ be a sequence of functions such that for every compact $
\mfK\subset \mbG$ there exists a $C_{\mfK}$ such that $\sup_{x \in \mfK}|\xi_n(x)|
\leq C_{\mfK}\mfr^{\alpha n}$ and such that $\xi_n =  \xi_{n+1}\ast \tilde{\rho}^{(n)}$. Then the sequence $\xi_n$ is Cauchy in $\mcC^{-\beta}(\mbG)$ for every $\beta>\alpha$ and the limit $\xi$ satisfies $\xi_n = \xi\ast \tilde{\varphi}^{(n)}$.
If furthermore, for some $x\in \mbG$ and $\gamma>-\alpha$ one has the bound
$$|\xi_n(y)| \leq \mfr^{\alpha n}\left(|x^{-1} y|^{\gamma+\alpha}+\mfr^{-(\gamma+\alpha)n}\right)$$
uniformly over $n\geq 0$ and $y\in \mbG$ such that $|x^{-1} y|\leq 1$, then $|\langle \xi,\psi^\lambda_x\rangle|\lesssim \lambda^\gamma$ for all $\lambda \leq 1$ and $\phi \in \mfB_{r}$, where $r= -[-\alpha]$.
\end{lemma}
\begin{proof}\,
The proof follows along the same steps as that of \cite[Lem. 13.24]{friz_hairer_20_introduction}, but one has to be careful since convolution is non-commutative in our setting. Let $\lambda \in (0,1]$ and $\psi \in \mfB_{r}$, we first establish the bound
\begin{equation}\label{eq:RecCauchyBound}
	|\langle \xi_{n}-\xi_{n+1},\psi^\lambda_x\rangle| \lesssim \lambda^{-\beta}\mfr^{(\alpha-\beta)n}
\end{equation}
uniformly over $\psi \in \mfB_{r}$, $\lambda \in (0,1]$ and locally uniformly over $x \in \mbG$. First observe the trivial bound,
\begin{equation*}
	|\langle \xi_{n}-\xi_{n+1},\psi^\lambda_x\rangle| \leq \sup_{x \in \bar\mfK}\left(|\xi_n(x)| + |\xi_{n+1}(x)|\right)\|\psi^\lambda_x\|_{L^1} \leq (1+\mfr^{\alpha})C_{\bar \mfK}\bar{C} \mfr^{\alpha n},
\end{equation*} 
where $\bar{C}:= \sup \{\int |\psi^\lambda(x)|\dd x \ : \ \psi \in \mfB_{r}\}<+\infty$ . Hence, when $\lambda\leq \mfr^{-n}$ the bound \eqref{eq:RecCauchyBound} holds directly.

In the case $\mfr^{-n}\leq \lambda$, using \eqref{eq:throwing_convolution} we rewrite
$$|\langle \xi_{n}-\xi_{n+1},\psi^\lambda_x\rangle| =\langle \xi_{n+1}\ast \tilde{\rho}^{(n)}-\xi_{n+1},\psi^\lambda_x\rangle| = |\langle \xi_{n+1},  \psi_x^\lambda*\rho^{(n)}   -\psi_x^\lambda\rangle|.$$
By Taylor's theorem and in particular Remark~\ref{rem:rescaled_taylor}
\begin{equation}\label{eq:TestFuncTaylorBnd}
	|\psi_x^\lambda(z) - \tilde{\opP}^r_y[\psi_x^\lambda](z)| \, \lesssim_{r}  \sum_{\delta>0}\lambda^{-(r+\delta)-|\mfs|}|y^{-1} z|^{r+\delta}\ ,
\end{equation}
where the sum runs over a finite set. It follows from \eqref{eq:def_convolution}, Remark~\ref{rem:factorising_constants}, Remark~\ref{rem:rescaled_taylor}  and \eqref{eq:PolynomialAnnhiliate} that we have 
\begin{equation*}
	\tilde{\opP}^r_y[\psi^\lambda_x]*\rho^{(n)} (z)= \int \tilde{\opP}^r_y[\psi^\lambda_x](z w^{-1}) \rho^{(n)}(w) dw=\int \tilde{\opP}^0_y[\psi^\lambda_x](z w^{-1}) \rho^{(n)}(w) dw= \psi_x^{\lambda}(y) 
\end{equation*}
and thus
\begin{equation*}
	\psi_x^\lambda*\rho^{(n)}(y) -\psi_x^\lambda(y) = (\psi_x^\lambda-\tilde{\opP}^r_y[\psi^\lambda_x])\ast \rho^{(n)}(y),
\end{equation*}
which by \eqref{eq:TestFuncTaylorBnd} is bounded uniformly by a multiple of $\sum_{\delta>0}\lambda^{-(r+\delta)-|\mfs|}\mfr^{-n(r+\delta)}\ $ and supported on a ball of radius $\lambda + \mfr^{-n}\leq 2\lambda$. Note that we were able to keep the convolution with $\rho^{(n)}$ on the right in both expressions by using  the rebased Taylor polynomial of Rem.~\ref{rem:rescaled_taylor}. Using the bound $|\xi_{n+1}|\lesssim_\alpha \mfr^{\alpha n}$ we conclude that
\begin{equation*}
	|\langle \xi_{n+1},  \psi_x^\lambda*\rho^{(n)}   -\psi_x^\lambda\rangle| \lesssim \sum_{\delta>0}\lambda^{-(r+\delta)}\mfr^{-n(r+\delta)} \mfr^{\alpha n} \lesssim \lambda^{-\beta}\mfr^{(\alpha-\beta)n}
\end{equation*}
where we used $\mfr^{-n}\leq \lambda$ and without loss of generality assumed that $r\geq \beta$ in the last line.
Hence $\{\xi_n\}_{n\geq 1}$ is Cauchy in $\mcC^{-\beta}(\mbG)$ and for any test function,
\begin{equation*}
	\langle \xi_n,\psi\rangle = \langle\xi_{n+1}, \psi* \rho^{(n)}\rangle = \langle \xi_{m+1}, \psi*\rho^{(n,m)}\rangle = \langle \xi, \psi*\varphi^{(n)}\rangle,
\end{equation*}
showing that $\xi_n=  \xi*\tilde{\varphi}^{(n)}$.

To prove the second claim, for any test function $\psi\in \mathfrak{B}_r$, $\lambda>0$ and $x\in \mbG$ we write,
\begin{equation*}
	\langle \xi,\psi^\lambda_x\rangle = \langle \xi_n, \psi^\lambda_x\rangle + \sum_{k\geq n} \langle \xi_{k+1}-\xi_k,\psi^\lambda_x\rangle,
\end{equation*}
where $n$ is chosen so that $\lambda \in [\mfr^{-(n+1)},\mfr^{-n}]$ and as a consequence
\begin{equation*}
	|\langle \xi_n,\psi^\lambda_x\rangle| \leq \lambda^{-|\mfs|} \mfr^{\alpha n} \int_{B_\lambda(x)} \left(|x^{-1} y|^{\gamma+\alpha}+ \mfr^{-(\gamma+\alpha)n}\right)\dd y \,  \lesssim\,  \lambda^{\gamma+\alpha}\mfr^{\alpha n} + \mfr^{-\gamma n} \lesssim \lambda^\gamma.
\end{equation*}
To bound the summands $\langle \xi_{k+1} -\xi_k,\varphi^\lambda_x\rangle$ we proceed as in the proof of the first claim to find that
\begin{align*}
	|\langle \xi_{k} -\xi_{k+1},\psi^\lambda_x\rangle| & = |\langle \xi_{k+1},  \psi^\lambda_x*\rho^{(k)}-\psi^\lambda_x\rangle| \\
	&= |\langle \xi_{k+1},  (\psi^\lambda_x-\tilde{\opP}^r_x[\psi^\lambda_x] ))*\rho^{(k)}\rangle| \\ 
	&\lesssim \sum_{\delta>0}\lambda^{-(r+\delta)-|\mfs|}\mfr^{-k(r+\delta)} \int_{B_{2\lambda}(x)} |\xi_{k+1}(y)|\dd y\\
	&\lesssim \sum_{\delta>0}\lambda^{-(r+\delta)-|\mfs|}\mfr^{k(\alpha-r-\delta)} \left(\int_{B_{2\lambda}(x)} |x^{-1} y|^{\gamma+\alpha}\dd y+\mfr^{-(\gamma+\alpha)(k+1)}\lambda^{|\mfs|}\right) \\
	&\lesssim \sum_{\delta>0}\left(\lambda^{\gamma+\alpha-r-\delta}\mfr^{k(\alpha-r-\delta)}  + \lambda^{-(r+\delta)}\mfr^{-k(\gamma+r+\delta)}\right)
\end{align*}
where the sum in $\delta$ is again over a finite set. Since $r+\delta>\alpha$ the quantity on the left is summable over $k\geq n$ and is of order $\lambda^{\gamma}$, concluding the proof.  
\end{proof}
We are now ready to proof Theorem~\ref{th:Recontruction}.
\begin{proof}[Proof of Theorem~\ref{th:Recontruction}]\,
For $m>0$ we first define the operators $\mcR^{(m,m)}:\msD^\gamma \rightarrow C(\mbG)$ by setting,
\begin{equation*}
	\left(\mcR^{(m,m)}f\right)(y) := \left( \Pi_yf(y)\ast \tilde{\varphi}^{(m)}\right)(y) = \langle\Pi_y f(y),{\varphi}^{(m)}_{y} \rangle.
\end{equation*}
We then set, for $n<m$,
\begin{equation*}
	\mcR^{(m,n)} f = \mcR^{(m,m)}f \ast \tilde{\rho}^{(m-1,n)} 
\end{equation*}
and recalling that  $\tilde{\varphi}^{(m+1)}\ast \tilde{\rho}^{(m)}=\tilde{\varphi}^{(m)}$ we find
\begin{align*}
	\mcR^{(m,n)}f - \mcR^{(m+1,n)}f
	&= \mcR^{(m,m)}f \ast \tilde{\rho}^{(m-1,n)}- \mcR^{(m+1,m+1)}f\ast\tilde{\rho}^{(m,n)}\\
	&= \big(\mcR^{(m,m)}f - \mcR^{(m+1,m+1)}f\ast\tilde{\rho}^{(m)}\big)\ast \tilde{\rho}^{(m-1,n)} \ .
\end{align*}
Using the identity
$$(F\ast \tilde{\phi})(x)= \langle F, \phi_x\rangle= \int F(y)\phi_x(y)dy \ ,$$
it follow by a straightforward computation that
\begin{align*}
	\big(\mcR^{(m,n)}f - \mcR^{(m+1,n)}f\big)(x)	
	&= \int \big(\mcR^{(m,m)}f - \mcR^{(m+1,m+1)}f\ast\tilde{\rho}^{(m)}\big)(y) {\rho}_x^{(m-1,n)}(y) dy\\
	&= \int\int \langle\Pi_yf(y)  - \Pi_zf(z),  {\varphi}_z^{(m+1)}\rangle {\rho}_y^{(m)}(z)dz  {\rho}_x^{(m-1,n)}(y) dy \ .
\end{align*}
Therefore, using the fact that $\Pi_z = \Pi_y \Gamma_{yz}$, we have,
\begin{equation*}
	\big(\mcR^{(m,n)}f - \mcR^{(m+1,n)}f\big)(x)	= \int\int \langle\Pi_y(f(y)  - \Gamma_{yz} f(z)),  {\varphi}_z^{(m+1)}\rangle {\rho}_y^{(m)}(z)dz  {\rho}_x^{(m-1,n)}(y) dy.
\end{equation*}
Then successively applying the facts that,
\begin{itemize}
	\item $\sup_{y,\,z \in B_{2\mu\mfr^{-m}}(0)}|\langle \Pi_y \tau,\varphi^{(m+1)}_z|\lesssim \| \Pi\|_{\gamma; B_{2\mu\mfr^{-m}}(0)} \mfr^{-\alpha m}|\tau|_\alpha$ for $\tau\in T_\alpha$, 
	\item $\|f(y)-\Gamma_{yz}f(z)\|_\alpha \lesssim\|f\|_{\gamma;B_{\mfr^{-m}}(y)}  \mfr^{(\alpha-\gamma)m}$ uniformly over $|y^{-1}z|\lesssim \mfr^{-m}$
	\item $\|\rho_x^{(n,m-1)}\|_{L^1}\lesssim 1$ uniformly over $m> n\geq 0$,
\end{itemize}
we establish the bound,
\begin{equation*}
	\|\mcR^{(m,n)}f - \mcR^{(m+1,n)}f\|_{L^\infty(\mfK)}\lesssim \| \Pi\|_{\gamma, \bar \mfK} \|f\|_{\gamma;\bar{\mfK}}  \mfr^{-\gamma m} \ ,
\end{equation*}
uniformly over $m\geq n\geq 0$, where $\bar{\mfK}$ denotes the two fattening of the set $\mfK$. It follows directly from the definition and properties of a model that we also have the bound,
\begin{equation}\label{eq:Approx_Reconstruct_Bnd}
	\|\mcR^{(n,n)} f\|_{L^\infty(\mfK)} \lesssim  \| \Pi\|_{\gamma, \bar \mfK} \|f\|_{\gamma;\bar{\mfK}}  \mfr^{-\underline{\alpha}n},
\end{equation}
where $\underline{\alpha}= \min A$. It follows that $\mcR^{(m,n)}f$ converges uniformly on compacts as $m\rightarrow \infty$ to $\mcR^{(n)}f$ which also satisfies the bound\eqref{eq:Approx_Reconstruct_Bnd}. Since it also holds that for every $m\geq n+1$,
\begin{equation*}
	\mcR^{(m,n)}f = \mcR^{(m,m)}f \ast \tilde{\rho}^{(m-1,n)} = \mcR^{(m,m)} f \ast \tilde{\rho}^{(m-1,n+1)}\ast \tilde{\rho}^{(n)}  = 	\mcR^{(m,n+1)}f \ast\tilde{\rho}^{(n)}\ , 
\end{equation*}
we find that
$$\mcR^{(n)}f= \mcR^{(n+1)}f \ast\tilde{\rho}^{(n)} \ .$$
Therefore we may apply Lemma~\ref{lem:Pre_Reconstruct_Lemma} to see that there exists a limit \linebreak $\mcR f:=\lim_{n\rightarrow \infty}\mcR^{(n)}f $.

With validity of the limit established we now turn to show the bounds \eqref{eq:Reconstruction_Regularity} and \eqref{eq:rec2}; this requires us to keep more careful track of the underlying sets in the proof. We begin with \eqref{eq:Reconstruction_Regularity}, first noting that if we define $f_x(y) := \Gamma_{y,x}f(x)$ then one has $\mcR^{(n,n)}f_x = \Pi_x f(x)\ast \tilde{\varphi}^{(n)}$ so that
\eqref{eq:Reconstruction_Regularity} can be written as the claim that for all $\lambda \in (0,1]$,
\begin{equation}\label{eq:Reconstruct_Post_Bnd_2}
	\sup_{\psi \in \mfB_{r}} | \langle \mcR(f-f_x),\psi^\lambda_x\rangle|\lesssim_{\mfK} 
	\lambda^{\gamma}\|f\|_{\gamma;B_{2\lambda}(x)}\|\Pi\|_{\gamma;B_{2\lambda}(x)} \ .
\end{equation}
Using that $|(f-f_x)(z)|_\alpha\lesssim\|f\|_{\gamma;B_{{|z^{-1}x|}}(x)} {|z^{-1}x|}^{\gamma-\alpha}$ for $x,\,z \in \mfK$, it follows that for all $y\in B_\lambda(x)$ one has
\begin{align}
	|(\mcR^{(n,n)}(f-f_x)(y)| = |\langle \Pi_y(f(y)-f_x(y)),\varphi^{(n)}_y\rangle| \notag &= |\langle \Pi_y(f(y)-\Gamma_{y,x}f(x)),\varphi^{(n)}_y\rangle|\\ \notag 
	&\lesssim \|\Pi\|_{\gamma;B_\lambda(y)} \|f\|_{\gamma;B_{C\mfr^{-n}}(y)} \sum_{\underline{\alpha}\leq\alpha\leq \gamma} \mfr^{-\alpha n}|y^{-1}x|^{\gamma-\alpha} \notag \\
	&\lesssim\|\Pi\|_{\gamma;B_\lambda(y)} \|f\|_{\gamma;B_{C\mfr^{-n}}(y)} \mfr^{-\underline{\alpha}n}({|y^{-1}x|}^{\gamma-\underline{\alpha}}+ \mfr^{(\underline{\alpha}-\gamma)n}), \label{eq:ApproxReconstruct_Bnd1}
\end{align}
where $C(\mfr):= \frac{\mfr}{\mfr-1}$ is as in Lemma 
\ref{lem:convergence_of_test_functions}. Given $n>n_0(\lambda)$ sufficiently larger, we have a uniform bound $C\mfr^{-n}\leq \lambda$. By the convergence of $\mcR^{(n,n)}$ in the first exponent, it follows that $\mcR^{(n)}$ satisfies the same bound and so inspecting the proof of the second half of Lemma~\ref{lem:Pre_Reconstruct_Lemma}, in particular noticing that we integrate the above estimate over $y\in B_\lambda(x)$, we conclude that \eqref{eq:Reconstruct_Post_Bnd_2} holds for the limit $\mcR$. 

Using the obvious notation we can also rewrite \eqref{eq:rec2} as
\begin{align*}
	\sup_{\psi \in \mfB_{r}} | \langle \mcR(f-f_x)- \bar{\mcR}(\bar{f}-\bar{f}_x) ,\psi^\lambda_x\rangle|\lesssim \lambda^\gamma &\big(\|f;\bar{f}\|_{\gamma;B_{2\lambda}(x)}\|\bar{\Pi}\|_{\gamma;B_{2\lambda}(x)} \\
	&+\|f\|_{\gamma;B_{2\lambda}(x)}\|\Pi-\bar{\Pi}\|_{\gamma;B_{2\lambda}(x)} \big)\ ,
\end{align*}
uniformly over $\lambda \in (0,1]$. This is seen very similarly to \eqref{eq:Reconstruct_Post_Bnd_2} but using this time that for $n$ large enough,
\begin{align*}
	&|\mcR^{(n,n)}(f-f_x)- \bar{\mcR}^{(n,n)}(\bar{f}-\bar{f}_x)(y)| \\
	&= |\langle \Pi_y(f(y)-\Gamma_{x,y}f(x))-\bar{\Pi}_y(\bar{f}(y)-\bar{\Gamma}_{y,x}\bar{f}(x)),\varphi^{(n)}_y\rangle|\\
	&= |\langle 
	\Pi_y(f(y)-\Gamma_{x,y}f(x)-\bar{f}(y)+\bar{\Gamma}_{y,x}\bar{f}(x))
	+
	(\Pi_y-\bar{\Pi}_y)(\bar{f}(y)-\bar{\Gamma}_{y,x}\bar{f}(x))
	,\varphi^{(n)}_y\rangle|\\
	&\lesssim \big(\|\Pi\|_{\gamma;B_{\lambda}(y)} \|f;\bar f\|_{\gamma;B_{\lambda}(y)}+\|\Pi-\bar{\Pi}\|_{\gamma;B_{\lambda}(y)} \|\bar f\|_{\gamma;B_{\lambda}(y)}\big)\sum_{\underline{\alpha}\leq\alpha\leq \gamma} \mfr^{-\alpha n}{|y^{-1}x|}^{\gamma-\alpha} \notag \ .
\end{align*}
It remains to show that the reconstruction map is unique for $\gamma>0$ and that $\mathcal{R}f$ is indeed an element of $C^{\underline{\alpha}}$. This is done exactly as in \cite[Sec.~3]{hairer_14_RegStruct}.
\end{proof}

\subsubsection{Functions as Modelled Distributions}
Let $\bar{\mcT}=(\bar{\opT},\bar{\opG})$ be the polynomial regularity structure equipped with the polynomial model. We show that in this setting the reconstruction theorem and its inverse map Taylor polynomials to H\"older functions and vice versa.
\begin{theorem}\label{th:hoelder_characterisation}
For $\gamma>0$ the reconstruction operator is an isomorphism between $\msD^\gamma (\mbG)$ and $\mcC^\gamma(\mbG)$. 
In particular the inverse of the reconstruction map is given by
\begin{equation*}\label{eq:inverse_of_reconstruction_map}
	\mcC^{\gamma}(\mbG) \ni f(\cdot) \mapsto \opbfP^{\gamma}_{(\,\cdot\,)}[f] \in \msD^\gamma(\mbG),
\end{equation*}
where $\opbfP^a$ is defined in Definition~\ref{def:Taylor_as_Modelled_Dist} 
\end{theorem}
\begin{proof}\,
Given a modelled distribution in $\msD^\gamma$ and since we are working with the polynomial regularity structure equipped with its canonical model, it follows directly from the bound satisfied by the image of the reconstruction operator, i.e. Equation\eqref{eq:Reconstruction_Regularity}, and the definition of $\mcC^\gamma$ in \eqref{eq:Positive_Holder_Def} that the reconstruction operator is a map $\msD^\gamma (\mbG) \to \mcC^{\gamma}(\mbG)$. Continuity also follows from \eqref{eq:Reconstruction_Regularity} and by linearity.

To see the other direction recall that given a H\"older continuous distribution $f\in \mcC^\gamma(\mbG)$ by Proposition \ref{prop:holder_inclusion_compact} the $\{X^I f\}_{d(I)<\gamma}$ are actually functions and, in particular, that $\tilde{P}_x= \Pi_x\opbfP^{[\gamma]}_{x}[f]$. Therefore, for $\lambda=|x^{-1}y|_\mbG$ 
\begin{equation}\label{eq:some_local_equation}
	\left| \Pi_x\big(\opbfP^{[\gamma]}_{x}[f] -\Gamma_{x,y}\opbfP^{[\gamma]}_{y}[f] \big) (\psi^\lambda_x) \right|
	= |\langle \tilde{P}_x -f,\psi^\lambda_x\rangle| +|\langle f-\tilde{P}_y, \psi^\lambda_x\rangle| \lesssim \lambda^\gamma \
\end{equation}
uniformly in $\psi\in \mfB_{r}$.
On the other hand, writing $\opbfP^{[\gamma]}_{x}[f] -\Gamma_{x,y}\opbfP^{[\gamma]}_{y}[f] = \sum_{d(I)<\gamma} c^I_{x,y} \bfeta^I$, we find that 
$$\Pi_x\big(\opbfP^{[\gamma]}_{x}[f] -\Gamma_{x,y}\opbfP^{[\gamma]}_{y}[f] \big) (\psi^\lambda_x)= \sum_{d(I)<\gamma}  c^I_{x,y} \Pi_e\bfeta^I (\psi^\lambda)= \sum_{d(I)<\gamma}  c^I_{x,y}\lambda^{d(I)} \Pi_e\bfeta^I (\psi) \ .$$
Using that $\mcP_\gamma$ is a finite dimensional vector space and the surjectivity of the linear map
$$C_c^\infty(B_1(e)) \to \mathbb{R}^{\dim \mcP_\gamma}, \quad \psi \mapsto \{ \Pi_e\bfeta^I (\psi) \}_{d(I)<\gamma}$$
we find that
\begin{equation}\label{eq:some_local_equation2}
	\sum_{d(I)<\gamma}  |c^I_{x,y}|\lambda^{d(I)}
	\lesssim_\gamma \sup_{\psi\in \mfB_{r}}\left|\sum_{d(I)<\gamma}  c^I_{x,y}\lambda^{d(I)} \Pi_e\bfeta^I (\psi)\right|\ .
\end{equation}
Together, \eqref{eq:some_local_equation} and \eqref{eq:some_local_equation2} imply that $|c^I_{x,y}|\lesssim \lambda^{\gamma- d(I)}$, we may then apply \linebreak Lemma~\ref{lem:to_bound_polynomials} to see that $\opbfP^{[\gamma]}_{(\,\cdot\,)}[f] \in \msD^\gamma $ .
\end{proof}
\subsubsection{Local Reconstruction}

We will require the following further refinement of the reconstruction theorem which is an analogue in our case of \cite[Prop.~ 7.2]{hairer_14_RegStruct}.
\begin{prop}\label{prop:support_localised_reconstruction}
In the setting of Theorem~\ref{th:Recontruction} one has the improved bound,
\begin{equation}\label{eq:support_localised_reconstruction}
	\sup_{\psi\in \mfB_m}|(\mcR f - \Pi_x f(x))(\psi^\lambda_x)| \lesssim \lambda^{\gamma}  \|\Pi\|_{\gamma;B_{2\mu\lambda}(x)} \sup_{y,\,z \in \supp (\psi^\lambda_x)} \sup_{\ell <\gamma} \frac{|f(z)-\Gamma_{zy}f(y)|_\ell}{{|y^{-1}z|}^{\gamma-\ell}}\ ,
\end{equation}
as well as, given a second model $\bar{M}(\bar{\Pi},\bar{\Gamma})$ and a modelled distribution $\bar{f}\in \msD^\gamma_{\bar{M}}$, the analogous bound, that for every $\lambda \in (0,1]$,
\begin{equation}\label{eq:local_reconstruct_diff}
	\begin{aligned}
		\sup_{\psi \in \mfB_{r}} |\langle \mathcal{R}_M f-\mathcal{R}_{\bar{M}}\bar{f} -\Pi_x f(x)+\bar{\Pi}_x \bar{f}(x),\psi^{\lambda}_x\rangle|
		\lesssim& \lambda^\gamma \bigg(\|f;\bar{f}\|_{\gamma;\supp (\psi^\lambda_x)}\|\bar{\Pi}\|_{\gamma;B_{2\mu\lambda}(x)}\\
		&+\|f\|_{\gamma;\supp (\psi^\lambda_x)}\|\Pi-\bar{\Pi}\|_{\gamma;B_{2\mu\lambda}(x)} \bigg)\ ,			
	\end{aligned}
\end{equation}
for any $x\in \mbG$ and any $\lambda \in (0,1]$.
\end{prop}
\begin{proof}\,
Since the right hand side of \eqref{eq:support_localised_reconstruction} is linear in $f$, as in \cite{hairer_14_RegStruct} we may assume it to be equal to $1$. 
We use the functions $\rho^{(n)}$ and $\varphi^{(n)}$ from Section~\ref{subsec:Reconstruction} and recall that they satisfy 
\begin{equation*}
	\varphi^{(n-1)}(x) = \int_{B_{\mfr^{-n+1}}}  \rho^{(n-1)}(y) \varphi_y^{(n)}(x) dy, \quad \varphi^{(n-1)}_y= \int_{B_{\mfr^{-n+1}}}  \rho^{(n-1)}(w) \varphi_{yw}^{(n)}(x) dw
\end{equation*}
in particular since $\int \rho^{(n)}(x)\dd x = 1$ we have $\int \varphi_y^{(n)} (\,\cdot\,)\dd y =1$ for all $n\geq 0$. 
Define, for fixed 
$\psi^\lambda_x$, the sets
\begin{equation*}
	\Lambda_n= \left\{ y\in \mbG \ : \ \supp (\varphi^{(n)}_y)\cap \supp (\psi^\lambda_x) \neq \emptyset \right\}\subset \mbG
\end{equation*}
which by definition is contained in $B_{\lambda+ C\mfr^{-n}}(x)$ with $C=\frac{\mfr}{\mfr{-1}}$ as in Lemma~\ref{lem:convergence_of_test_functions}, as well as a (measurable) function
\begin{equation*}
	\pi_n: \Lambda_n\to \supp (\varphi^{(n)}_y)\cap \supp (\psi^\lambda_x).
\end{equation*}

Next, let 
\begin{align*}
	R_n &= \int_{\Lambda_n} \big( \mathcal{R}f- \Pi_{\pi_n(y)}f(\pi_n(y))\big) (\psi_x^\lambda \varphi^n_y) dy\\
	&= \langle\mcR f,\psi_x^\lambda\rangle -    \int_{\Lambda_n}\Pi_{\pi_n(y)}f(\pi_n(y)) (\psi_x^\lambda \varphi^n_y) dy \ .
\end{align*}
It follows that for $n_0= \min \{n\in \mathbb{ N} \ : \ \mfr^{-n}\leq \lambda \}$, one has 
\begin{equation}\label{eq:local_rec_large_scale}
	\left|(\mcR f - \Pi_x f(x))(\psi^\lambda_x)- R_n\right| = \left| \int_{\Lambda_{n}}\big(\Pi_x f(x)-  \Pi_{\pi_n(y)} f (\pi_n(y))\big) (\psi_x^\lambda \varphi^n_y) dy\right| \lesssim \lambda^\gamma
\end{equation}
as well as for $n >  n_0$
\begin{align*}
	R_{n-1}- R_{n}
	=&  \int_{\Lambda_{n-1}}\Pi_{\pi_{n-1}(y)}f(\pi_{n-1}(y)) (\psi_x^\lambda \varphi^{n-1}_y) dy -
	\int_{\Lambda_n}\Pi_{\pi_n(z)}f(\pi_n(z)) (\psi_x^\lambda \varphi^n_z) dz\\
	=&  \int_{B_{\mfr^{-n+1}}} \rho^{(n-1)}(w)\int_{\Lambda_{n}}\Big( \Pi_{\pi_{n-1}(zw^{-1})}f(\pi_{n-1}(zw^{-1}))
	-\Pi_{\pi_n(z)}f(\pi_n(z))\Big) (\psi_x^\lambda \varphi^n_z) dzdw \ .
\end{align*}
Since $|\pi_n(z)^{-1}\pi_{n-1}(zw^{-1})|\leq \bar{C}\mfr^{-n}$ and for $\tau\in \opT_\alpha$ such that $|\tau| \leq 1$ 
\begin{equation*}
	|\Pi_{\pi_n(z)} \tau (\psi_x^\lambda \varphi^n_z)| \lesssim \lambda^{|\fraks|} \mfr^{-\alpha n}\ ,
\end{equation*}
we find that 
\begin{equation*}
	\int_{\Lambda_{n}} \Big| \big(\Pi_{\pi_{n-1}(zw^{-1})}f(\pi_{n-1}(zw^{-1}))
	-\Pi_{\pi_n(z)}f(\pi_n(z))\big) (\psi_x^\lambda \varphi^n_z) \Big| dz\lesssim \mfr^{-\gamma n}
\end{equation*}
and thus conclude
\begin{equation*}
	|R_n-R_{n-1}| \lesssim \mfr^{-\gamma n}.
\end{equation*}
A similar argument gives $|R_n|\to 0$ as $n \to \infty$ which combined with \eqref{eq:local_rec_large_scale} concludes the proof of \eqref{eq:support_localised_reconstruction}. The proof of the analogous bound, \eqref{eq:local_reconstruct_diff}, follows in a similar manner.
\end{proof}	
\subsection{Singular Modelled Distributions}\label{subsec:singular_modelled}
As in \cite{hairer_14_RegStruct} we will eventually be concerned with solutions to SPDEs that take values in spaces of modelled distributions with permissible singularities in some regions of the domain. Our main example will be modelled distributions on space-time domains that are allowed to be discontinuous at $\{t=0\}$, see Section~\ref{sec:evolution}. However, as in \cite{hairer_14_RegStruct} we build the notion of singular modelled distributions allowing for singularities on more general sets, generalisations of which have been used in \cite{gerencser_hairer_19_domains, gerencser_hairer_21_boundary_renorm} to study singular equations with boundary conditions. 

We fix a homogeneous sub-Lie group $P\subset \mbG$ with associated Lie algebra for which we write $\mfp \subset \mfg$. The assumption that $P$ be a homogeneous sub-Lie group means that the scaling map $\mfs$ restricts to a map $\bar{\mfs}:=\mfs|_{\mfp}:\mfp \rightarrow \mfp$ which is diagonalisable. We fix a decomposition $\mfg = \mfp \oplus \mfp^c$ such that $\mfp^c$ is also invariant under $\mfs$ and define the homogeneous dimension of $P$ and its complement, $P^c := \exp(\mfp^c)$ as, 
\begin{equation}\label{eq:hyperplane_dim}
|\bar{\mfs}|:= \text{trace} (
\mathfrak{s}|_{\mathfrak{p}}) \quad\text{and}\quad  |\mfm| = \text{trace}(\mfs|_{\mfp^c})\ .
\end{equation}
Furthermore, we set 
\begin{equation*}
|x|_P := 1 \wedge d_{\mbG}(x,P) = 1\wedge \inf\left\{ z \in P\,:\, {|x^{-1}z|} \right\}, \quad |x,y|_P := |x|_P \wedge |y|_P.
\end{equation*}
Note that since $P$ is closed (being the image under $\exp$ of a linear subspace of $\mfg$), one sees easily that the
infimum above is actually a minimum. Given $\mfK \subset \mbG$ we define the set
\begin{equation*}
\mfK_P := \left\{ (x,y) \in (\mfK\setminus P)^2\,:\, x\neq y \, \text{ and }\, {|x^{-1}y|} \leq |x,y|_{P}\right\}.
\end{equation*}
That is $\mfK_P$ contains all the points in $\mfK$ that are closer to each other than they are to $P$.
\begin{definition}[Singular Modelled Distributions]\label{def:sing_modelled_dist}
Given a regularity structure $\mcT$ and a subgroup $P$ as above, for any $\gamma>0$, $\eta \in \mbR$ and maps $f:\mbG\setminus P \rightarrow \mcT$, we set 
\begin{equation*}\label{eqtwonorms}
	\|f\|_{\gamma,\eta;\mfK} := \sup_{x \in \mfK\setminus P} \sup_{\zeta<\gamma} \frac{|f(x)|_\zeta}{|x|_P^{(\eta-\zeta)\wedge 0}},\quad
	\llbracket f\rrbracket_{\gamma,\eta;\mfK} := \sup_{x \in \mfK\setminus P} \sup_{\zeta<\gamma} \frac{|f(x)|_\zeta}{|x|^{\eta-\zeta}_P}.
\end{equation*}
Then given a model $M=(\Pi,\Gamma)$ as well as a sector $V$, the space $\msD^{\gamma,\eta}_{P,M}(V)$ consists of all functions $f: \mbG\setminus P \rightarrow V_{< \gamma}$ such that for every compact set $\mfK \subset \mbG$,  
\begin{equation*}
	\vertiii{f}_{\gamma,\eta;\mfK} := \|f\|_{\gamma,\eta;\mfK} + \sup_{(x,y)\in \mfK_P} \sup_{\zeta<\gamma} \frac{|f(x)-\Gamma_{x,y}f(y)|_\zeta}{{|y^{-1}x|}^{\gamma-\zeta} |x,y|^{\eta-\gamma}_P} <+\infty.
\end{equation*}
For two models $M=(\Pi,\Gamma),\, \bar{M}=(\bar{\Pi},\bar{\Gamma})$ and two modelled distributions $f \in  \msD^{\gamma,\eta}_{P,M},\,\bar{f} \in \msD^{\gamma,\eta}_{P,\bar{M}}$ we also set,
\begin{equation*}
	\vertiii{f;\bar{f}}_{\gamma,\eta;\mfK}  := \|f-\bar{f}\|_{\gamma,\eta;\mfK} + \sup_{(x,y) \in \mfK_P} \sup_{\zeta <\gamma} \frac{|f(x)-\bar{f}(x) - \Gamma_{x,y} f(y) + \bar{\Gamma}_{xy}\bar{f}(y)|_\zeta}{{|y^{-1}x|}^{\gamma-\zeta}|x,y|^{\eta-\gamma}_{P}}.
\end{equation*}
If $V$ is a sector of regularity $\alpha \in A$, where appropriate we will use the shorthand $\msD^{\gamma,\eta}_{\alpha;P,M} = \msD^{\gamma,\eta}_{P,M}(V)$ and we will drop the dependence on the model when the context is clear.
\end{definition}
\begin{remark}\label{rem:sing_modelled_is_modelled}
{We refer to \cite{hairer_14_RegStruct} for more intuition regarding the definition of these spaces and their properties - all of which carry over to our setting. In particular \cite[Rem. 6.4]{hairer_14_RegStruct} discusses the relationship between the spaces $\msD^{\gamma,\eta}_P$ and $\msD^\gamma$.}
\end{remark}
\begin{remark}
The family of norms $\llbracket f\rrbracket_{\gamma,\eta;\mfK}$ and the two following lemmas play a role when we consider fixed-point maps in Section \ref{sec:evolution} below. Their utility is in allowing us to extract small scaling parameters in terms of the distance to the subgroup. In the semi-linear evolution equation setting this allows us to obtain fixed points on sufficiently short time intervals.
\end{remark}
\begin{lemma}\label{lem:SingularNorm_Ordering}
Let $\mfK\subset \mbG$ be a compact domain such that for every $x\in \mfK$ and $\bar{x}:= \argmin_{y\in P}|x^{-1}y|$ one has that the points $\bar{x}\big( \lambda\cdot (\bar{x}^{-1} x)\big) \in \mfK$ for every $\lambda\in [0,1]$. Also let $f \in \msD^{\gamma,\eta}_P$ for some $\gamma>0$ and assume that for every $\zeta<\eta$ the map $x\mapsto \mcQ_\zeta f(x)$ extends continuously in such a way that for $x\ \in P$ one has $\mcQ_\zeta f(x) =0$. Then, one has the bound, 
\begin{equation*}
	\llbracket f\rrbracket_{\gamma,\eta;\mfK} \lesssim \vertiii{f}_{\gamma,\eta;\mfK},
\end{equation*}
where the implied constant depends affinely on $\|\Gamma\|_{\gamma;\mfK}$ but is otherwise independent of $\mfK$. Similarly, if $\bar{f}\in \msD^{\gamma,\eta}_{P,\bar{M}}$ with respect to a different model $\bar{M}=(\bar{\Pi},\bar{\Gamma})$ and is such that 
$$\lim_{x\to P}\mcQ_\zeta(f(x)-\bar{f}(x)) =0$$
for every $\zeta<\eta$, then one has the bound
\begin{equation*}
	\llbracket f -\bar{f} \rrbracket_{\gamma,\eta;\mfK} \lesssim \vertiii{f;\bar{f}}_{\gamma,\eta;\mfK} + \|\Gamma-\bar{\Gamma}\|_{\gamma;\mfK}\left(\vertiii{f}_{\gamma,\eta;\mfK}+ \vertiii{\bar{f}}_{\gamma,\eta;\mfK}\right),
\end{equation*}
with proportionality constant also depending affinely on $\|\Gamma\|_{\gamma;\mfK}$ and $\|\bar{\Gamma}\|_{\gamma;\mfK}$.
\end{lemma}
\begin{proof}\,
The proof follows almost exactly as that of \cite[Lem. 6.5]{hairer_14_RegStruct} but here using the sequence $x_0:=x,\, x_\infty := \bar{x}= \argmin_{z\in P} {|x^{-1}z|}$ and $x_n:=  {x_\infty\big( 2^{-n} \cdot (x^{-1}_\infty x_0 )\big)}$. 
\end{proof}
\begin{lemma}
Let $\gamma >0$, $\kappa \in (0,1)$ and assume $f,\,\bar{f}$ satisfy the assumptions of Lemma~\ref{lem:SingularNorm_Ordering}. Then, for every compact $\mfK \subset \mbG$, one has
\begin{equation*}
	\vertiii{f;\bar{f}}_{(1-\kappa)\gamma,\eta;\mfK} \lesssim \llbracket f-\bar{f}\rrbracket^\kappa_{\gamma,\eta;\mfK}\left(\vertiii{f}_{\gamma,\eta;\mfK} + \vertiii{\bar{f}}_{\gamma,\eta;\mfK}\right)^{1-\kappa}.
\end{equation*}
\end{lemma}
\begin{proof}\,
Follows by a direct adaptation of the proof of \cite[6.6]{hairer_14_RegStruct}. One need only replace $\mbR^d$ there by $\mbG$ here.
\end{proof}
\subsubsection{Reconstruction Theorem for Singular Modelled Distributions}
Since the reconstruction is purely local, it follows from our earlier proof that for any singular modelled distribution,  $f\in \msD^{\gamma,\eta}_P$, there exists a unique element $\tilde{\mcR}f\in \mcS'(\mbG\setminus P)$, i.e. in the dual of smooth functions compactly supported away from $P$, such that,
\begin{equation*}
\langle \tilde{\mcR}f - \Pi_x f(x), \psi^\lambda_x\rangle \lesssim \lambda^{\gamma},
\end{equation*}
for all $x\notin P$ and $\lambda \ll d_{\mbG}(x,P)$. However, we show below that under appropriate assumptions there exists a natural extension of $\tilde{\mcR}f$ to an actual distribution on $\mbG$ with regularity $\mcC^{\underline{\alpha}}$. 	%
\begin{prop}[Singular Reconstruction]\label{prop:singular_reconstruction_v2} Let $f\in \msD^{\gamma,\eta}_P(V)$, $\gamma>0$ and $\eta\leq \gamma$. Then, provided $\alpha \wedge \eta >-|\mfm|$, where $|\mfm|$ is the scaled dimension of the complement of the singular hyperplane defined by \eqref{eq:hyperplane_dim}, there exists a unique distribution $\mcR f \in \mcC^{{\alpha} \wedge \eta}_\mfs$ such that $(\mcR f)(\varphi) = (\tilde{\mcR}f)(\varphi)$ for any smooth test function compactly supported away from $P$. If $f$ and $\bar{f}$ are given with respect to two models $M$ and $\bar{M}$ then, for any compact $\mfK$, it holds that
\begin{equation*}
	\|\mcR_M f- \mcR_{\bar{M}} \bar{f}\|_{\mcC^{{\alpha} \wedge \eta}(\mfK)} \lesssim \vertiii{f;\bar{f}}_{\gamma,\eta;\bar{\mfK}} + \vertiii{M;\bar{M}}_{\gamma;\bar{\mfK}},
\end{equation*} 
where the constant depends on semi-norms of $f,\,\bar{f}$ and $M,\,\bar{M}$ on $\bar{\mfK}$.
\end{prop}
We provide the proof of Proposition~\ref{prop:singular_reconstruction_v2} at the end of this section, let us first make some preparatory observations. Recalling the decomposition $\mfg=\mfp \oplus \mfp^c$ defined at the start of the subsection, we define 
the projections $\pi_c: \mfg\to  \mfp^c$ and $\pi_{\mfp}: \mfg\to  \mfp$. Then using the decomposition $X = X^{\mfp}+ X^c \in \mfp \oplus \mfp^c$, we define the map
\begin{equation}\label{eq:defining_tilde_Phi}
\tilde{\Phi}:\mfg \rightarrow \mbG, \quad	\tilde{\Phi}(X) = \exp(X^{\mfp})\exp(X^c).
\end{equation}
Similarly to Remark \ref{rem:decompose_exp_map} one sees that $\tilde{\Phi}$ is a global diffeomorphism.
We then define the map
\begin{equation}\label{eq:Np_def}
N_P: \mbG\to \mathbb{R}_+, \qquad x\mapsto \left|\exp\left(\pi_c\circ \tilde{\Phi}^{-1} (x)\right)\right| 
\end{equation}
and observe the following properties.
\begin{itemize}
\item For $x\in \mbG$ one has that $N_P(x)=0$ if and only if $x\in P\ .$ This follows from the fact that $P= \exp(\mfp)$.
\item For $x\in \mbG$ and $\delta>0$ one has the identity 
\begin{equation}\label{eq:dilation_scaling_of_N}
	N_P (\delta \cdot x)= \delta N_P(x)\ .
\end{equation}
Indeed, writing $x= \tilde{\Phi} (\tilde{X})$, we have $$N_P (\delta \cdot x)=|\exp( \pi_c(\dil_\delta \tilde{X}))|=|\exp(\dil_\delta \pi _c( \tilde{X}))|= |\delta\cdot \exp( \pi_c( \tilde{X}))|=\delta N_P(x) \ .$$
\item For any $x\in \mbG$ and $y\in P$ one has  
\begin{equation}\label{eq:translation_invariance_of_N}
	N_P (y x)= N_ P(x)\ .
\end{equation}
This follows from the observation that, writing $x= \tilde{\Phi}(\tilde{X})$ and $y=\tilde{\Phi}(Y)= \exp(Y)$ one finds that $yx=\exp(\opH (Y,\tilde{X}^\mfp))\exp(\tilde{X}^{c})$ where $\opH$ was defined in \eqref{eq:BCH}. Identity \eqref{eq:translation_invariance_of_N} then follows from the fact that $\opH (Y,\tilde{X}^\mfp)  \in \mfp$.
\item \hspace{0.2em}The map $N_P$ is Lipschitz continuous on $\mbG$ and it follows from Remark~{\ref{rem:triangle_inequality}} that $N_P$ is smooth on ${ \mbG\setminus P}$. 
\item \hspace{0.2em}There exists a constant $C>0$ such that for all $x\in \mbG$
\begin{equation}\label{eq:N_dist_compare_1}
	C N_P(x)\leq d_{\mbG}(x,P) \ .
\end{equation}
In a neighbourhood of the origin this follows directly from the fact that $N_P$ is Lipschitz continuous. Homogeneity of $N_p$ (given by \eqref{eq:dilation_scaling_of_N} above) and of the distance function $x\mapsto d_{\mbG}(x,P)$ then shows that this constant is in fact uniform on all of $\mbG$.
\item \hspace{0.2em} For all $x\in \mbG$
\begin{equation}\label{eq:N_dist_compare_2}
	d_{\mbG}(x,P)\leq N_P(x) \ .
\end{equation}
Indeed, write $x= \tilde{\Phi}({X})= \exp(X^{\mfp})\exp(X^c)$ and note that 
$$d(x,P)\leq |\exp(X^{\mfp})^{-1} x|= |\exp(X^c)|= N_P(x)\ .$$
\end{itemize}

\begin{proof}[Proof of Proposition~\ref{prop:singular_reconstruction_v2}]\,
The proof is analogous to that of \cite[Lem.6.9]{hairer_14_RegStruct}, the only element that does not adapt ad verbatim, is the construction of the partition of unity $\varphi_{x,n}$. We therefore present an alternative construction of a partition of unity on $\mbG$, which satisfies all the required conditions.

First let $\varphi:\mbR_+\rightarrow [0,1]$ be a smooth compactly, supported function such that $\supp(\varphi)=[1/2,2]$ and with the property that for all $r\in \mbR_+$,
\begin{equation*}
	\sum_{n\in \mbZ} \varphi(2^nr) =1.
\end{equation*}
Secondly, let $\mfZ\subset P$ be a lattice (see Section~\ref{sec:discrete_subgroups}) and let $\tilde{\varphi}$ be smooth, compactly supported, such that
\begin{equation}\label{eq:substep_for_partition_of_unity}
	\sum_{y\in\mfZ}  \tilde{\varphi}_y(x) =1
\end{equation}
for all $x$ in the $\frac{2}{C}$-fattening of $P\subset \mbG$, where $C$ is the constant in \eqref{eq:N_dist_compare_1}and \eqref{eq:N_dist_compare_2}. 
For $y\in \mfZ$ we then set
$$\phi_{y} (x) := \varphi(C N_P(x)) \tilde{\varphi}_y(x) \ , $$
where the constant $C$ is as in \eqref{eq:N_dist_compare_1}. To conclude, we then set for every $n \geq 0$, $y\in \mfZ$ and $x \in \mbG$,
$$\phi_{n,y} (x) :=\phi_{y} (2^n \cdot x).$$
By using \eqref{eq:dilation_scaling_of_N} and \eqref{eq:translation_invariance_of_N}, it directly follows that $\phi_{n,y}(x) =(\phi_{0,e}) \left( y^{-1} (2^n\cdot x) \right)$. Since $\phi_{1,e}$ has compact support and is such that for all $x\in \supp(\phi_{1,e})$, one has $ d_{\mbG}(x,P)\geq C N_P(x)\geq 1/2$, it only remains to check that $\{\phi_{n,y}\}_{n \in \mbZ, y\in \mfZ}$ is in fact a partition of unity. Indeed let $x\in \mbG$, then

$$\sum_{n\in \mbZ, y\in \mfZ} \phi_{n,y} (x) = \sum_{n \in \mbZ, y\in \mfZ}  \varphi(CN_P(2^n\cdot x)) \tilde{\varphi}_y(2^n\cdot  x) = \sum_{n\in \mbZ} \varphi(2^n C N_P(x))  {\sum_{y\in \mfZ}{\tilde{\varphi}_y(2^n \cdot  x)}}= 1 \ ,$$
where in the second last equality we used that 
$
\sum_{y\in \mfZ}\tilde{\varphi}_y(2^n \cdot  x)=1
$
whenever $d(x, P)\leq \frac{1}{C}2^{-n+1}$ and in the last equality we used \eqref{eq:substep_for_partition_of_unity}. The remainder of the proof of \cite[Lem. 6.9]{hairer_14_RegStruct} then adapts ad verbatim by also also making use of Proposition~\ref{prop:support_localised_reconstruction}.
\end{proof}
\begin{remark}\label{rem:reconstruction_smooth_models}
If the model $M$ is smooth, i.e. $\Pi_x\tau \in C^\infty(\mbG)$ for every $\tau \in \opT$, one finds exactly as in \cite[Rem.~3.15]{hairer_14_RegStruct} that for any modelled distribution $f\in \msD^{\gamma}$ with $\gamma>0$ one has the identity 
$\mathcal{R}_Mf(x) = \left(\Pi_x f(x)\right)(x)$ (and in particular $\mathcal{R}_Mf$ is a continuous function).
\end{remark}	
\subsection{Convolution with Singular Kernels}\label{subsec:schauder}
In this section we describe how to lift the action of singular kernels onto the regularity structure. While most arguments adapt from \cite{hairer_14_RegStruct} some care has to be taken due to the fact that convolutions are not commutative and we do not have an explicit formula for Taylor expansions. This latter issue is circumvented using Lemma~\ref{lem:to_bound_polynomials} which allows us to reduce our analysis to similar expressions as appear in \cite{hairer_14_RegStruct}. The examples we have in mind are the singular part of Greens functions of left invariant differential operators satisfying the following assumption. 
\begin{assumption}\label{ass:K0_decompose}
For $\beta \in (0, |\fraks|)$, the kernel $K:\mbG\setminus \{e\}\to \mathbb{R}$ can be decomposed as 
\begin{equation}\label{eq:local_kernel_decompose}
	K(x)=\sum_{n\in\mathbb{N}} K_n(x)
\end{equation}
where the smooth functions $K_n$ are supported on $B_{2^{-n}}$ and 
\begin{itemize}
	\item for each $I\in \mbN^d$ there exists a constant $C(I)>0$, uniform in $n\in \mbN$ such that
	\begin{equation*}
		\sup_{x \in \mbG}	|X^I K_n(x)|\leq C(I) 2^{(|\fraks|-\beta+d(I))n},
	\end{equation*}
	\item for any multi-indices $I,\,J\in \mbN^d$ there exists a constant $C(I,J)>0$ uniform in $n\in \mbN$ such that,
	\begin{equation*}
		\left|\,\int_{\mbG} \eta^I(x) X^J K(x) dx\,\right|\leq C(I,J) 2^{-\beta n},
	\end{equation*}
	\item there exists an integer $r$, such that
	\begin{equation*}
		\int_{\mbG} \eta^I(x) K(x) \, dx=0
	\end{equation*}
	for all multi-indices $I\in \mathbb{N}^d$ with scaled degree $d(I)\leq r$.
\end{itemize}
\end{assumption}
\begin{remark}
We note that all of the analysis in the remainder of this section also applies to kernels of the form $K:(\mbG\setminus \{e\})^2\rightarrow \mbR$ satisfying an analogue of \cite[Ass.~5.1 \& Ass.~5.4]{hairer_14_RegStruct} adapted to our setting. Although Assumption \ref{ass:K0_decompose} is somewhat less general we choose to work with it for two reasons; firstly it is simpler to verify and secondly it highlights the role that translation invariance plays in our applications. Lemma~\ref{lem:kernel_decompose} which is an amalgam of \cite[Lem.~5.5 \& Lem.~7.7]{hairer_14_RegStruct} in our setting, shows that fundamental solutions of left-translation invariant, homogeneous linear operators can always be decomposed into a compactly support part satisfying Assumption \ref{ass:K0_decompose} and a sufficiently well-behaved remainder.
\end{remark}
\begin{remark}
Although we work explicitly with the one-parameter kernels of Assumption \ref{ass:K0_decompose} it will sometimes be convenient in the proofs below to define $K(x,y):= K(y^{-1}x)$ and use the notation $K(f)(x):= \int K(x,y)f(y)\dd y = f\ast K(x)$. Note that under this convention, for any left-translation invariant vector field $X$, one has $(XK)(f)=X(Kf)$. 
\end{remark}
From now on we shall exclusively work with regularity structures $\mcT$ models $M$ satisfying the following assumption.
\begin{assumption}\label{ass:polynomial_sector}
For each $a\in \triangle$, the vector space $\opT_a$ coincides with the linear span of abstract monomials $\pmb{\eta}^I$ with $d(I)=a$ and the model $M\in \mcM_\mcT$ restricted to the polynomial sector $\bar{\opT}=\bigoplus_{a\in \triangle} \opT_a$, is the canonical polynomial model.
\end{assumption}
We point out that the assumption $K$ annihilates polynomials causes no real restriction on the type of kernels since the result \cite[Lem.~5.5]{hairer_14_RegStruct} adapts in a straightforward manner to our setting, see also Lemma \ref{lem:kernel_decompose} below. We point out that this assumption is convenient but not crucial for the theory, c.f. \cite{hairer_singh_22_manifolds}. In the remainder of this subsection we show how the action of kernels of this type are lifted onto the regularity structure and act on modelled distributions. Given a $\gamma \in \mbR\setminus \triangle$ we write $\mcK_\gamma$ for this lift; it corresponds to $K$ in the sense that for $f\in\msD^\gamma$
\begin{equation}\label{eq:ker_recon_commute}
\mathcal{R}\mathcal{K}_\gamma f=K(\mathcal{R}f)
\end{equation}
and in that it satisfies an appropriate version of the classical Schauder estimates. 
\begin{definition}
Given a sector $V$, a map $\mathcal{I}:V\mapsto \bar{\opT}$ is called an abstract integration map of order $\beta>0$ if it satisfies the following properties:
\begin{enumerate}
	\item For each $\alpha\in A,  \ \mathcal{I}:V_\alpha\to \opT_{\alpha+\beta}$, where $\opT_{\alpha+\beta}:=\{0\}$ for $\alpha+\beta\notin A$.
	\item It annihilates polynomials, that is $\mathcal{I}:\bar{\opT}\cap V\to  \{0\}$.
	\item For each $\tau \in \opT,\ \Gamma\in \mbG:\  (\mathcal{I}\ \Gamma- \Gamma \mathcal{I})(\tau)\in \bar{\opT}.$
\end{enumerate}
\end{definition}
Assume that the kernel $K$ satisfies Assumption \ref{ass:K0_decompose} for some $\beta>0$. We associate to $K$ the map $\mcJ:\mbR^d\to L(\opT,\bar{\opT})$ which for every $\tau\in \opT_\alpha$ is given by
\begin{equation}\label{eq:introduce_J(x)_operator}
\mcJ(x)\tau:= \sum_n \opbfP^{\alpha+\beta}_x (\Pi_x \tau\ast K_n),
\end{equation}
where the last sum is seen to converge absolutely by first observing that by Assumption \ref{ass:K0_decompose} for any $\tau\in \opT_\alpha$
\begin{equation*}
|X^I (\Pi_x \tau\ast K_n)(x)|= | \Pi_x \tau\ast X^I K_n(x)|=  |\Pi_x \tau (X^I_1K_n(x,\cdot))|\lesssim 2^{-n(\alpha+\beta- d(I))}\ ,
\end{equation*}
and then using Lemma~\ref{lem:coefficients_of_taylor}.
\begin{definition}\label{def realise K}
Given a regularity structure $\mcT$ equipped with an Integration map $\mathcal{I}$, a kernel $K$ and a model $M=(\Pi,\Gamma)$, we say the model $M$ realises $K$ for $\mcI$ if for each $\tau\in \opT_\alpha$ and each $x\in \mbG$,
\begin{equation*}
	\Pi_x\mathcal{I}\tau= K(\Pi_x\tau)- \Pi_x \mcJ(x)\tau .
\end{equation*}
\end{definition}
Now we can define the lift of the kernel K, namely for $f\in \msD^\gamma (V)$ we set: 
\begin{equation*}
\mathcal{K}_\gamma f(x):=\mcI f(x)+\mcJ(x)f(x)+\mcN_\gamma f(x),
\end{equation*}
where
$$(\mathcal{N}_\gamma f)(x)=\sum_n \opbfP^{\gamma+\beta}_x \big( (\mathcal{R}f-\Pi_xf(x)) \ast K_n\big)$$
where the last sum converges by the same argument as for $\mcJ (x)\tau$.
\begin{remark}
Given a kernel $K$ satisfying Assumption \ref{ass:K0_decompose} a regularity structure and model satisfying Assumption \ref{ass:polynomial_sector},
it turns out that one can always extend the regularity structure and model to be equipped with an integration map $\mathcal{I}$ realizing the kernel $K$. This is the content of the extension theorem found as \cite[Thm.~5.14]{hairer_14_RegStruct}, which holds in our setting as well. While we do not reproduce the whole proof since it is a straightforward adaptation of the original one, we present the main steps below, see Lemmas~\ref{lemma:commutation}, \ref{lem:analytic_pi_bounds} and \ref{lem:analytic_gamma_bound}, so that the interested reader will easily be able to fill in the remaining details.
\end{remark}
The next theorem which is an analogue of \cite[Thm.~5.12]{hairer_14_RegStruct}, confirms that $\mathcal{K}_\gamma$ does indeed correspond to $K$ in the sense of \eqref{eq:ker_recon_commute} and satisfies the desired Schauder estimates.
\begin{theorem}\label{th:schauder}
Let 
$\gamma\in \mathbb{R}\setminus \triangle$ and $\beta>0$ be such that 
$\gamma+\beta\not\in \triangle$, let
$K:\mbG\setminus\{e\}\rightarrow \mbR$ be a kernel satisfying Assumption~\ref{ass:K0_decompose} for $r\geq\gamma+\beta$, let $\mcT=(\opT,\opG)$ be a regularity structure and $M=(\Pi,\Gamma)$ be a model satisfying Assumption \ref{ass:polynomial_sector}. 
Furthermore assume that $\mcT$ is equipped with an abstract integration operator and $M$ realises $K$ for $\mcI$. Then for any sector $V$ of regularity $\alpha \in A$ the operator $\mathcal{K}_\gamma$ is a continuous linear map from $\msD_M^\gamma(V)$ to $\msD_{(\alpha+\beta)\wedge 0}^{\gamma+\beta}$ satisfying the identity
\begin{equation*}
	\mathcal{R}\mathcal{K}_\gamma f=K(\mathcal{R}f),
\end{equation*}
for all $f\in \msD_M^\gamma(V)$. Furthermore, if we denote by $M=(\Pi,\Gamma),\ \bar{M}=(\bar{\Pi},\bar{\Gamma})$ two models as above and by $\mathcal{K}_\gamma$, respectively $\bar{\mathcal{K}}_{\gamma}$ the associated operators, one has the bound
$$
\|\mathcal{K}_\gamma f; \bar{\mathcal{K}}_{\gamma}\bar{f} \|_{\gamma+\beta;\mfK} 
\lesssim_C
\vertiii{ f,\bar{f}}_{\gamma;\bar{\mfK}} \\
+\|\Pi-\bar{\Pi}\|_{\gamma;\bar{\mfK}}  +\|\Gamma-\bar{\Gamma}\|_{\gamma;\bar{\mfK}} \ ,
$$
where the implicit constant depends on the norms of $M, \bar M$ and $f\in\msD_M^\gamma(V)$ and $\bar{f}\in \msD^\gamma_{\bar M}(V)$.
\end{theorem}
The fact that the next lemma (\cite[Lem.~5.16]{hairer_14_RegStruct}) still holds in our setting is the underlying reason that all proofs extend in a rather straight forward manner from \cite{hairer_14_RegStruct} and one does not require a more involved notion of abstract integration map, which is for example the case on general Riemannian manifolds c.f.\cite{hairer_singh_22_manifolds}.
\begin{lemma}\label{lemma:commutation}
In the setting of Theorem \ref{th:schauder} one has the identity
\begin{equation*}
	\Gamma_{x,y}(\mcI+\mcJ(y))= (\mcI+\mcJ(x))\Gamma_{x,y}
\end{equation*}
for all $x,y\in \mbG$.
\end{lemma}
\begin{proof}\,
The proof consists of unravelling the Definitions and using the fact that the map $\Pi_x$ is injective when restricted to polynomials, exactly as in \cite{hairer_14_RegStruct}. 
\end{proof}
We introduce the following quantity; for $I\in \mathbb{N}^d, \alpha>0, n\in \mathbb{N}$ and $x,y,z\in \mbG$, set 
\begin{equation}
K^{I,\alpha}_{n,xy}(z):=  X^I_1 K_n(y,z)-\tilde{\opP}^{\alpha+\beta-d(I)}_x[X^I_1 K_n(\cdot,z)] (y)=X^I_1 \big(K_n(\,\cdot\,,z)-\tilde{\opP}^{\alpha+\beta}_x[K_n(\cdot,z)]\big)(y)
\end{equation}
where the second equality follows from Remark \ref{rem:derivative_of_taylor}. Here we reiterate that we are using the notation $K(x,y):=K(y^{-1}x)$ where $K$ satisfies Assumption \ref{ass:K0_decompose}. Taylor's theorem and Remark \ref{rem:rescaled_taylor} then yield that
\begin{align}\label{eqeq}
K^{I,\alpha}_{n,xy}(z) 
&=\sum_{|J|\leq [\alpha+\beta-d(I)]+1, d(J)\geq \alpha+\beta-d(I)} \int_{\mbG} X_1^{J}(X_1^I K_n)(x\tilde{z},z) Q^J(x^{-1}y, d\tilde{z}) \ .
\end{align}
As in \cite{hairer_14_RegStruct} the motivation for defining these quantities comes from the identity
\begin{align*}
\Pi_x (\mathcal{I}\tau)(\phi)= K(\Pi_x\tau)(\phi) - \Pi_xJ(x)\tau(\phi) &= \sum_{n} \int \big(\Pi_x\tau (K_n(y,\cdot)) - \tilde{\opP}_x^{\alpha+\beta}[K_n(\Pi_x\tau)](y)\big)\phi(y) \, dy  \\
&= \sum_{n} \int \Pi_x\tau\big(K_n(y,\cdot) - \tilde{\opP}_{x,1}^{\alpha+\beta}[K_n(\cdot,\cdot) ](y)\big) \phi(y) \, dy  \\
&= \sum_{n} \int \Pi_x\tau(K^{0,\alpha}_{n,xy}) \phi(y)  \, dy
\end{align*}
where we use the subscript for $\tilde{\opP}_{x,1}[K(\cdot,\cdot)](y)$ to clarify that one expands in the first coordinate.
The next Lemma collects the results of \cite[Lem.~5.18, Lemma 5.19]{hairer_14_RegStruct}, the proofs of which adapt ad verbatim.
\begin{lemma}\label{lem:analytic_pi_bounds}
Let $\alpha\in A$ and $\tau\in \opT_\alpha$ and assume $\alpha+\beta\notin \triangle$, then the following bound holds 
\begin{equation}
	|(\Pi_y\tau) (K^{I,\alpha}_{n,xy})|\lesssim \|\Pi\|_{\alpha,B_2(x)}(1+\|\Gamma\|_{\alpha,B_2(x)})\sum_{\delta>0} 2^{\delta n}|y^{-1}x|^{\delta+\alpha+\beta -d(I)} \ ,
\end{equation} 
where the sum over $\delta$ runs over a finite set of positive real numbers. The same bound holds for $|(\Pi_x\tau) (K^{I,\alpha}_{n,xy})|$, as well as the analogue bound on the difference of two models.

Furthermore one has the bound 
\begin{equation}
	\sum_n \left|\int_{\mbG} (\Pi_x\tau) (K^{I,\alpha}_{n,xy})\phi_x^\lambda(y) \dd y\,\right|\lesssim \lambda^{\alpha+\beta} \|\Pi\|_{\alpha;B_2(x)}(1+\|\Gamma\|_{\alpha;B_2(x)})\ ,
\end{equation} 
as well as the analogous bound for the difference of two models, where all these bounds hold uniformly over $x\in \mbG$, $\lambda \in (0,1]$ and $\phi\in \mathfrak{B}_r$
\end{lemma}

As in \cite{hairer_14_RegStruct} we introduce for $x,y\in \mbG$ the following operator
\begin{equation}
\mcJ_{x,y}:=\mcJ(x)\Gamma_{x,y}-\Gamma_{x,y} \mcJ(y)
\end{equation}
\begin{lemma}\label{lem:analytic_gamma_bound}
Under the assumptions of Theorem \ref{th:schauder} one has for each $a\in \Delta$, $\tau\in \opT_\alpha$
\begin{equation*}
	|\mcJ_{x,y}\tau|_a\lesssim \|\Pi\|_{\alpha, B_2(x)}(1+\|\Gamma\|_{\alpha, B_2(x)})|y^{-1}x|^{\alpha+\beta-a}|\tau|
\end{equation*}
uniformly in $x,y\in \mbG$ satisfying $x\in B_1(y)$. Again, the analogous bound for the difference of two models holds as well.
\end{lemma}
\begin{proof}\,
First we write $\mcJ_{x,y}=\sum_n \mcJ^n_{x,y}$ where each $\mcJ^n_{x,y}$ is defined by replacing $K$ by only one summand $K_n$ in the definition of $\mcJ_{x,y}$.
Observe that since $\mcJ^n_{x,y}\tau\in \bar{\opT}_{<\alpha+\beta}$ it suffices by Lemma~\ref{lem:to_bound_polynomials} 
to show that
$$p^n_\tau(w):=\Pi_e(\mcJ^n_{x,y}\tau)(w)\in \mathcal{P}_{\alpha+\beta}$$
satisfies 
\begin{equation}\label{required bound}
	\sum_{n} |(X^I p^n_\tau)(e)|\lesssim \|\Pi\|_{\alpha, B_2(x)}(1+\|\Gamma\|_{\alpha, B_2(x)})|y^{-1}x|^{\alpha+\beta-d(I)}|\tau|\ .
\end{equation}
We treat the cases $|y^{-1}x|\leq 2^{-n}$ and $|y^{-1}x|> 2^{-n}$ separately. In the case $|y^{-1}x|\leq 2^{-n}$, using the definitions of the polynomial regularity structure, we find
\begin{equation*}
	p^n_\tau(x^{-1} z)= \Pi_e(\mcJ^n_{x,y}\tau)(x^{-1}z)= \Pi_x(\mathcal{J}^n_{x,y}\tau)(z)
\end{equation*}
and thus
\begin{equation*}
	X^Ip^n_\tau(e)= X^I (\Pi_x\mcJ^n_{x,y}\tau)(x)\ .
\end{equation*}
Using the analogous notation $\mcJ^n(x)$ for the operator consisting only of one summand in \eqref{eq:introduce_J(x)_operator}, we find
\begin{align*}
	\Pi_x\mcJ^n(x)\Gamma_{x,y}\tau (z) =
	\sum_{\gamma\leq\alpha} \Pi_x\left(\mcJ^n(x)\mcQ_\gamma\Gamma_{x,y}\tau\right) (z) &=\sum_{\gamma\leq\alpha} \Pi_x \mathbf{P}^{\gamma+\beta}_x [\big(\Pi_x \mcQ_\gamma\Gamma_{x,y}\tau\big)_2 K_n(\cdot,\cdot)] (z)\\
	&=(\Pi_y \tau) \left(\tilde{\opP}^{\alpha+\beta}_{x,1}[K_n(\cdot,\cdot)](z)\right)\\
	&\quad
	-\sum_{\gamma\leq\alpha} \Pi_x \mcQ_\gamma\Gamma_{x,y}\tau\big( \tilde{\opP}^{\alpha+\beta}_{x,1}[K_n(\cdot,\cdot)]-\tilde{\opP}^{\gamma+\beta}_{x,1}[K_n(\cdot,\cdot)]\big)(z)\\
	&= \tilde{\opP}^{\alpha+\beta}_{x}[(\Pi_y \tau)_2K_n(\cdot,\cdot)](z)\\
	&\quad
	-\sum_{\gamma<\alpha} \Pi_x \mcQ_\gamma\Gamma_{x,y}\tau\big( \tilde{\opP}^{\alpha+\beta}_{x,1}[K_n(\cdot,\cdot)]-\tilde{\opP}^{\gamma+\beta}_{x,1}[K_n(\cdot,\cdot)]\big)(z)
\end{align*}
and 
\begin{align*}
	\Pi_x\Gamma_{x,y}\mcJ^n(y)\tau (z) 
	= \Pi_y\mcJ^n(y)\tau (z) = \tilde{\opP}_y^{\alpha+\beta}[(\Pi_y\tau)_2 K_n(\cdot,\cdot)] (z)
	=\tilde{\opP}^{\alpha+\beta}_{x}\left[  \left(\Pi_y\tau\right)_2\big( \tilde{\opP}_{y,1}^{\alpha+\beta}[K_n(\cdot,\cdot)]\big) \right](z)  \ ,
\end{align*}
where in the third equality we used the fact that $\tilde{\opP}^{\alpha+\beta}_{x}$ acts as the identity map on polynomial functions of degree less then $\alpha+\beta$. Therefore,
\begin{align*}
	p_\tau^n(x^{-1}z) &= 
	\underbrace{\tilde{\opP}^{\alpha+\beta}_{x}\Big[(\Pi_y \tau)_2\left(K_n(\cdot,\cdot)- \tilde{\opP}_{y,1}^{\alpha+\beta}[K_n(\cdot,\cdot)]\right) \Big](z)}_{:=p^n_{\tau,1}(x^{-1}z)} \\
	&\quad-\sum_{\gamma<\alpha} \underbrace{\Pi_x \mcQ_\gamma\Gamma_{x,y}\tau\big( \tilde{\opP}^{\alpha+\beta}_{x,1}[K_n(\cdot,\cdot)]-\tilde{\opP}^{\gamma+\beta}_{x,1}[K_n(\cdot,\cdot)]\big)(z)}_{:=p^{n,\gamma}_{\tau, 2}(x^{-1}z)} 
\end{align*}
Using the general formula $X^I(\tilde{\opP}_x^a[f])(x)=X^I({\opP}_x^a[f])(e)= X^If(x)$ for $d(I)< a$ we find that
\begin{equation*}
	X^I p^n_{\tau,1} (e)=X^I \Big[(\Pi_y \tau)_2\big(K_n(\cdot,\cdot)- \tilde{\opP}_{y,1}^{\alpha+\beta}[K_n(\cdot,\cdot)]\big) \Big](x)= (\Pi_y \tau)(K^{I,\alpha+\beta}_{n,x,y})
\end{equation*}
and so by Lemma~\ref{lem:analytic_pi_bounds} we find that
$$\sum_{n\, :\, |y^{-1}x|\leq 2^{-n}} |X^I p^n_{\tau,1} (e)| \lesssim |y^{-1}x|^{\alpha+\beta-d(I)}|\tau| \ . $$
Concerning $p^{n,\gamma}_{\tau, 2}$, since one has the identity
\begin{equation*}
	p^{n,\gamma}_{\tau, 2}(x^{-1}z)=-\tilde{\opP}^{\gamma+\beta}_{x}[(\Pi_x \mcQ_\gamma\Gamma_{x,y}\tau)_2K_n(\cdot,\cdot)](z)+  \tilde{\opP}^{\alpha+\beta}_{x}[(\Pi_x \mcQ_\gamma\Gamma_{x,y}\tau)_2K_n(\cdot,\cdot)](z)\ , 
\end{equation*}
we find that $X^I p^{n,\gamma}_{\tau, 2}(e)= 0$ unless $d(I)\in [\gamma+\beta,\alpha+\beta)$, in which case 
\begin{equation*}
	X^I p^{n,\gamma}_{\tau, 2}(e)= \Pi_x \mcQ_\gamma\Gamma_{x,y}\tau\big( X^I_1 K_n(x,\cdot)\big) \ .
\end{equation*}
Thus,
\begin{equation*}
	\sum_{n\, :\,  |y^{-1}x|\leq 2^{-n}} |X^I p^{n,\gamma}_{\tau, 2}(e)| \lesssim \sum_{n\, :\,  |y^{-1}x|\leq 2^{-n}} |y^{-1}x|^{\alpha-\gamma} 2^{-n(\gamma +\beta -d(I))}|\tau|_a \lesssim |y^{-1}x|^{\alpha+\beta -d(I)}|\tau|\ ,
\end{equation*}
where we used that the case $d(I)= \gamma+\beta$ does not arise as by Assumption~\ref{ass:polynomial_sector}.

To treat the case $|y^{-1}x|> 2^{-n}$, one writes %
\begin{equation*}
	p_\tau^n(z) = \sum_{\gamma\leq
		\alpha} \underbrace{\Pi_x\mcJ^n(x)(\mcQ_\gamma\Gamma_{x,y}\tau)(xz)}_{=:q_{\tau, 1}^{n,\gamma}(z) } - \underbrace{\Pi_x\Gamma_{x,y}\mcJ^n(y)\tau(xz)}_{=:q_{\tau, 2}^n(z)}
\end{equation*}
and using the fact that $\gamma+ \beta\notin \triangle$ one finds 
\begin{equation*}
	|X^I q_{\tau, 1}^{n,\gamma}(e)|= \begin{cases}
		\left| \Pi_x \mcQ_\gamma\Gamma_{x,y}\tau \left( X^I_1 K_n(x,\cdot) \right) \right|, & \text{if }d(I)< \gamma+ \beta,\\
		0,&\text{otherwise},
	\end{cases}
\end{equation*}
which is bounded uniformly over $|y^{-1}x|< 1$ by a constant multiple of \linebreak $|y^{-1}x|^{\alpha-\gamma} 2^{-n(\gamma+\beta- d(I))}$. Concerning $q_{\tau, 2}^n$, we find that
\begin{equation*}
	q_{\tau, 2}^n(z)= \Pi_x\Gamma_{x,y}\mcJ^n(y)\tau(xz)=\Pi_y \mcJ^n(y)\tau(xz) = {\opP}_{y,1}^{\alpha+\beta}[(\Pi_y\tau)_2  K_n(\cdot,\cdot)] (y^{-1}xz))
\end{equation*}
and therefore by Lemma~\ref{lem:second_polynomial}
\begin{align*}
	|X^I q_{2}^n(e)|  &\lesssim \sum_{d(I)<d(J)\leq \alpha+\beta} \sup_{d(J')\leq d(J)} \left| \Pi_y\tau \left( X^{J'}_1K_n(y,\cdot)\right)\right| \, |y^{-1}x|^{d(J)-d(I)}|\tau|\\
	& \lesssim \sum_{d(I)\leq d(J)< \alpha+\beta} \sup_{d(J')\leq d(J)} 2^{-n(\alpha+\beta- d(J'))}|y^{-1}x|^{d(J)-d(I)}|\tau|\\
	&\lesssim\sum_{d(I)\leq d(J)< \alpha+\beta} 2^{-n(\alpha+\beta- d(J))}|y^{-1}x|^{d(J)-d(I)}|\tau|\ .
\end{align*}
Since $p^n_{\tau}= \sum_{\gamma\leq \alpha} q_{\tau, 1}^{n,\gamma} + q_{\tau,2}^n$ one obtains (\ref{required bound}) by summing over $\left\{n\in \mathbb{N} \ : \ |y^{-1}x|> 2^{-n}\right\}$.
\end{proof}
\begin{proof}[Proof of Theorem~\ref{th:schauder}]\,

First, we need to check that for $a \in A$ it holds that
\begin{equation}\label{eq:schauder_main_step}
	| \mathcal{K}_\gamma f(x)-\Gamma_{x,y}\mathcal{K}_\gamma f(y)|_a\lesssim |y^{-1}x|^{\gamma+\beta-a}\ .
\end{equation}
For $a\notin \triangle$ this bound follows from Lemma \ref{lemma:commutation}, continuity of the operator of $\mcI$ and properties of modelled distributions by an ad verbatim adaptation of the same argument in the proof of \cite[Thm.~ 5.12]{hairer_14_RegStruct}.
The only place, where the proof \cite[Thm.~ 5.12]{hairer_14_RegStruct} does not adapt directly is when showing the bound for $a \in \triangle$.
As in the proof of Lemma~\ref{lem:analytic_gamma_bound} we use Lemma~\ref{lem:to_bound_polynomials} to circumvent the difficulty of not having explicit Taylor expansions in our setting. The rest of the proof adapts almost ad verbatim.

We first note that the polynomial part of $\mathcal{K}_\gamma f(x)-\Gamma_{x,y}\mathcal{K}_\gamma f(y)$ is given by
\begin{equation}\label{eq polynomial part}
	P:=\underbrace{\Gamma_{x,y}\mathcal{N}_\gamma f(y)}_{=:P_1} - \underbrace{\mathcal{N}_\gamma f(x)}_{=:P_2} + \underbrace{ \mcJ(x)(\Gamma_{x,y}f(y)-f(x) )}_{=:P_3}
\end{equation}
and thus, in order to prove \eqref{eq:schauder_main_step} for the polynomial part it suffices by Lemma~\ref{lem:to_bound_polynomials} to show that for any $d(I)<\gamma+\beta$, $p:=\Pi_e P$ satisfies
\begin{equation}
	|X^I p(e)| \lesssim  |y^{-1}x|^{\gamma+\beta-d(I)} \ .
\end{equation}
Using \eqref{eq:local_kernel_decompose} we define the analogous decompositions $\mcJ=\sum_n \mcJ^n$ and $\mcN_\gamma=\sum_n \mcN^n_\gamma$ and similarly write $P=\sum_{n=0} P^n$, etc. 

As usual we use different strategies for small and large scales, starting with the case $2^{-n}\leq |y^{-1}x|$. In this case we separately estimate $p_i:=\Pi_e P_i$ for $i\in \{1,2,3\}$ . 
\begin{itemize}
	\item Noting that 
	$$p^n_{1}(z)= \Pi_x  \Gamma_{x,y}\mathcal{N}_\gamma f(y) (xz)=  {P}_{y,1}^{\gamma+\beta}\left[(\mathcal{R}f-\Pi_yf(y))_2 \big( K_n(\cdot,\cdot)\big)\right] (y^{-1}xz),$$
	we find by Lemma~\ref{lem:second_polynomial}
	\begin{align*}
		|X^I p^n_{1}(e)| 
		&\lesssim \sum_{d(I)\leq d(J)<\gamma +\beta} \sup_{d(J')\leq d(J)} \left|\big(\mathcal{R}f-\Pi_yf(y)\big) \big( X^{J'}_1K_n(y,\cdot)\big)\right| |y^{-1}x|^{d(J)-d(I)}\\
		& \lesssim \sum_{d(I)\leq d(J)< \gamma+\beta} \sup_{d(J')\leq d(J)} 2^{-n(\gamma+\beta- d(J'))}|y^{-1}x|^{d(J)-d(I)}\\
		&\lesssim\sum_{d(I)\leq d(J)< \gamma+\beta} 2^{-n(\gamma+\beta- d(J))}|y^{-1}x|^{d(J)-d(I)} \ .
	\end{align*}
	
	\item Regarding $p^n_2$ we find that $X^I p^n_2(e)= \left( \mathcal{R}f-\Pi_xf(x)\right) \left( X^{I}_1K_n(x,\cdot)\right)$ and thus by the Reconstruction, see Theorem~\ref{th:Recontruction},
	\begin{equation*}
		| X^I p^n_2(e)| \lesssim 2^{-n(\gamma+\beta- d(I))} \ .
	\end{equation*}
	\item Lastly, for $p^n_3$, using the definition of $\mcJ (x)$ and properties of models and modelled distributions we find 
	\begin{align*}
		|X^I p^n_3(e)|
		&\leq \sum_{\delta\in A :\  d(I)-\beta <\delta <\gamma} \left|\big(\Pi_x \mcQ_{\delta}(\Gamma_{x,y}f(y)-f(x) \big)(X^I_1K(x,\cdot))\right| \\
		&\lesssim  \sum_{\delta\in A :\  d(I)-\beta <\delta <\gamma} |y^{-1}x|^{\gamma-\delta} 2^{-n(\beta+\delta-d(I))} \ .
	\end{align*}
\end{itemize}
Therefore summing over $2^{-n}\leq |y^{-1}x|$, since $\beta+\delta\notin \triangle$ for any $\delta\in A$ and $\beta+\gamma\notin \triangle$,   we find 
\begin{equation*}
	\sum_{n \,:\, 2^{-n}\leq |y^{-1}x|} |X^I p^n(e)|\leq  \sum_{n \,:\,\ 2^{-n}\leq |y^{-1}x|}\sum_{i=1}^3 |X^I p_i^n(e)| \lesssim |y^{-1}x|^{\gamma+\beta-d(I)} \ .
\end{equation*}

Next we turn to the case $2^{-n}> |y^{-1}x|$; here a computation analogous to the one preceding \cite[Eq.~(5.48)]{hairer_14_RegStruct} gives
\begin{equation}\label{Schauder large scale expression to bound}
	X^I p^n(e)= (\Pi_yf(y)-\mathcal{R}f)(K^{I,\gamma}_{n,x,y}) -\sum_{\zeta\leq d(I)-\beta} \big( \Pi_x\mcQ_\zeta(\Gamma_{x,y}f(y)-f(x)\big) \big(X^I_1K_n(x,\cdot)\big)
\end{equation}
which can be bounded exactly the way it is done there.

Finally, one may follow the concluding steps of the proof of \cite[Thm.~5.12]{hairer_14_RegStruct}, finding that for any $\phi \in \mfB$,
\begin{equation*}
	\big(\Pi_x(\mcK_\gamma f(x))-K(\mathcal{R}f)\big) (\phi^\lambda_x)= \sum_n \int \big( \Pi_x f(x)-\mathcal{R}f\big) (K_{n,xy}^\gamma)\phi_x^\lambda(y) dy\,
\end{equation*}
and thus one obtains the identity $\mcR\mcK_\gamma f = K(\mcR f)$.

The proof of the difference bound follows by similar steps as to those presented here and applied in the proof of \cite{hairer_14_RegStruct}.
\end{proof}
\subsubsection{Schauder Estimates for Singular Modelled Distributions}
Since our main application will be to semi-linear evolution equations we will often require a Schauder estimate for modelled distributions with permissible singularities near a given subgroup $P\subset \mbG$, as in Section~\ref{subsec:singular_modelled} on singular modelled distributions.
\begin{prop}\label{prop:singular_schauder}
In the setting of Section~\ref{subsec:singular_modelled} and Theorem~\ref{th:schauder}, given a sector $V$ of regularity $\alpha \in A$, the operator $\mcK_\gamma$ is well defined on $\msD^{\gamma,\eta}_P(V)$ for $\eta<\gamma$ provided that  $\gamma>0$ and  $\eta\wedge \alpha>-|\mfm|$. Furthermore, if $\gamma+\beta\notin \triangle$ and $\eta+\beta\notin \triangle$, one has $\mcK_\gamma f\in \msD^{\gamma+\beta,(\eta\wedge \alpha)+\beta}_P$ continuity bound
\begin{equation}\label{eq:singular_schauder}
	\vertiii{\mathcal{K}_\gamma f; \bar{\mathcal{K}}_\gamma\bar{f}}_{\gamma+\beta,(\eta\wedge \alpha)+\beta;\mfK} \lesssim 
	\vertiii{f; \bar{f}}_{\gamma,\eta;\bar{\mfK}} 
	+\|\Pi-\bar{\Pi}\|_{\gamma;\bar{\mfK}} +\|\Gamma-\bar{\Gamma}\|_{\gamma;\bar{\mfK}} \ 
\end{equation}
over any compact set $\mfK \subset \mbG$ and admissible models, where the implicit constant is uniform in  the semi-norms of $M, \bar M$ and $f\in \msD_{P,M}^{\gamma,\eta}(V),\,\bar{f}\in \msD_{P,\bar{M}}^{\gamma,\eta}(V)$ on $\bar{\mfK}$.
\end{prop}
\begin{proof}\,
The proof follows exactly along the same lines as the one of \cite[Prop.~6.16]{hairer_14_RegStruct}, where as in the proof of Theorem~\ref{th:schauder} the only modification needed is due to the fact that our Taylor expansions are non explicit. Using the same notation as in the proof of Theorem \ref{th:schauder}, we recall \eqref{eq polynomial part}
\begin{equation}
	P:=\underbrace{\Gamma_{x,y}\mathcal{N}_\gamma f(y)}_{=:P_1}- \underbrace{\mathcal{N}_\gamma f(x)}_{=:P_2} + \underbrace{ \mcJ(x)(\Gamma_{x,y}f(y)-f(x) )}_{=:P_3}.
\end{equation}
This time, in order to show that $\mcK_\gamma f \in \msD^{\gamma+\beta,(\eta\wedge \alpha)+\beta}_{P}$, by Lemma~\ref{lem:to_bound_polynomials}, we are required to show that for $p(e):=\Pi_e P = \sum_{i=1}^3 \Pi_e P_i$,
\begin{equation}
	|X^I p(e)| \lesssim {|y^{-1}x|}^{\gamma+\beta-d(I)} |x,y|^{(\eta\wedge \alpha)+\beta-\gamma}_P\, ,
\end{equation}
There are three scales, which need separate arguments. 
\begin{itemize}
	\item In the case $2^{-n}\leq |y^{-1}x|$ one proceeds exactly as in the proof of Theorem~\ref{th:schauder} using the decomposition $p^n = \sum_{i=1}^3 p^n_i$.
	\item In the cases $2^{-n}\in  \left(|y^{-1}x|, \frac{1}{2}|x,y|_P\right)$ and $2^{-n}\geq  \frac{1}{2}|x,y|_P$ one uses Equation~\eqref{Schauder large scale expression to bound} and proceeds exactly as in \cite{hairer_14_RegStruct} from there onwards.
\end{itemize}
Again, the proof of the difference bound follows analogously.
\end{proof}
\subsection{Local Operations}\label{subsec:local_ops}
To handle SPDEs using regularity structures we require suitable extensions to modelled distributions of standard local operations.
%			such as differentiation, multiplication and composition with smooth functions. 
%			
Since these extensions adapt mutatis mutandis from \cite{hairer_14_RegStruct}, we only provide pointers to the relevant literature.
\begin{enumerate}
\item Differential operators can be lifted to the regularity structure and the model exactly as in \cite[Def.~14.1]{friz_hairer_20_introduction} and one retains \cite[Prop. 5.28 \& Prop. 6.15]{hairer_14_RegStruct}.
We note that any left invariant differential operator $\mfL$ of homogeneous degree $\beta$ on $\mbG$ is of the form
$$\mfL=\sum_{d(I)= \beta} a_I X^I,$$ 
for some $I\in \mbN^d$. Thus, in order to lift such an $\mfL$, it suffices to lift each of the differential operators $\{X_i\}_{i=1}^d$ to abstract differential operators. 
In Section~\ref{sec:derivatives_and_abstract_polynomials} we lifted the differential operators $X_i$ to the polynomial regularity structure. 
\item 	Next, recall the notion of a product on a regularity structure, c.f.\ \cite[Def.~14.3]{friz_hairer_20_introduction}. In particular \cite[Thm.~4.7 \& Prop.~6.12]{hairer_14_RegStruct} still holds.
\item 	Lastly, we mention that  composition of modelled distributions with smooth functions, c.f.\ \cite[Section~4.1]{hairer_14_RegStruct} as well as the coresponding results adapt ad verbatim to the setting of homogeneous Lie groups. 
\end{enumerate}
%			%
\subsection{Symmetries}\label{subsec:symmetries}
As in \cite[Sec.~3.6]{hairer_14_RegStruct}, we shall consider modelled distributions which respect certain symmetries of $\mbG$. In this work we will only consider the symmetries of $\mbG$ under the canonical left action of a discrete subgroup $\mfG\subset \mbG$ as in Section~\ref{sec:discrete_subgroups} acting by $\mfG\times\mbG\ni (n, x)\mapsto (n x)\in \mbG$. We extend this to an action on function $\psi:\mbG\rightarrow \mbR$ by pull-back, i.e.
\begin{equation*}
(n^* \psi)(x):= \psi(n^{-1}x) = \psi_n(x).
\end{equation*}
For a regularity structure $\mcT=(\opT,\opG)$ we give the following definition, by analogy with \cite[Def.~3.33]{hairer_14_RegStruct} but restricted to this more straightforward setting. 
\begin{definition}\label{def:adapted_model}
Given a discrete sub-group $\mfG\subset \mbG$ as above, we say that a model $M=(\Pi,\Gamma)$ is adapted to the action of $\mfG$ if, for every test function, $\phi \in C^\infty_c(\mbG)$, $x \in \mbG$, $\tau \in \opT$ and $n \in \mfG$ one has
$$(\Pi_{nx}\tau)(n^* \psi) = (\Pi_x  \tau)(\psi)\quad \text{and}\quad \Gamma_{nx,ny}  = \Gamma_{x,y}.$$
A modelled distribution $f :\mbG\rightarrow \opT$ is said to be symmetric if $ f( nx )= f(x)$ for every $x \in \mbG$ and $n \in \mfG$.
\end{definition}
The following proposition is an amalgam of \cite[Prop.~~3.38 \& Prop.~~5.23]{hairer_14_RegStruct}.
\begin{prop}
Let $\mcT$ be a regularity structure, $\mfG\subset \mbG$ be a discrete subgroup as above and $M=(\Pi,\Gamma)$ be adapted to the action of $\mfG$ (according to Definition \ref{def:adapted_model}). Then, for every modelled distribution $f\in \msD^\gamma$, for some $\gamma >0$, symmetric with respect to the action of $\mfG$, the following hold;
\begin{enumerate}
	\item  For every $\phi \in C^\infty_c(\mbG)$ and $n\in \mfG$ one has $(\mcR f) (n^* \phi) = (\mcR f)(\phi)$.
	\item For any $x\in \mbG$ and $n\in \mfG$ one has $(\mcK_\gamma f)(nx)= (\mcK_\gamma f)(x).$
\end{enumerate}
\end{prop}
\begin{proof}\,
The proof of the first point follows exactly the steps of the proof of \linebreak \cite[Prop.~~3.38]{hairer_14_RegStruct} in our simplified setting.

The proof of the second follows similarly along the lines of the proof of \cite[Prop.~5.23]{hairer_14_RegStruct}, in particular using the assumption that the model is adapted to the action of $\mfG$ and the vector fields $\{X_i\}_{i=1}^d$ are left-translation invariant.
\end{proof}
\subsection{Bounds on Models}\label{subsec:bounds_on_models}
We state two results which, as in the Euclidean setting, allow one to reduce the number of stochastic estimates needed in order to obtain convergence of models to only $\Pi|_{\opT_{<0}}$.
The proof of the next proposition, \cite[Prop 3.31]{hairer_14_RegStruct}, adapts ad verbatim to our setting
\begin{prop}\label{prop:recursive_bounds_1}
Let $\mcT=(\opT,\opG)$ be a regularity structure and $(\Pi, \Gamma)$ a model. For $\alpha\in (0,\infty)\cap A$, the action of $\Pi$ on $\opT_\alpha$ is fully determined by the action of $\Pi$ on $\opT_{<\alpha}$ as well as the action of $\Gamma$ on $\opT_\alpha$. Furthermore, one has the bound
$$\sup_{x\in \mfK} \sup_{\lambda<1} \sup_{\phi\in \mfB_r} \sup_{\tau\in \opT_\alpha\setminus\{0\}} \frac{|\Pi_x \tau(\phi_{x}^\lambda)|}{\lambda^{\alpha} |\tau |_\alpha} \leq \| \Pi \|_{\alpha;\bar{\mfK}} \| \Gamma\|_{\alpha;\bar{\mfK}},$$
as well as the analogous difference bound.
\end{prop}
Furthermore, one has the following simple consequence of Lemma~\ref{lemma:commutation} and Lemma~\ref{lem:analytic_gamma_bound}.
\begin{prop}\label{prop:recursive_bounds_2}
Under the assumptions of Theorem \ref{th:schauder} one has for each $\alpha \in A\setminus \triangle$, $\tau\in \opT_\alpha$ and $\delta\in A$ one has
\begin{equation*}
	\sup_{x,y\in \mfK\, :\, {|y^{-1}x|}<1 } \frac{|\Gamma_{x,y} \mathcal{I} \tau|_{\delta+\beta}}{|\tau|_{\alpha-\beta} {|y^{-1}x|}^{\alpha-\delta}} \lesssim ( 1+\| \Gamma\|_{\alpha;\bar{\mfK}}) \|\Pi\|_{\alpha;\bar{ \mfK}}\,\,+ \sup_{x,y\in \mfK\, :\, {|y^{-1}x|}<1 } \frac{|\Gamma_{x,y} \tau|_\delta}{|\tau|_{\alpha-\beta} {|y^{-1}x|}^{\alpha-\delta}} \ .
\end{equation*}
Again, the analogous bound for the difference of two models holds as well. 
\end{prop}
\section{Applications to Semilinear Evolution Equations}\label{sec:evolution}
In this section we specialize to the setting, when the homogeneous Lie group $\mbG$ has distinguished time direction, see Section~\ref{subsec:kernel_examples} below for illustrative examples. Specifically, assume we are given decomposition of the Lie algebra $\mfg=\mfp^c \oplus  \mfp$ where both summands are $\fraks$-invariant subspaces and furthermore $\mfp^c$ is one dimensional. In line with this decomposition, in this section we deviate from the convention stated above Equation~\eqref{eq:eigenvectors} and reorder the basis of $\mfg$ so that we always have the time component, i.e. $X_1\in \mfp^c$, in the first slot. Under this assumption we also fix the basis $\bar{X}=\{X_i\}_{i=2}^{d}$ of $\mfp$ where $\mfs X_i =\mfs_i X_i$. %
We note that $P:= \exp(\mfp) \subset\mbG$ is a homogeneous subgroup as in \ref{subsec:singular_modelled} and we recall the notation defined there,
\begin{equation*}
\bar{\mfs }:= \mfs|_{\mfp}\quad \text{and}\quad 	|\bar{\mfs}| := \text{trace}(\mfs|_{\mfp}).
\end{equation*}
In particular we have $|\mfs| = \mfs_1+|\bar{\mfs}|$. From now on we identify the Lie group 
$\mbG$ with $\mathbb{R}\times P$ under the diffeomorphism
$$\mathbb{R}\times P\to \mbG, \quad (t,x)\mapsto x\exp(tX_{1})\ .$$
and we shall often use the notation $z=(t,x) \in \mbR\times P=\mbG$. Let us collect some useful facts;
\begin{itemize}
\item Under this identification we see that $\mathbb{R}\times \{e\}\subset \mbG$ is a homogeneous subgroup.
\item The scaling map behaves as expected, in that for any $z=(t,x)\in \mbG$ and $\delta>0$
$$ \delta\cdot (t,x)= (\delta^{\mfs_1}t, \delta\cdot x ) \ ,$$
where $\delta\cdot x$ is understood to be the restriction of the dilation to $P$.
\item The map $\tilde{\Phi}$ from \eqref{eq:defining_tilde_Phi} is related to our decomposition here, in that for any $X^{\mfp}\in \mfp$ and $t\in \mathbb{R}$, we have $\tilde{\Phi}(X^{\mfp} + tX_{1})=(t,\exp(X^{\mfp}) )$.
\item Recall the map $N_P:\mbG\to \mathbb{R}_{+}$ defined by \eqref{eq:Np_def}. There exists a constant $c'>0$ such that  $N_P(t,x)=c'|t|^{\frac{1}{\mfs_1}}$. In particular, by \eqref{eq:N_dist_compare_1} and \eqref{eq:N_dist_compare_2} there exist a constant $c>0$ such that for any $(t,x)\in \mbG$
\begin{equation}\label{eq:comparability}
	\frac{1}{c} |t|^{\frac{1}{\mfs_1}} \leq d\left( (t,x),P\right) \leq c |t|^{\frac{1}{\mfs_1}}\ .
\end{equation}
\end{itemize}
A core assumption for the remainder of this section will be that of non-anticipativity.
\begin{assumption}\label{ass:non_anticipative}
We say that $G:\mbG\setminus\{e\}\rightarrow \mbR$ is non-anticipative if for any $(t,x),\,(s,y) \in {\mbG}$ one has $G((s,y)^{-1}(t,x))=0$ whenever $s\geq t$. 
\end{assumption}
We will also require an assumption of prescribed homogeneity on the kernel, with respect to the dilation map.
\begin{assumption}\label{ass:kernel_scaling}
For $\sigma\in \mathbb{R}$, we say that $G:\mbG\setminus\{e\}\rightarrow \mbR$ is smoothly $\sigma$-homogeneous if it is smooth and for any $z\in \mbG\setminus \{e\}$ and $\lambda>0$,
\begin{equation*}\label{eq:evol_kernel_scaling}
	G(\lambda \cdot z) = \lambda^{\sigma} G(z).
\end{equation*}
\end{assumption}
We now have the following analogue of \cite[Lem.~7.4]{hairer_14_RegStruct}.
\begin{lemma}\label{lem:K_hat_def}
Given $G:\mbG\setminus \{e\} \rightarrow \mbR$ satisfying Assumption~\ref{ass:non_anticipative}  and Assumption~\ref{ass:kernel_scaling} with $\sigma=-|\bar{\mfs}|$, then there exists a smooth function $\hat{G}:P\rightarrow \mbR$ such that for all $(t,x)\in \mbG$ with $t>0$,
%%						%
\begin{equation}\label{eq:evol_kernel_time_scale}
	G(t,x) = t^{-\frac{\bar{\mfs}}{\mfs_1}} \hat{G}\left(t^{-\frac{1}{\mfs_1}}\cdot x \right).
\end{equation}
For every multi-index $I\in \mathbb{N}^{d-1}$ and every $n> 0$ there exists a constant $C>0$ such that uniformly over $x \in \mbG$,	%
\begin{equation}\label{eq:spatial_kernel_schwarz}
	|\bar{X}^I \hat{G}(x)| \leq C(1+|x|^2)^{-n}.
\end{equation}
\end{lemma}
\begin{proof}\,
The proof follows exactly the same steps of \cite[Lem.~7.4]{hairer_14_RegStruct} after replacing the usual derivatives there with the vector fields $\bar{X}=\{X_i\}_{i=2}^d$. 
\end{proof}

\begin{lemma}\label{lem:kernel_decompose} 
Let $G:\mbG\setminus \{e\} \rightarrow \mbR$ satisfy Assumption~\ref{ass:non_anticipative} and Assumption \ref{ass:kernel_scaling} with $\sigma = -|\bar{\mfs}|$. Then, for any $r>0$, there exist smooth functions  $K:\mbG\setminus \{e\} \rightarrow \mbR$ and $K_{-1}: \mbG\to \mathbb{R}$, both satisfying Assumption~\ref{ass:non_anticipative} and such that $G= K+ K_{-1}$ and
\begin{itemize}
	\item $K$ is compactly and satisfies Assumption \ref{ass:K0_decompose} with $\beta= \mfs_1$.
	\item $K_{-1}:\mbG\rightarrow \mbR$ is globally smooth and such that $X^I K_{-1}\in L^\infty_\loc (\mbR;L^1(P))$ for any $I\in \mbN^{d}$.
\end{itemize}	
\end{lemma}
\begin{proof}\,
The proof readily adapts from \cite[Lem.~5.5 \& Lem.~7.7]{hairer_14_RegStruct} with the only real change being to skip the last step in the proof of  \cite[Lem.~7.7]{hairer_14_RegStruct} and instead using Faa di Bruno's formula along with \eqref{eq:spatial_kernel_schwarz} to obtain the claimed integrabillity of $K_{-1}$. 
\end{proof}
\begin{remark}
Note that more careful analysis yields improved bounds on $K_{-1}$, however, $K_{-1}\in L^\infty_\loc(\mbR;L^1(P))$ suffices for our purposes.
\end{remark}

\subsection{Short Time Behaviour of Kernel Convolutions}
Together, the following two results show, in our setting, that the lift of the kernel applied to a modelled distribution, $f$, can be controlled on sets of the form $O_T:= \{z\in \mbG \,:\, d(z,P)\leq T\}$ only using information about $f$ on the same set. Note that using \eqref{eq:comparability} there exists a constant $c>0$ such that $[0,cT]\times P\subset O_T$. Hence, the corresponding notion of local solutions, described in Section~\ref{subsec:example_fixed_point}, corresponds to the usual one.

First, we introduce some final pieces of notation, let $\bR^+ :\mbR\times P \rightarrow \mbR$ be a map such that for all $x\in P$ one has $\bR^+(t,x)=1$ for $t>0$ and $\bR^+(t,x)=0$ for $t\leq 0$. From now on we shall also use subspaces of the previously introduced H\"older spaces (Definition~\ref{def:Holder}), modelled distributions (Definition~\ref{def:modelled_dist}) and singular modelled distributions (Definition~\ref{def:sing_modelled_dist}) writing, for example
\begin{equation*}
\bar{\mcC}^\alpha(\mbG)\subset {\mcC}^\alpha(\mbG) ,\quad \bar{\msD}^{\gamma}_{\alpha}\subset {\msD}^{\gamma}_{\alpha},\quad \bar{\msD}^{\gamma,\eta}_{\alpha,P}\subset {\msD}^{\gamma,\eta}_{\alpha,P},
\end{equation*}
where we allow the set $\mfK$ in the relevant definitions to be any closed subset $\mfK\subseteq O_T$ for some $T>0$ (in particular $\mfK$ is not necessarily compact). We shall often write the corresponding semi-norms for example as
\begin{equation*}
\|\,\cdot\,\|_{\alpha;O_T\cap \mfK},\quad \vertiii{\,\cdot\,}_{\gamma;O_T\cap\mfK},\quad \vertiii{\,\cdot\,}_{\gamma;\eta;O_T\cap \mfK}\ ,
\end{equation*}
where in particular we allow $\mfK= O_T$.

With these notions at hand the proof of the next theorem adapts directly from the proof of \cite[Thm.~7.1]{hairer_14_RegStruct}. 
\begin{theorem}\label{th:K0_time_scaling}
Let $\gamma>0$, $\mcT$ be a regularity structure and $M:= (\Pi,\Gamma)$ and $\bar{M}= (\bar{\Pi},\bar{\Gamma})$ models and 
$K: \mbG\setminus\{e\}\rightarrow \mbR$ a non-anticipative kernel (see Assumption~\ref{ass:non_anticipative}) such that the assumptions of  Theorem~\ref{th:schauder} are satisfied for some $\beta > 0$ and a sector $V$ of regularity $\alpha>-\mfs_1$. Then, for every $T\in (0,1]$,  and $\eta>-\mfs_1$ and for $\kappa>0$ small enough
\begin{align}\label{eq:K0_scaling}
	\vertiii{\mathcal{K}_\gamma \bR^+ f}_{\gamma+\beta,(\eta \wedge \alpha)+\beta-\kappa;O_T} &\lesssim T^{\kappa/\mfs_1}\vertiii{f}_{\gamma,\eta;O_T}\\
	\vertiii{\mcK_\gamma \bR^+f;\bar{\mcK}_\gamma \bR^+ f}_{\gamma+\beta,(\eta \wedge \alpha)+\beta-\kappa;O_T} & \lesssim   T^{\kappa/\mfs_1}\left(\vertiii{f;\bar{f}}_{\gamma,\eta;O_T}+\vertiii{M;\bar{M}}_{\gamma;O_2}\right).\label{eq:K0_scaling_diff}
\end{align}
The constant in the first bound depends only on $\vertiii{M}_{\gamma;O_2}$, while in the second it may also depend on $\vertiii{f}_{\gamma,\eta;O_T}\vee\vertiii{\bar{f}}_{\gamma,\eta;O_T}$ and $\vertiii{M}_{\gamma;O_2} \vee \vertiii{\bar{M}}_{\gamma;O_2}$.
\end{theorem}
\begin{proof}\,
The proof follows almost ad verbatim the proof of \cite[Thm.~7.1]{hairer_14_RegStruct}, only using our modified definition of the sets $O_T:=\{z\in \mbG\,:\,\dd_{\mbG}(z,P)\leq T \}$.
\end{proof}
\begin{remark}
In fact, given any closed set $\mfK\subset \mbG$ the bounds \eqref{eq:K0_scaling} and \eqref{eq:K0_scaling_diff} both hold if $O_T$ is replaced on the left hand side by $O_T\cap \mfK$ and on the right hand side by $O_T\cap \bar{\mfK}$. However, we will not make use of this fact in this article.
\end{remark}
While Theorem~\ref{th:K0_time_scaling} treats the lift of the singular part of the kernel the following lemma shows that the application of the smooth remainder can be lifted into the polynomial regularity structure in a similar manner. 
\begin{lemma}\label{lem:kernel_remainder_lift}
Let $K_{-1}:\mbG\rightarrow \mbR$ be a smooth, non-anticipative kernel on $\mbG$, such that for any $I \in \mbN^{d}$ the functions $X^I K_{-1}(t,\,\cdot\,)$ are bounded in $L^{1}(P)$, locally uniformly in $t\in \mbR$. Then, under the assumptions of Theorem~\ref{th:K0_time_scaling} one has the bound
\begin{equation*}
	\vertiii{\opbfP_{(\,\cdot\,)}^\gamma [K_{-1} \mcR_M\bR^+ f ] }_{\gamma+\beta, (\eta \wedge \alpha) +\beta-\kappa;O_T}\lesssim T^{\kappa/\mfs_1} \vertiii{ f}_{\gamma, {\eta};O_T}   
\end{equation*}
as well as the analogous difference bound 
$$\vertiii{\opbfP_{(\,\cdot\,)}^\gamma [K_{-1} \mcR_M \bR^+ f ]; \opbfP_{(\,\cdot\,)} ^\gamma [K_{-1}\mcR_{\bar{M}}\bR^+ \bar{f} ] }_{\gamma+\beta, (\eta \wedge \alpha)+\beta -\kappa;O_T}\lesssim T^{\kappa/\mfs_1} \left( \vertiii{ f; \bar{f}}_{\gamma, {\eta};O_T} + \vertiii{M;\bar{M}}_{\gamma, O_1}\right) \ .  $$
\end{lemma}
\begin{proof}\, The argument adapts mutatis mutandis form the proof of \cite[Lem.~7.3]{hairer_14_RegStruct}, replacing  the compact support assumption therein by the integrability property of $K_{-1}$.
\end{proof}
As discussed in the proof of \cite[Lem.~7.3]{hairer_14_RegStruct} the factor $T^{\kappa/\mfs_1}$ in both estimates of Lemma~\ref{lem:kernel_remainder_lift} could, in this case, be replaced with any arbitrary power of $T$, since $K_{-1}$ is smooth. Here, we choose to specify $T^{\kappa/\mfs_1}$ to as to agree with Theorem~\ref{th:K0_time_scaling}, since they will always be applied together.
\subsection{Initial Conditions}
We will only consider evolution equations on domains without boundary so that our only boundary data is the initial condition and proceed as in \cite[Sec.~7.2]{hairer_14_RegStruct}.
\begin{lemma}\label{lem:Initial_cond}
Let $u_0\in {\mcC}^\alpha( {P})$ such that $\|u_0\|_{{\mcC}^\alpha( {P})}<\infty$ and $\mcT=(\opT,\opG)$ be a regularity structure containing the polynomial regularity structure over $\mbG$ and let $G$ be a kernel satisfying Assumptions~\ref{ass:non_anticipative} and \ref{ass:kernel_scaling}.
We set for $t\neq 0$
\begin{equation*}
	G(u_0)(t,x) := \int_{P} G((0,y)^{-1}(t,x))u_0(y)\dd y, 
\end{equation*}	
where we used the suggestive but only formal notation if $\alpha\leq 0$.
Then, for every $\gamma>0$ the map $\opbfP^\gamma_{(\,\cdot\,)} [G(u_0)] : \mbG\setminus  P \rightarrow \opT$ belongs to $\msD^{\gamma,\alpha}_{0;P}$ for every $\gamma>(\alpha\wedge 0)$ and furthermore, for any $T>0$, one has
$$\vertiii{\opbfP^\gamma_{(\,\cdot\,)} [G(u_0)]}_{\gamma, {\eta};O_T}<\infty\,.$$
\end{lemma} 
\begin{proof}\,
We first decompose the kernel as in Lemma~\ref{lem:kernel_decompose} and write
$$G(u_0)(t,x)= K(u_0)(t,x)+K_{-1}(u_0)(t,x)\ .$$ 
For the summands of $K(u_0)(t,x)$, one can argue as in the proof of \cite[Lem.~ 7.5]{hairer_14_RegStruct} or as in Section~\ref{subsec:schauder}. 
The desired bound for $K_{-1}(u_0)(t,x)$ follows exactly as in Lemma \ref{lem:kernel_remainder_lift}.
\end{proof}
\subsection{An Example Fixed Point Theorem}\label{subsec:example_fixed_point}
At this point, we have all ingredients at hand to straightforwardly see that the very general fixed point Theorem \cite[Theorem~7.8]{hairer_14_RegStruct} as well as the other parts of \cite[Section~7.3]{hairer_14_RegStruct} 
adapt to the setting of homogeneous Lie Groups. For the sake of conciseness and with the primary example of Anderson-type models in mind, we refrain from presenting this material in all generality and instead show an example fixed point theorem. This result covers Anderson type equations, such as the one treated in Section~\ref{sec:worked_application}. Recall here the notation introduced at the start of this section, in particular the identity $|\mfs|=\mfs_1+|\bar{\mfs}|$.
\begin{theorem}\label{th:Anderson_fixed_point}
Let $\alpha \in (-\mfs_1,0]$, $\gamma \in (-\alpha,+\infty)$, $\eta \in (-\mfs_1 -\alpha,0]$ and $G:\mbG\setminus \{e\}\rightarrow \mbR$ be a kernel satisfying Assumptions~\ref{ass:non_anticipative} and \ref{ass:kernel_scaling} with $\sigma= -|\bar{\mfs}|$ and the decomposition $G=K+K_{-1}$ as in Lemma \ref{lem:kernel_decompose}. Furthermore, let $\mcT=(\opT,\opG)$
be a regularity structure containing the polynomial regularity structure, equipped with an abstract integration operator $\mcI$ of order $\mfs_1$ and containing an element $\Xi\in \opT_{\alpha}$, where $\alpha=\min A$. Then, given a model $M=(\Pi,\Gamma)$ which realises $K$ for $\mcI$ and satisfies the assumptions of Theorem~\ref{th:schauder} along with any $v \in \msD^{\gamma,\eta}_{0;P,}$, the map
\begin{equation}\label{eq:abstract_fixed_point_map}
	\mcM_{T;v}:\msD^{\gamma,\eta}_0 \rightarrow \msD^{\gamma,\eta}_0 \qquad U\mapsto \mcK_\gamma\bR^+(U\Xi) + \bP^\gamma_{(\,\cdot\,)}[K_{-1}\bR^+(U \Xi)] + v,
\end{equation}
is well-defined and there exists a unique fixed point for some $T>0$. 

Furthermore, if $\Pi_e\Xi$ and $v$ are $\mfG$-periodic and the model is adapted to the action of $\mfG$ on the group (as defined in Section~\ref{subsec:symmetries}), then the unique fixed point is as well.

Crucially, the solution map depends continuously on the model.
\end{theorem}
\begin{remark}
Note that using that the abstract mild equation given by \eqref{eq:abstract_fixed_point_map} is linear in $U$, the local time of existence  $T>0$ can be chosen independent of the initial condition. Thus, given suitable bounds on the model over a suitably large set of the form $\{z\in \mbG\,:\, d(z,P)\leq T+1\}$, by restarting the equation one obtains existence of a fixed point for arbitrary $T>0$. We refer to \cite[Theorem~5.2]{hairer_labbe_18_multiplicative} for details.
\end{remark}
\begin{remark}
In practice we will usually take $v = \bP^\gamma[G(u_0)]$ so that by Lemma~\ref{lem:Initial_cond} the condition $v\in \msD^{\gamma,\eta}_{0;P}$ is satisfied provided $u_0 \in \mcC^{\eta}(P)$, where $\eta+\alpha>-\mfs_1$. The notation $\mcK$ and $K_{-1}$ is carried over from Section~\ref{sec:r_structures} and Section~\ref{sec:evolution}, in particular $\mcK$ is the lift of $K$, the compactly supported, singular part of the fundamental solution, see Theorem~\ref{th:K0_time_scaling} and $K_{-1}$ the smooth remainder, see Lemma~\ref{lem:kernel_remainder_lift}.
\end{remark}
\begin{proof}[Proof sketch of Theorem~\ref{th:Anderson_fixed_point}]\,
The proof is a straightforward adaptation of the proof of \cite[Theorem~7.8]{hairer_14_RegStruct} and					
follows by applying Theorem~\ref{th:K0_time_scaling} and Lemma~\ref{lem:kernel_remainder_lift} to show that the map $\mcM_{\bar{T};v}$ has a unique fixed point in $\msD^{\gamma,\eta}_{0}$ for some $T>0$.
Since $\alpha=\min A$, the modelled distribution $x\mapsto \Xi$ is an element of $\msD^{\gamma,\gamma}_{\alpha;P}$ for any $\gamma>0$. Assuming $U\in \msD^{\gamma,\eta}_{0;P}$ 
one has
$	U\Xi \in \msD^{\gamma+\alpha,\eta+\alpha}_{\alpha;P}.$
So by using the singular Schauder estimate, Proposition~\ref{prop:singular_schauder}, along with the assumption that $\eta\leq 0$ we find $(\eta+\alpha)\wedge \alpha =  \eta+\alpha>-\mfs_1$, so that
\begin{equation*}
	\mcK(U\Xi ) \in \msD^{\gamma+\alpha+\beta,\eta+\alpha%\wedge \alpha) 
		+ \beta}_{(\alpha+\beta)\wedge 0;P}\ .
\end{equation*}
Then, one readily checks that  $\msD^{\gamma+\alpha+\beta,\eta+\alpha + \beta}_{(\alpha+\beta)\wedge 0;P}\hookrightarrow \msD^{\gamma,\eta}_{0;P}$ which concludes the proof that $\mcM_{T;v}$ is a self-mapping on $\msD^{\gamma,\eta}_{\alpha;P}$. By comparing $\mcM_{T;v} U$ and $\mcM_{T;v}\bar{U}$ for $U,\,\bar{U} \in \msD^{\gamma,\eta}_0$, one then shows that $\mcM_{T;v}$ is a contraction for sufficiently small $T>0$. 

From the steps outlined above one furthermore obtains the claims of $\mfG$-periodicity and the continuous dependence of the solution on the model.
\end{proof}
\subsection{Concrete Examples of Differential Operators and Kernels}\label{subsec:kernel_examples}
The theory developed in the proceeding sections provides the analytic framework to treat singular SPDEs using the tools of regularity structures where the linear part of the equation is given by an operator $\mfL$ satisfying the assumptions of Folland's theorem, Theorem~\ref{th:folland_translation}. In this section we present a non-exhaustive list of homogeneous Lie groups and linear differential operators, $\mfL$, to which our results apply. Generically, these equations are of the form,
\begin{equation*}
\mfL u = \partial_t u - \bar{\mfL} u = F(u, \bar{X} u,\ldots,\xi),\quad u\tzero=u_0, 
\end{equation*}
where $\xi$ is a suitable noise, the non-linearity is linear in the noise and only depends on lower order derivatives of $u$ than are contained in $\mfL$ which satisfies the assumptions of Theorem~\ref{th:folland_translation}, c.f. \cite{folland_75_subelliptic}.
\begin{enumerate}[label=\textbf{Setting \arabic*}]
\item \label{setting:one} \hspace{0.2em} $P$ is a stratified Lie Group (see Definition \ref{def:stratified}) with a basis $\{X_i\}_{i=1}^m$ of $W_1$, generating the Lie Algebra. We equip $\mbG:=\mathbb{R}\times P$ with the trivial homogeneous Lie Group structure 
\begin{equation*}
	\mbG\times \mbG  \ni\left((t,x), (t',x')\right) \mapsto (t+t',x x')
\end{equation*}
with the extended scaling $\lambda \cdot (t,x) = (\lambda^2 t, \lambda \cdot x)$. The heat type operator associated to the sub-Laplacian
\begin{equation*}
	\mfL=\partial_t - \sum_{i=1}^m X_i^2.
\end{equation*}
satisfies the assumptions of Theorem~\ref{th:folland_translation}.
and it follows from \cite[Thm.~3.1 \& Prop.~3.3]{folland_75_subelliptic} that $\mfL$ is non-anticipative.
We shall discuss a notable example of the setting, that of the heat operator on the Heisenberg group, below.
\item \label{setting:two} \hspace{0.2em} A generalisation of the above setting is to take $P$ a homogeneous Lie Group and equip $\mbG:= \mbR\times P$ with the same structure as above, but this time define the scaling $\lambda \cdot (t,x)= (\lambda^{2q}t,\lambda \cdot x)$ for $q \geq 2$ and an integer. A natural class of operators is given by
\begin{equation*}
	\mfL_Q=\partial_t -  Q(X_1,\ldots,X_m)\ ,
\end{equation*}
where $Q$ is a polynomial of homogeneous degree $2q$ and such that $\mfL_Q$  and $\mfL_Q^\ast$ are hypoelliptic and $\mfL_Q$ is non-anticipative, see \cite[Eq.~7.8]{hairer_14_RegStruct}. A notable example is the heat type operator associated to the bi-sub-Laplacian on a stratified Group.

\item \label{setting:three} \hspace{0.2em} The homogeneous Lie Group $\mbG= \mathbb{R}\times P$ is equipped with a non-trivial Lie group structure. In this case we look at differential operators of the form,
\begin{equation*}
	\mfL=X_0- Q(X_1,\ldots,X_m) \ .
\end{equation*}
where $X_0 = \partial_t + \bar{X}_0$ and the operators $\mfL$, $\mfL^\ast$ satisfy the criteria of Folland's theorem and $\mfL$ is non-anticipative. Examples are parabolic/restricted H\"ormander operators, c.f.\ \cite[Def.~1.1]{hairer_16_Advanced} or \cite[Ch.~3]{bell_95_degenerate} and the kinetic Fokker--Planck operator fits into this setting.
\end{enumerate}
\subsubsection{Heat Operator on the Heisenberg Group} \label{subsec:heisenberg_heat}
In the context of \ref{setting:one} and \ref{setting:two} we recall that the Heisenberg group, $\mbH^n$, defined in Section \ref{subsec:group_examples}, is a stratified Lie group. Identifying $\mbH^n = \mbR^{2n}\times \mbR$ we recall that 
\begin{equation*}
A_i(x,y,z) = \partial_{x_i} + y_i \partial_z,\quad B_i(x,y,z) = \partial_{y_i} - x_i \partial_z,\quad C(x,y,z) = \partial_z
\end{equation*}
are left-translation invariant vector fields.
It is readily checked that the operator
\begin{equation*}
\mfL = \partial_t -\sum_{i=1}^n (A_i^2+B_i^2),
\end{equation*}
is homogeneous with respect to the scaling, 
\begin{equation*}
\lambda \cdot (t,x,y,z):=(\lambda^{2}t,\lambda x,\lambda y,\lambda^2z),
\end{equation*}
Applying \cite[Thm.~2.1]{folland_75_subelliptic} there exists a unique, fundamental solution $K:\mbH^n\rightarrow \mbR$ associated to $\mfL$ which is smooth away from $e$ and satisfies the assumptions of Lemma \ref{lem:kernel_decompose}. As a result our framework allows for the study of semi-linear evolution equations of the form,
\begin{equation*}
\partial_t u - \mfL u = F(u,A_1u,\ldots,A_n u,B_1u,\ldots,B_nu, \xi), \quad u\tzero =u_0.
\end{equation*}
\subsubsection{Matrix Exponential Groups and Kolmogorov Type Operators}\label{subsec:k_operators}
A wide class of operators fitting into \ref{setting:three} above are the Kolmogorov (or K-type) operators on $\mbR\times \mbR^n$ for $n\geq 1$. The group structure is a matrix exponential group, as defined in Section \ref{subsec:group_examples}. Given two, rational, $n\times n$ block matrices,
\begin{equation*}
A = \begin{pmatrix}
	A_0 & \cdots & 0\\
	\vdots & \ddots & \vdots\\
	0 &\cdots& 0
\end{pmatrix},\qquad B = \begin{pmatrix}
	0 & B_1 & 0& \cdots &0\\
	0& 0 & B_2 &\cdots &0\\
	\vdots & \vdots & \ddots & \ddots & \vdots\\
	0 & 0&\cdots & 0 & B_{k}\\
	0 & 0 &\cdots  &0 &0
\end{pmatrix}. 
\end{equation*}
with $A_0$ constant, positive definite and of rank $q\leq n$ and each $B_i$ a $p_{i-1}\times p_{i}$ block matrix of rank $p_i$, where $q =p_0 \geq p_1\geq \cdots \geq p_{k}$ and $\sum_{i=1}^{k}p_i=n$. Then the linear operator, defined for $u:\mbR\times \mbR^{n}\rightarrow \mbR$ by
\begin{equation*}
\mfL u(t,z) = \partial_tu(t,z) - \nabla \cdot (A \nabla u(t,z)) + B z \cdot \nabla u(t,z),
\end{equation*}
satisfies H\"ormander's condition. Equipping $\mbR\times \mbR^n$ with the matrix exponential group structure associated to $B$ (and defined in Section \ref{subsec:group_examples}) it is readily checked that $\mfL$ satisfies the assumptions of Theorem~\ref{th:folland_translation}. In fact, there is an explicit formula for the fundamental solution. We first let
\begin{equation*}
C(t) := \frac{1}{t} \int _0^t \exp(-sB^\top) A \exp(-s B) \dd s
\end{equation*}
and then using the same notation for the effective spatial dimension, $|\bar{\mfs}|:=\sum_{i=0}^k (2i+1)p_i$,
\begin{equation*}
K(t,z) = \left(\frac{1}{(4\pi)^{n} t^{|\bar{\mfs}|} \det C(t)}\right)^{1/2} \exp\left(-\frac{C(t)^{-1} z \cdot z}{4t}\right).
\end{equation*}
The kinetic Fokker--Plank operator falls into this class. Consider the domain $\mbR\times \mbR^{2d}$, with variables $(t,x,v)$ and set, as block matrices,
\begin{equation*}
A = \begin{pmatrix}
	\mbI_d & 0\\
	0 & 0
\end{pmatrix}, \quad B= \begin{pmatrix}
	0& \mbI_d\\
	0&0
\end{pmatrix}.
\end{equation*}
Then the associated operator is
\begin{equation*}
\mfL u(t,v,x) = \partial_t u(t,v,x) - \Delta_v u(t,v,x) - v\cdot \nabla_x u(t,v,x).
\end{equation*}
and the fundamental solution in fact has the explicit form,
\begin{equation*}\label{eq:kFP_kernel}
K(t,v,x)= \frac{2\sqrt{3}}{d(4\pi)^d t^{2d+1}} \exp\left(-\frac{|v|^2}{4t} -\frac{3\left|x+\frac{v}{2}t\right|^2}{t^3}\right).
\end{equation*}
See~\cite[Sec.~7]{ilin_khashminski_64_brownian} for a derivation in the case $d=1$ and \cite{manfredini_97} in the general case. 
\section{A Regularity Structure for Anderson Equations}\label{sec:PAM_structure}
In this section we present a brief construction of a sufficiently rich regularity structure $\mcT=(\opT,\opG)$ in order to solve abstract fixed point equations of the form 
\begin{equation}\label{eq:abstract_equation}
U= \mathcal{I} (U\Xi)+U_0 \ ,
\end{equation}			
as treated by Theorem~\ref{th:Anderson_fixed_point}, where $\mathcal{I}$ denotes the abstract part of the lift of a $2$ regularising kernel as in Section~\ref{subsec:schauder}, $\Xi$ is the lift of a noise of regularity $\alpha$ and $U_0$ is the polynomial part of $U$. In order for the equation to be subcritical one needs to impose $\alpha>-2$ and thus it follows
from the previous discussions that we can look for a fixed point in a space $\msD^\gamma_P$ for $\gamma<2$. Thus we only need to incorporate abstract polynomials $\bfeta_i$ with $\fraks_i<2$ into our regularity structure and since one has the following identities in this case 
\begin{equation}\label{eq:algebraic_simplification}			
\eta_i(xy)= \eta_i(x)+\eta_i(y), \qquad \opP^\gamma_x[f](\,\cdot\,)= f(x)+ \sum_{i\, :\,  \fraks_i< \gamma} \eta_i(\cdot) X_i f(x)
\end{equation}		
one can work with essentially a truncated version of the algebraic framework developed in the Abelian setting $\mathbb{R}^d$ by \cite{bruned_hairer_zambotti_19_algebraic}.
Therefore, in the remainder of this short section, we freely use notations and definitions from \cite{bruned_hairer_zambotti_19_algebraic}, often without further explanation. 
\begin{remark}
Let us emphasize that for general equations the construction described below does not apply. We believe it would be interesting to study the ramifications of working on an general homogeneous Lie group to the results and constructions in \cite{bruned_hairer_zambotti_19_algebraic}, \cite{bruned_chandra_chevyrev_hairer_21_renormalising}.
\end{remark}
We define two edge types $\mathfrak{L}=\{\mathfrak{t},\Xi\}$ and declare $|\mathfrak{t}|=2$ and $|\Xi|= \alpha$ and as well as a scaling in the sense of \cite{bruned_hairer_zambotti_19_algebraic} $\underline{ \fraks}=(\fraks_1,\ldots, \fraks_n)$, 
where $n= \max\{i\leq d: \fraks_i<2\}$. \footnote{ Recall form Section~\ref{sec:analysis_background} that $\fraks_1,\ \fraks_2,..,\fraks_n$ are the first $n$ eigenvalues of the scaling $\fraks$ in increasing order and that $X_i$ are the corresponding eigenvectors}
Motivated by \eqref{eq:abstract_equation} we define the naive (normal) rule, c.f. \cite[Def.~5.7]{bruned_hairer_zambotti_19_algebraic}
\begin{equation}\label{eq:naive_rule}
\mathring{\mathfrak{R}}(\mathfrak{
	t})= \{(\mathfrak{t},\Xi), (\mathfrak{t}), (\Xi), ()\}, \quad \mathring{\mathfrak{R}}(\Xi)=\{()\}
	\end{equation}
	and consider its completion $\mathfrak{R}$ constructed in \cite[Prop. 5.21]{bruned_hairer_zambotti_19_algebraic} (note that this second step is only necessary if $\alpha\leq -3/2$).
	We denote by $\CT^{\ex}$ and $\CT_+^{ex}, \CT_-\subset \CT_-^{ex}$ the corresponding spaces constructed in \cite[Def.~5.26, Def.~5.29]{bruned_hairer_zambotti_19_algebraic} and summarise some important properties of these spaces.
	\begin{prop}\label{prop:struct}
The triple $\{\CT^{\ex}, \CT^{\ex}_+, \CT_-^{ex}, \CT_-\}$ has the following algebraic structure.
\begin{itemize}
	\item The spaces $\CT^{\ex}_+, \CT_-, \CT_-^{ex}$ are graded Hopf algebras with coproduct  denoted by by  $\triangle^+$ and $\triangle^-$ respectively. 
	\item The space $\CT^{\ex}$ is a right comodule over $\CT^{\ex}_+$ .
	\item The spaces $\CT^{\ex}$ and $\CT_+^{ex}$ are left comodules over $\CT_-$ .
	\item These spaces satisfy the co-interaction property, i.e. the following diagram commutes 
	\begin{equation*}
		\begin{tikzcd}
			\CT^{\ex} \arrow{r}{\triangle^+} \arrow[swap]{d}{\triangle^-} & \CT^{\ex}\otimes \CT^{\ex}_+ \arrow{d}{\mathcal{M}^{(1,3)(2)(4)}\circ(\triangle^- \otimes\triangle^-)} \\%
			\CT^{\ex}_-\otimes \CT^{\ex} \arrow{r}{\id\otimes \triangle^+}&  \CT^{\ex}_-\otimes \CT^{\ex} \otimes \CT^{\ex}_+
		\end{tikzcd}
	\end{equation*} 
\end{itemize} 
\end{prop}
Note that the fact that we can leverage the framework developed in \cite{bruned_hairer_zambotti_19_algebraic} relies crucially on the fact that it suffices to include polynomials of degree $<2$, which translates into the algebraic relation $\triangle^+(\bfeta_j) = \bfeta_j\otimes 1 + 1\otimes \bfeta_j$ whenever $\fraks_j<2$. In general one would need a modification $\tilde{\triangle}^+$ of the map $\triangle^+$ such that 		\linebreak
$\tilde{\triangle}^+(\bfeta_j) = \sum_{d(I)+d(J)=\fraks_j} C_j^{I,J}\bfeta^I \otimes \bfeta^J$. Furthermore, one would need a modification $\tilde{\triangle}^-$ of triangle $\triangle^-$ and both modifications, $\tilde{\triangle}^+$ and $\tilde{\triangle}^-$, would depend on the group structure of $\mathbb{G}$, since the form of higher order Taylor polynomials depends crucially on it. To circumvent these considerations we restrict ourselves to working with the following subspaces.

Let $\opT$ consists of the span of those basis vectors (trees) $T\in \CT^{\ex}$ where $T$ and its grading $|T|_+$ (c.f. \cite[Def.~5.3]{bruned_hairer_zambotti_19_algebraic}) satisfy the following properties
\begin{itemize}
\item\hspace{0.2em} it is planted and such that $|T|_+<\gamma$ or $T=T'\Xi$ where $T'$ is planted and satisfies $|T'|_+<\gamma$,
\item \hspace{0.2em} the only polynomial decorations appearing are at the second highest node and of degree $<\gamma$.
\end{itemize}
Similarly, we set 
$\CT_+\subset\CT_+^{ex}$ to consist of the subalgebra  generated by those trees $T\in \CT^{\ex}_+$ satisfying
\begin{itemize}
\item \hspace{0.2em} $|T|_+<\gamma$
\item \hspace{0.2em} the only polynomial decorations appearing are at the second highest node and of degree $<\gamma$.
\end{itemize} 
Observe that Proposition \ref{prop:struct} still holds with $\CT^{\ex}_+$ replaced by $\CT_+$ and $\CT^{\ex}$ replaced by $\opT$. We define $\mathcal{G}_+$ to be the character group of  $\CT_+$. From now on our regularity structure shall be given by 
\begin{equation}\label{eq:constructed_reg_struc}
\mcT=(\opT,\opG) \ ,
\end{equation}
where $\opG$ consists of the elements $\Gamma\in L(T,T) $ of the form $\Gamma= (\id\otimes f)\triangle^+$
for some $f\in \mathcal{G}_+$.
\subsection{Smooth Models}
In order to define the models, we have to deviate slightly from \cite{bruned_hairer_zambotti_19_algebraic}. For $i\in \{1,...,n\}$ we write $\bfeta_i$ instead of $X_i$ for abstract polynomials and we declare a map $\bfPi: \CT \to C^\infty (\mbG)$ (c.f. \cite[Def.~6.8]{bruned_hairer_zambotti_19_algebraic}) to be \textit{admissible} for a smooth function $\xi\in C^{\infty}(\mbG)$ and a kernel $K$ as in Assumption~\ref{ass:K0_decompose}, if for every $x\in \mbG$
$$\bfPi \Xi(x) = \xi(x), \quad \Pi\bfeta_i (x)= \eta_i (x) $$
and 
$$\bfPi \mathcal{I} \tau= K \bfPi\tau  , \quad \bfPi \mathcal{I}_i \tau= (X_iK)\bfPi\tau \ ,$$
where we write $\mathcal{I}$ resp. $\mathcal{I}_i$ for the operation of attaching an edge of type $\mathfrak{t}$, resp. $(\mathfrak{t},e_i)$ to the root.
For $\bfPi$ an admissible map as above and $z\in \mathbb{G}$ we recursively define $f_z\in \mathcal{G}_+$ and $\Pi_z$ by
\begin{align*}
f_z(\mathcal{I} \tau) &=\, -\hspace{-0.2em} \sum_{d(I)<|\mathcal{I}^\mathfrak{t}(\tau)|_+} \eta^I(z^{-1}) X^I K(\Pi_z\tau)(z)\\
\Pi_z \mathcal{I}\tau(\bar{z})&= K(\Pi_z\tau)(\bar{z}) \,-\hspace{-0.2em} \sum_{d(I)<|\mathcal{I}(\tau)|_+} \eta^I(z^{-1}\bar{z}) X^I K(\Pi_z\tau)(z)\ ,
\end{align*}
and by the analogous expression for $\mathcal{I}_i\tau$, where $i\in \{1,...,n\}$, c.f. \cite[Lem.~6.9]{bruned_hairer_zambotti_19_algebraic}.
We say that $\bfPi$ is \textit{canonical}, if 
it satisfies $\bfPi \Xi \tau =\bfPi \Xi \bfPi \tau$ for every $\tau$ such that $\Xi \tau\in \mathcal{T}$, and it does not see the extended decoration, i.e is reduced in the sense of \cite[Def.~6.21]{bruned_hairer_zambotti_19_algebraic}.
One can check that for such canonical lifts $\bfPi$
the maps $\Pi_z= (\bfPi\otimes f_z) \triangle^+$ and $\Gamma_{\gamma_{x,y}}$
where $\gamma_{x,y}=  (f_x^{-1}\otimes f_y)\triangle^+$ form a model
\begin{equation}\label{eq:canonical_model}
M=(\Pi, \Gamma_{\gamma})\ . 
\end{equation}
The renormalisation group $\mathcal{G}_-$ is defined to be the character group of the Hopf algebra $\CT^{\ex}_-$. We observe that in our case $\CT^{\ex}_-$ coincides (as an algebra) with the free algebra spanned by a family of linear trees. (In the case $\alpha>-3/2$, one furthermore has $\CT^{\ex}_-=\CT_-$  and we list the trees explicitly in the next section).
The group $\mathcal{G}_-$ acts on models as follows. For $g\in\mathcal{G}_-$ let $R_g=(g\otimes \id )\triangle^-$, we set  
\begin{equation}\label{eq:local_ref}
\hat{M}^g=(\hat{\Pi}^g, \hat{\Gamma}^g)= (\Pi R_g, \Gamma_{\gamma R_g}) \ .
\end{equation}
One can check that every element the orbit of a smooth canonical model as in \eqref{eq:canonical_model} is indeed a model.

\begin{remark}\label{rem:reduce_convergence}
Note that in order to show convergence of a family of models $\hat{M}^{(\eps)}=(\hat{\Pi}^{(\eps)}, \hat{\Gamma}^{(\eps)})$ of the form \eqref{eq:local_ref} for the regularity structure $\mathcal{T}= (\opT, \opG) $ for some smooth underlying noises $\xi_\epsilon$ and a sequence of elements $g_\epsilon\in {\mathcal{G}}_-$, it is sufficient to show convergence of 
$\hat{\Pi}^\eps|_{\opT_{\leq 0}}$ due to Propositions~\ref{prop:recursive_bounds_1} and \ref{prop:recursive_bounds_2}.
\end{remark}
We emphasise again that this section relied on the assumption $\gamma<2$.
\section{The Anderson Equation on Stratified Lie Groups}\label{sec:worked_application}
In this section we apply the machinery developed in this article in order to solve a class of parabolic Anderson type models on compact quotients of arbitrary $d$-dimensional, stratified, Lie groups. Let $\mbG$ be a stratified Lie group (recall Definition~\ref{def:stratified}) with Lie algebra $\mfg$ and $\mfG$ a lattice subgroup with its canonical 
left action on $\mbG$.  
Recall from Section~\ref{sec:discrete_subgroups} the definition of the quotient map $\pi: \mbG \to S =\mbG/ \mfG$,  we shall, without loss of generality assume that there exists a fundamental 
domain of this action, $K\subset \mbG$, such that $e\in B_N(e) \subset K $ for some $N\in \mathbb{N}$, large enough.\footnote{This is for example achieved by replacing the homogeneous norm on $\mbG$ with a multiple of itself. We only require this condition in order to make the integrands in Section~\ref{subsec:stochastic_estimates} supported on the fundamental domain $K$.}
For $m\leq d$, let $X_1,...,X_m$ be a basis of $W_1$, the family of left invariant vector fields generating $\mfg$ and we write $\tilde{X}_i=\pi_\ast X_i$ for the push-forward vector fields on $S$. Finally we recall the notion of convolution $*_S$ on $S$ induced by the right action of $\mbG$ on $S$, as defined in Section~\ref{sec:discrete_subgroups}. 

We shall solve equations where the driving noise satisfies the following assumption for some $\bar{\alpha}>0$.
\begin{assumption}\label{ass:covariance_assumption}
For $\bar{\alpha}\in (-|\fraks|/2, 0)$, assume that $\xi$ is of the form $\xi =\tilde{\xi}*_S \mfc_{\bar{\alpha}}$, where $\tilde{\xi}$ denotes a Gaussian white noise\footnote{Recalling that $S$ is a measure space, $\tilde{\xi}$ can be defined as a centred Gaussian field, over $L^2(S)$ and with covariance given by the $L^{2}(S)$ inner product.} on $S$ 
and $\mfc_{\bar{\alpha}}: \mbG\setminus\{ e \} \to \mathbb{R}$ is a function with bounded support, satisfying $|\mfc_{\bar{\alpha}} (x)|\lesssim |x|^{-|\fraks|+(|\fraks|/2+ \bar{\alpha})}$
as well as $\mfc_{\bar{\alpha}}(x^{-1})= \mfc_{\bar{\alpha}} (x)$ for all $x\in \mbG$.
\end{assumption}
\begin{remark}
For $\xi$ satisfying Assumption~\ref{ass:covariance_assumption}, a simple calculation shows that for any $\phi, \psi\in L^2(S)$,
$$ \mbE[\langle \xi, \phi\rangle \langle \xi, \psi\rangle ] = \int_S  (\phi*_S \mfc_{\bar{\alpha}})(x)( \psi*_S \mfc_{\bar{\alpha}})(x)\dd x \ . $$ 
In particular, coloured noise where the regularisation is given by a negative fractional power of the sub-Laplacian fits into our setting, c.f.\ \cite{baudoin_ouyang_tindel_wang_22_hPAM_ito}.
\end{remark}
\begin{remark}
The assumption that $\mfc_{\bar{\alpha}}(z)=\mfc_{\bar{\alpha}}(z^{-1})$ is not crucial, but allow us to reuse computations previously carried out for equations on Euclidean domains in \cite{hairer_pardoux_15_wong_zakai}.  A robust treatment of similar equations in the Euclidean setting, driven by more general driving noise, even beyond the Gaussian setting, is given in \cite{chandra_shen_17_nonGaussian}. A similar remark holds for the choice of mollifier $\rho$ in Theorem~\ref{th:main_application} below.
\end{remark}	
In the above setting, we have the following theorem.
\begin{theorem}\label{th:main_application}
Fix $T>0$ arbitrary and let $\xi$ be a noise satisfying Assumption~\ref{ass:covariance_assumption} with $\bar{\alpha}\in (-3/2,-1)$ and $\eta \in (-(2+\bar{\alpha}),0)$. For $\eps\in (0,1)$, let $u_{\eps}: [0,T]\times S\to \mathbb{R}$ be the unique solution to the random PDE 
\begin{equation}\label{eq:random_pde}
	\partial_t u_\eps= \sum_{i=1}^m \tilde{X}_i^2 u_\eps + u_\eps(\xi_\eps-c_\eps), \qquad u|_{t=0}=u_0\in \mcC^{\eta}(S) \ , 
\end{equation}
where $\xi_{\eps}:= \xi^\alpha *_S \rho^\eps$, for $\rho\in \mathcal{C}_c^{\infty}(B_1(e))$ satisfying $\rho(z)=\rho(z^{-1})$ and $\{c_\eps(\rho)\}_{\eps\in (0,1)}$ is a family of (diverging) constants, depending on $\rho$ and such that $|c_\eps| \lesssim \varepsilon^{2+2\bar{\alpha}}$.
Then, there exists a random function $u: [0,T]\times S\to \mbR$, independent of the chosen mollifier, $\rho\in \mathcal{C}_c^{\infty}(B_1(e))$, such that
$$\sup_{(t,x)\in [0,T]\times S} t^{-\eta}|u_\eps(t,x)-u(t,x)|\to 0$$
in probability as $\eps\to 0$.
\end{theorem}
\begin{remark}
We point out that the proof of Theorem~\ref{th:main_application} given below modifies mutatis mutandis to the case where the noise is allowed to depend on time.
In this case the natural modification of Assumption~\ref{ass:covariance_assumption}  is to assume that the noise is of the form $\xi =\tilde{\xi}\ast  \mfc_{\bar{\alpha}}$, where $\tilde{\xi}$ is a space-time white noise on $\mbR\times S$ and $\mfc_{\bar{\alpha}}$ satisfies $|\mfc_{\bar{\alpha}}(t,x)|\leq (|x|+|t|^{1/2})^{-|\fraks|/2-1+ \bar{\alpha}}$. The only difference in the proof is to replace all convolutions and integrals over $S$ with convolutions and integrals over $\mbR\times S$.
We point out that this form of noise is not covered by \cite{baudoin_ouyang_tindel_wang_22_hPAM_ito}, where the white in time assumption is required in order to make use of martingale arguments, see also Remark~\ref{rem:other_paper} below.
\end{remark}
\begin{remark}
The range of $\bar{\alpha}$ considered in Theorem~\ref{th:main_application} does not cover the full subcritical regime of the Anderson equation on a stratified Lie group. Indeed one expects an analogue of Theorem~\ref{th:main_application} to hold for any $\bar{\alpha} \in (-2,0)$. The main obstacle is the absence of a BPHZ type theorem in the setting of homogeneous Lie groups, see \cite{chandra_hairer_16,hairer_steele_23_spectral} for the corresponding result on $\mbR^d$. Let us further point out that in order for the Carnot group $\mbG$ to be non-trivial it must have scaled dimension $|\mfs|\geq 4$, hence the sub-critical regime for $\bar{\alpha}$ is necessarily contained in Assumption~\ref{ass:covariance_assumption}. 
\end{remark}

\begin{remark}\label{rem:other_paper}
Let us point to the work \cite{baudoin_ouyang_tindel_wang_22_hPAM_ito}, where, using It\^o calculus methods, the authors construct a solution-theory in the case, when the noise is white in time and coloured in space, corresponding to the full subcritical regime on the Heisenberg group. This probabilistic notion of solution circumvents the need for explicit renormalisation as in \eqref{eq:random_pde}. Furthermore, their results are obtained on the infinite volume group $\mbH^n$, rather than the compact quotient space that we consider here. We expect that by using weighted modelled distributions, c.f. \cite{hairer_labbe_18_multiplicative}, together with the techniques developed in \cite{hairer_pardoux_15_wong_zakai}, one can recover the solution constructed in \cite{baudoin_ouyang_tindel_wang_22_hPAM_ito} together with a Wong--Zakai type theorem in the analogue of the regime ${\bar{\alpha}\in (-3/2, -1)}$, treated by Theorem~\ref{th:main_application}. If one is interested in the full subcritical regime however,
this would not just require a suitable BPHZ-type theorem in the setting of homogeneous Lie groups, but furthermore a systematic way to single out the It\^o model among an (arbitrarily) high dimensional family of models obtained.
\end{remark}
\begin{proof}[Proof of Theorem~\ref{th:main_application}]\,
We follow a usual strategy in the theory of regularity structures, showing the equivalent theorem for the pulled back equations on $\mbG$. For this, we work with the regularity structure $\mathcal{T}=(\opT,\opG)$ constructed in \eqref{eq:constructed_reg_struc} with $|\Xi|= (-3/2,\bar{\alpha})$.
Note that since we are in \ref{setting:one} we can directly apply Theorem~\ref{th:folland_translation} to see that there exists a unique, fundamental solution $G:(\mbR \times \mbG)\setminus (0,e) \rightarrow \mbR$ and satisfying the assumptions of Lemma~\ref{lem:kernel_decompose}, so that we have $G= K+K_{-1}$ enjoying the properties described therein.

Denote
by $M^{(\varepsilon)}=(\Pi^{(\eps)}, \Gamma^{\varepsilon})$ 
the canonical model for $\mathcal{T}$, satisfying $\Pi^{(\eps)}_z \Xi = \xi_{\eps}$ and realising $K$ for $\mathcal{I}$. From now on we shall use the 
common tree notation, with the obvious changes in interpretation, as defined for example in \cite[Sec.~4.1]{hairer_pardoux_15_wong_zakai}, and write 
$$ \<Xi>:= \Xi ,\quad \<IXi>= \mathcal{I}\Xi, \quad \<Xi2>= \Xi \mathcal{I} \Xi, \quad \text{etc.}$$
Since $\bar{\alpha}\in (-3/2,-1$) the rule defined in \eqref{eq:naive_rule} is complete and $\mathcal{T}_-$ coincides (as an algebra) with the free commutative algebra generated by $\left\{	\<Xi>,\ \<Xi2>,\ \<Xi3> \right\}$.
Since the noise is Gaussian and centred, it is sufficient to work with the subgroup $\tilde{\mcG}_-\subset \mcG_-$ consisting of $g\in \mcG_-$ satisfying
$$ g(\tau)= 0 \quad \Rightarrow \quad \tau\in \Span \left\{  \<Xi>,\ \<Xi3>\right\} \ . $$
For each $\varepsilon\in (0,1)$ we fix $g_\eps \in \tilde{\mcG}_-$ to be the character specified by
$$g_\eps(\<Xi2>) = \mbE[\xi_\eps(e) K(\xi_\eps )(e)]\ .$$
Thus, applying  Theorem~\ref{th:Anderson_fixed_point}, we obtain a solution $U_\eps: [0,T]\times \mbG\to \opT_{<\gamma}$ to the abstract  lifted equation, \eqref{eq:abstract_fixed_point_map}, for the model $\hat{M}^{(\varepsilon)}:=M_\eps^{g_\eps}$.
Next we show that $u^R_\eps:=\mathcal{R}_{\hat{M}^{(\varepsilon)}} U_\eps$ actually solves Equation \eqref{eq:random_pde} in several steps.
\begin{itemize}
	\item First, note that for any $z=(t,x)\in [0,T]\times \mbG$, the abstract solution has the explicit form 
	$$U_\eps (z)= u_\eps^0 (z)\big( \mathbf{1} +\<IXi2>+ \<IXi> \big) 
	+\sum_{\fraks_i<\gamma} u^i_\eps (z) \bfeta_i \ ,$$
	since we are working with the expansion up to order $\gamma<\frac{3}{2}$. In particular $\hat{\Pi}_z^{(\varepsilon)}U_\eps (z) (z)= u_\eps^0 (z)$.
	\item We observe that therefore
	$$ (U_\eps \,\, \<Xi>) (z)= u_\eps^0 (z)\big( \<Xi>+ \<Xi2>+ \<Xi3> \big) 
	+\sum_{\fraks_i<\gamma} u^i_\eps (z) \bfeta_i\, \, \<Xi>\ .$$
	\item We also note that $\hat{\Pi}_z^{(\varepsilon)}  \left(\sum_{\fraks_i<\gamma} u^i_\eps (z) \bfeta_i\<Xi>\right)(z)=0$ and 
	$$\hat{\Pi}_z^{(\varepsilon)} \left( u_\eps^0 (z)\big( \<Xi>+ \<Xi2>+ \<Xi3>\big) \right) (z)
	= u_\eps^0 (z) \left( \hat{\Pi}_z^{(\varepsilon)} \<Xi>(z) + \hat{\Pi}_z^{(\varepsilon)}\big(  \<Xi2>+ \<Xi3>\big)(z) \right)
	= u_\eps^0 (z) (\xi_\eps(z) - c_\eps )\ .$$
	\item Thus, we can conclude using Remark~\ref{rem:reconstruction_smooth_models} that 
	$$\mathcal{R}_{\hat{M}^{(\varepsilon)}} ( U^g\<Xi>) 
	= u_\eps^0 (z) (\xi_\eps(z) - c_\eps )= \left(\mathcal{R}_{\hat{M}^{(\varepsilon)}}U^g\right)(z) \, (\xi_\eps(z) - c_\eps ) \  .$$
\end{itemize}
Thus it only remains to show that the family of models $\hat{M}^{(\varepsilon)}$ converges to a limiting model $M$, which is the content of Proposition~\ref{prop:models_actually_converge}, stated below.
\end{proof}
\begin{prop}\label{prop:models_actually_converge}
The familly of models $\hat{M}^{(\varepsilon)}=(\hat{\Pi}^{(\varepsilon)}, \hat{\Gamma}^{\varepsilon})$ constructed in the proof of Theorem~\ref{th:main_application} converges as $\varepsilon\to 0$.
\end{prop}
\begin{proof}\,
Note that by Remark~\ref{rem:reduce_convergence} it is sufficient to show convergence of $\hat{\Pi}^{(\varepsilon)}|_{\opT_{<0}}$ in the model topology, which follows from a stright-forward adaptation of Kolmogorov's criterion to models using the molifiers $\varphi^{(n)}$ constructed in Lemma~\ref{lem:convergence_of_test_functions} together with the estimates obtained in Proposition~\ref{prop:stochastic_estimates}.
\end{proof}
\begin{remark}
For a detailed proof of a Kolmogorov type criterion for general models on Euclidean domains, i.e. $\mathbb{G}=\mathbb{R}^{d}$, in the spirit above,  we refer to \cite[Appendix B]{hairer_steele_23_spectral}. 
\end{remark}
\subsection{Stochastic Estimates for the Renormalised Model}\label{subsec:stochastic_estimates}
In this subsection we obtain convergence the renormalised models.  						
We shall freely use techniques of \cite[Sec~10]{hairer_14_RegStruct} and as well as graphical notation similar to that in \cite[Sec.~5]{hairer_pardoux_15_wong_zakai}, with some slight changes in interpretation. For example we write
\begin{equation*}
\bigl( \Pi_{(0,e)}^{(\eps)} \<Xi2>\bigr)(\varphi^\lambda) = \;
\begin{tikzpicture}[scale=0.35,baseline=0.3cm]
	\node at (0,-1)  [root] (root) {};
	\node at (-2,1)  [dot] (left) {};
	\node at (-2,3)  [dot] (left1) {};
	\node at (0,1) [var] (variable1) {};
	\node at (0,3) [var] (variable2) {};
	
	\draw[testfcn] (left) to  (root);
	
	\draw[kernel1] (left1) to (left);
	\draw[rho] (variable2) to (left1); 
	\draw[rho] (variable1) to (left); 
\end{tikzpicture}\;
+ \;
\begin{tikzpicture}[scale=0.35,baseline=0.3cm]
	\node at (0,-1)  [root] (root) {};
	\node at (-2,1)  [dot] (left) {};
	\node at (-2,3)  [dot] (left1) {};
	\node at (0,2) [dot] (variable1) {};
	\node at (0,2) [dot] (variable2) {};
	
	\draw[testfcn] (left) to (root);
	
	\draw[kernel1] (left1) to (left);
	\draw[rho] (variable2) to (left1); 
	\draw[rho] (variable1) to (left); 
\end{tikzpicture}\; ,
\end{equation*}
with the following interpretations;
\begin{itemize}
\item The node \tikz[baseline=-3] \node [root] {}; represents the origin, $(0,e)\in\mbR \times \mbG$ while the edge \tikz[baseline=-0.1cm] \draw[testfcn] (1,0) to (0,0); represents integration against the rescaled test function $\varphi^{\lambda}$.
\item The node \tikz[baseline=-3] \node [var] {}; represents an instance of the $\mfG$ left periodic white noise on $\mbG$ while the edge \tikz[baseline=-0.1cm] \draw[rho] (0,0) to (1,0); represents the kernel $\mfc_{\bar{\alpha}}\ast \rho_{\eps}$. Note therefore, that when $\eps =0$ our diagrams will still contain dotted lines, as opposed to \cite{hairer_pardoux_15_wong_zakai}.
\item The nodes \tikz[baseline=-3] \node [dot] {}; represent dummy variables $z =(t,x)\in \mbR\times \mbG$ which are to be integrated out and the edges \tikz[baseline=-0.1cm] \draw[kernel] (0,0) to (1,0); represent integration against the kernel $K$. A barred arrow 
\tikz[baseline=-0.1cm] \draw[kernel1] (0,0) to (1,0);
represents a factor of $K(t-s,y^{-1}x) - K(-s,y^{-1})$, where
$(s,y)$ and $(t,x)$ are the coordinates of the start and end
points of the arrow respectively.
\end{itemize}
As in \cite[Sec.~10.2]{hairer_14_RegStruct} and \cite[Sec.~5.1.1]{hairer_pardoux_15_wong_zakai} we will also use the notation $\mcW^{(\eps;k)}\tau$ to denote the kernel associated to the $k^{\text{th}}$-homogeneous Wiener chaos component of $\Pi_{(0,e)}\tau$ and similarly $\hat{\mcW}^{(\eps;k)}\tau$ for the renormalised model. For example, we find that $\hat{\mcW}^{(\eps;2)}\<Xi2>((s,y);x_1,x_2)=\mcW^{(\eps;2)}\<Xi2>((s,y);x_1,x_2)$ is given by
\begin{equation*}
\mcW^{(\eps;2)}\<Xi2>((s,y);x_1,x_2) := \iint_{\mbR\times \mbG} (\mfc_{\bar{\alpha}} \ast \rho_\eps) (y_1^{-1}x_1) (\mfc_{\bar{\alpha}} \ast \rho_{\eps})(y^{-1}x_2) \left(K(s-s_1,y_1^{-1}y)-K(-s_1,y_1^{-1})\right) \dd s_1\dd y_1.
\end{equation*}
\begin{prop}\label{prop:stochastic_estimates}
Fix $\delta>0$ small enough. Then, for every $\lambda \in (0,1)$, $\varphi\in \mfB_r$ and $z=(t,x)\in \mathbb{R}\times \mbG$, there exist random variables
\begin{equation}\label{eq:renormalised_model}
	(\hat{\Pi}_{z}\<Xi>)(\varphi^\lambda),\qquad  (\hat{\Pi}_{z}\<Xi2>)(\varphi^\lambda), \qquad (\hat{\Pi}_{z}\<Xi3>)(\varphi^\lambda)
\end{equation}
such that for any $p\geq 1$ the following estimates hold uniformly over $\lambda,\varepsilon\in (0,1)$
\begin{enumerate}[label={\arabic*a)}]
	\item \label{it:Xi1_eps_bound}  $\quad\mbE\left[|(\hat{\Pi}^{(\eps)}_{z}\<Xi>)(\varphi^\lambda)|^p\right] \lesssim \lambda^{p(\bar{\alpha}-\delta)}$,
	\item \label{it:Xi2_eps_bound} $\quad
	\mbE\left[|(\hat{\Pi}^{(\eps)}_{z}\<Xi2>)(\varphi^\lambda)|^p\right] \lesssim \lambda^{p(2+2\bar{\alpha}-\delta )}$,
	\item \label{it:Xi3_eps_bound}
	$\quad\mbE\left[|(\hat{\Pi}^{(\eps)}_{z}\<Xi3>)(\varphi^\lambda)|^p\right] \lesssim \lambda^{p(4+3\bar{\alpha}-\delta )}$,
\end{enumerate}
as well as
\begin{enumerate}[label={\arabic*b)}]
	\item \label{it:Xi1_diff_bound} $\quad\mbE\left[|(\hat{\Pi}^{(\eps)}_{z}\<Xi>)(\varphi^\lambda)-(\hat{\Pi}_z\<Xi>)(\varphi^\lambda)|^p\right] \lesssim \varepsilon^{p\delta}  \lambda^{p(\bar{\alpha}-\delta)}$ ,
	\item \label{it:Xi2_diff_bound} $\quad\mbE\left[|(\hat{\Pi}^{(\eps)}_{z}\<Xi2>)(\varphi^\lambda)-(\hat{\Pi}_z\<Xi2>)(\varphi^\lambda)|^p\right] \lesssim \varepsilon^{p\delta}\lambda^{p(2+2\bar{\alpha}-\delta )}$ ,
	\item \label{it:Xi3_diff_bound} $\quad\mbE\left[|(\hat{\Pi}^{(\eps)}_{z}\<Xi3>)(\varphi^\lambda)-(\hat{\Pi}_z\<Xi3>)(\varphi^\lambda)|^p\right] \lesssim \varepsilon^{p\delta}\lambda^{p(4+3\bar{\alpha}-\delta )}$ .
\end{enumerate}
\end{prop}
\begin{proof}
\,
Firstly, we note that by translation invariance of the noise it is enough to consider $z=(0,e)$. Furthermore, since all random variables belong to a finite Wiener chaos, it suffices to show these estimates for $p=2$, \cite[Lem.~10.5]{hairer_14_RegStruct}. 
The first two estimates, \ref{it:Xi1_eps_bound} and \ref{it:Xi1_diff_bound}, follow readily by Ito's isometry.		

For the next items we recall \cite[Lem.~ 10.14 \& 10.18]{hairer_14_RegStruct}, the natural analogues of which hold in our setting as well.					
We can write (with the adaptation of the meaning of the diagrams described above) the Wiener chaos decomposition of the second symbol as 
\begin{equation}\label{eq:decompPiXiTwo}
	\bigl(\hat \Pi_{(0,e)}^{(\eps)} \<Xi2>\bigr)(\varphi^\lambda) = \;
	\begin{tikzpicture}[scale=0.35,baseline=0.3cm]
		\node at (0,-1)  [root] (root) {};
		\node at (-2,1)  [dot] (left) {};
		\node at (-2,3)  [dot] (left1) {};
		\node at (0,1) [var] (variable1) {};
		\node at (0,3) [var] (variable2) {};
		
		\draw[testfcn] (left) to  (root);
		
		\draw[kernel1] (left1) to (left);
		\draw[rho] (variable2) to (left1); 
		\draw[rho] (variable1) to (left); 
	\end{tikzpicture}\;
	- \;
	\begin{tikzpicture}[scale=0.35,baseline=0.3cm]
		\node at (0,-1)  [root] (root) {};
		\node at (-1,1)  [dot] (left) {};
		\node at (0,3)  [dot] (top) {};
		\node at (1,1) [dot] (right) {};
		
		\draw[testfcn] (left) to  (root);
		
		\draw[kernel] (right) to (root);
		\draw[rho] (top) to (right); 
		\draw[rho] (top) to (left); 
	\end{tikzpicture}\;.
\end{equation}		
The first summand is the graphical representation of the iterated integral, $I_2((\hat{\mathcal{W}}^{(\varepsilon;2)} \<Xi2>)(z))$ integrated against a test function centred at $(0,e)\in \mbR\times \mbG$; see for example the analogous second term in \cite[Eq.~(10.23)]{hairer_14_RegStruct}.
The second summand is the graphical analogue of the first term in \cite[Eq.~(10.23)]{hairer_14_RegStruct} given by $I_0((\hat{\mathcal{W}}^{(\varepsilon;0)} \<Xi2>)(z)).$
Next note that by virtually an identical calculation as in the beginning of the proof of \cite[Thm.~10.19]{hairer_14_RegStruct} one finds,
\begin{equation}\label{eq:double_schauder}
	|\langle ((\hat{\mathcal{W}}^{(\varepsilon;1)} \<IXi>)(z), (\hat{\mathcal{W}}^{(\varepsilon;1)} \<IXi>)(\bar{z}) \rangle|\lesssim
	|z|^{2\bar{\alpha}+4}+ 
	|\bar{z}|^{2\bar{\alpha}+4} \ ,
\end{equation}
which implies
$$|\langle (\hat{\mathcal{W}}^{(\varepsilon;2)} \<Xi2>)(z), (\hat{\mathcal{W}}^{(\varepsilon;2)} \<Xi2>)(\bar{z}) \rangle|\lesssim
(|z|^{2\bar{\alpha}+4}+ 
|\bar{z}|^{2\bar{\alpha}+4}) |\bar{z}^{-1}z|^{2\bar{\alpha}}\ .
$$
Together with the simple estimate $I_0((\hat{\mathcal{W}}^{(\varepsilon;0)} \<Xi2>)(z))|\lesssim |z|^{2\bar{\alpha}+2}$
this gives \ref{it:Xi2_eps_bound}.

Next we turn to show \ref{it:Xi3_eps_bound}. First, note that by a similar calculation as for \eqref{eq:double_schauder} one finds
$$|\langle (\hat{\mathcal{W}}^{(\varepsilon;2)} \<IXi2>)(w), (\hat{\mathcal{W}}^{(\varepsilon;2)} \<IXi2>)(\bar{w}) \rangle|
\lesssim |w|^{{2(4+2\bar{\alpha})} }+ |\bar{w}|^{2(4+2\bar{\alpha})}  \ .
$$
Therefore, together with the bound
$I_0((\hat{\mathcal{W}}^{(\varepsilon;0)} \<IXi2>)(w))|\lesssim |w|^{2\bar{\alpha}+4}$ we find
$$\mbE\left|\big((\hat{\Pi}^{(\eps)}_{(0,e)}\<IXi2>)\diamond(\hat{\Pi}^{(\eps)}_{(0,e)}\<Xi> )\big)(\varphi^\lambda) \right|^2 \lesssim\lambda^{2(4+3\bar{\alpha})} \ , $$
where $\diamond$ denotes the Wick product.
As in \cite[Sec.~5.3.2]{hairer_pardoux_15_wong_zakai} we find 
\begin{equation*}
	\left(\hat \Pi_{(0,e)}^{(\eps)} \<Xi3>\right)(\varphi^\lambda)=
	\begin{tikzpicture}[scale=0.35,baseline=0.5cm]
		\node at (0,-1)  [root] (root) {};
		\node at (-2,1)  [dot] (left) {};
		\node at (-2,3)  [dot] (left1) {};
		\node at (-2,5)  [dot] (left2) {};
		\node at (0,1) [var] (variable1) {};
		\node at (0,3) [var] (variable2) {};
		\node at (0,5) [var] (variable3) {};
		
		\draw[testfcn] (left) to  (root);
		
		\draw[kernel1] (left1) to (left);
		\draw[kernel1] (left2) to (left1);
		\draw[rho] (variable3) to (left2); 
		\draw[rho] (variable2) to (left1); 
		\draw[rho] (variable1) to (left); 
	\end{tikzpicture}
	\; + \;
	\left(
	\begin{tikzpicture}[scale=0.35,baseline=0.5cm]
		\node at (0,-1)  [root] (root) {};
		\node at (-2,1)  [dot] (left) {};
		\node at (-2,3)  [dot] (left1) {};
		\node at (-2,5)  [dot] (left2) {};
		\node at (0,2) [dot] (variable2) {};
		\node at (0,5) [var] (variable3) {};
		
		\draw[testfcn] (left) to  (root);
		
		\draw[kernel1] (left1) to (left);
		\draw[kernel1] (left2) to (left1);
		\draw[rho] (variable3) to (left2); 
		\draw[rho] (variable2) to (left1); 
		\draw[rho] (variable2) to (left); 
	\end{tikzpicture}
	\; - \;
	\begin{tikzpicture}[scale=0.35,baseline=0.5cm]
		\node at (0,-1)  [root] (root) {};
		\node at (-2,1)  [dot] (left) {};
		\node at (-1,3)  [dot] (left1) {};
		\node at (-2,5)  [dot] (left2) {};
		\node at (0,2) [dot] (variable2) {};
		\node at (0,5) [var] (variable3) {};
		
		\draw[testfcn] (left) to  (root);
		
		\draw[kernel] (left1) to (left);
		\draw[kernel1] (left2) to (left);
		\draw[rho] (variable3) to (left2); 
		\draw[rho] (variable2) to (left1); 
		\draw[rho] (variable2) to (left); 
	\end{tikzpicture}
	\right)
	\; + \;
	\begin{tikzpicture}[scale=0.35,baseline=0.5cm]
		\node at (0,-1)  [root] (root) {};
		\node at (-2,1)  [dot] (left) {};
		\node at (-2,3)  [dot] (jnct) {};
		\node at (-0.75,3)  [dot] (left1) {};
		\node at (-2,5)  [dot] (left2) {};
		\node at (0.5,3) [var] (variable2) {};
		
		\draw[testfcn] (left) to  (root);
		
		\draw[kernel1] (left1) to (left);
		\draw[kernel1] (left2) to (left1);
		\draw[rho] (jnct) to (left2); 
		\draw[rho] (variable2) to (left1); 
		\draw[rho] (jnct) to (left); 
	\end{tikzpicture}
	\; - \;
	\begin{tikzpicture}[scale=0.35,baseline=0.3cm]
		\node at (0,-1)  [root] (root) {};
		\node at (-1,1)  [dot] (left) {};
		\node at (-0.5,2)  [dot] (topl) {};
		\node at (1,1) [dot] (right) {};
		\node at (2,3)  [var] (variable) {};
		\node at (0,3)  [dot] (topr) {};
		
		\draw[testfcn] (right) to  (root);
		
		\draw[kernel] (left) to (root);
		\draw[rho] (topr) to (topl); 
		\draw[rho] (topl) to (left); 
		\draw[kernel1] (topr) to (right); 
		\draw[rho] (right) to (variable); 
	\end{tikzpicture}\; 
\end{equation*}
and introduce the notation $Q_\eps(s,y)= K(s,y) (\mfc_{\alpha}*\rho_\eps)^{*2}(y)$. We see that $c_\eps=\iint_{\mbR\times \mbG} Q_\eps(z)\dd z $, and define the renormalised kernel
$\msR Q_\eps(s,y):= Q_\eps(s,y)- c_\eps \delta_{(0,e)}$. Then, as in \cite[Eq.~(5.14)]{hairer_pardoux_15_wong_zakai}, using the notation \tikz[baseline=-0.1cm] \draw[kernelBig] (0,0) to (1,0); for the renormalised kernel, one finds
\begin{equation*}
	\left(\hat \Pi_{(0,e)}^{(\eps)} \<Xi3>\right)(\varphi^\lambda) = \;
	\begin{tikzpicture}[scale=0.35,baseline=0.5cm]
		\node at (0,-1)  [root] (root) {};
		\node at (-2,1)  [dot] (left) {};
		\node at (-2,3)  [dot] (left1) {};
		\node at (-2,5)  [dot] (left2) {};
		\node at (0,1) [var] (variable1) {};
		\node at (0,3) [var] (variable2) {};
		\node at (0,5) [var] (variable3) {};
		
		\draw[testfcn] (left) to  (root);
		
		\draw[kernel1] (left1) to (left);
		\draw[kernel1] (left2) to (left1);
		\draw[rho] (variable3) to (left2); 
		\draw[rho] (variable2) to (left1); 
		\draw[rho] (variable1) to (left); 
	\end{tikzpicture}
	\; + \;
	\begin{tikzpicture}[scale=0.35,baseline=0.5cm]
		\node at (0,-1)  [root] (root) {};
		\node at (-2,1)  [dot] (left) {};
		\node at (-2,3)  [dot] (left1) {};
		\node at (-2,5)  [dot] (left2) {};
		\node at (0,5) [var] (variable3) {};
		
		\draw[testfcn] (left) to  (root);
		
		\draw[kernelBig] (left1) to (left);
		\draw[kernel1] (left2) to (left1);
		\draw[rho] (variable3) to (left2); 
	\end{tikzpicture}
	\; - \;
	\begin{tikzpicture}[scale=0.35,baseline=0.5cm]
		\node at (0,-1)  [root] (root) {};
		\node at (-2,1)  [dot] (left) {};
		\node at (-2,3)  [dot] (left1) {};
		\node at (-2,5)  [dot] (left2) {};
		\node at (0,3) [dot] (variable2) {};
		\node at (0,5) [var] (variable3) {};
		
		\draw[testfcn] (left) to  (root);
		
		\draw[kernel1] (left2) to (variable2);
		\draw[rho] (variable3) to (left2); 
		\draw[rho] (left) to (left1); 
		\draw[rho] (variable2) to (left1); 
		\draw[kernel] (variable2) to (root); 
	\end{tikzpicture}
	\; + \;
	\begin{tikzpicture}[scale=0.35,baseline=0.5cm]
		\node at (0,-1)  [root] (root) {};
		\node at (-2,1)  [dot] (left) {};
		\node at (-2,3)  [dot] (jnct) {};
		\node at (-0.75,3)  [dot] (left1) {};
		\node at (-2,5)  [dot] (left2) {};
		\node at (0.5,3) [var] (variable2) {};
		
		\draw[testfcn] (left) to  (root);
		
		\draw[kernel1] (left1) to (left);
		\draw[kernel1] (left2) to (left1);
		\draw[rho] (jnct) to (left2); 
		\draw[rho] (variable2) to (left1); 
		\draw[rho] (jnct) to (left); 
	\end{tikzpicture}
	\; - \;
	\begin{tikzpicture}[scale=0.35,baseline=0.3cm]
		\node at (0,-1)  [root] (root) {};
		\node at (-1,1)  [dot] (left) {};
		\node at (-0.5,2)  [dot] (topl) {};
		\node at (1,1) [dot] (right) {};
		\node at (2,3)  [var] (variable) {};
		\node at (0,3)  [dot] (topr) {};
		
		\draw[testfcn] (right) to  (root);
		
		\draw[kernel] (left) to (root);
		\draw[rho] (topr) to (topl); 
		\draw[rho] (topl) to (left); 
		\draw[kernel1] (topr) to (right); 
		\draw[rho] (right) to (variable); 
	\end{tikzpicture}\;.
\end{equation*}
Similarly,
\begin{equation*}
	\big((\hat{\Pi}^{(\eps)}_{(0,e)}\<IXi2>)\diamond(\hat{\Pi}^{(\eps)}_{(0,e)}\<Xi> )\big)(\varphi^\lambda) =    
	\;
	\begin{tikzpicture}[scale=0.35,baseline=0.5cm]
		\node at (0,-1)  [root] (root) {};
		\node at (-2,1)  [dot] (left) {};
		\node at (-2,3)  [dot] (left1) {};
		\node at (-2,5)  [dot] (left2) {};
		\node at (0,1) [var] (variable1) {};
		\node at (0,3) [var] (variable2) {};
		\node at (0,5) [var] (variable3) {};
		
		\draw[testfcn] (left) to  (root);
		
		\draw[kernel1] (left1) to (left);
		\draw[kernel1] (left2) to (left1);
		\draw[rho] (variable3) to (left2); 
		\draw[rho] (variable2) to (left1); 
		\draw[rho] (variable1) to (left); 
	\end{tikzpicture}
	\; 
	- \;
	\begin{tikzpicture}[scale=0.35,baseline=0.3cm]
		\node at (0,-1)  [root] (root) {};
		\node at (-1,1)  [dot] (left) {};
		\node at (-0.5,2)  [dot] (topl) {};
		\node at (1,1) [dot] (right) {};
		\node at (2,3)  [var] (variable) {};
		\node at (0,3)  [dot] (topr) {};
		
		\draw[testfcn] (right) to  (root);
		
		\draw[kernel] (left) to (root);
		\draw[rho] (topr) to (topl); 
		\draw[rho] (topl) to (left); 
		\draw[kernel1] (topr) to (right); 
		\draw[rho] (right) to (variable); 
	\end{tikzpicture}\;
\end{equation*}
and thus it only remains the bound the covariance of the diagrams
\begin{equation}
	\begin{tikzpicture}[scale=0.35,baseline=0.5cm]
		\node at (0,-1)  [root] (root) {};
		\node at (-2,1)  [dot] (left) {};
		\node at (-2,3)  [dot] (left1) {};
		\node at (-2,5)  [dot] (left2) {};
		\node at (0,5) [var] (variable3) {};
		
		\draw[testfcn] (left) to  (root);
		
		\draw[kernelBig] (left1) to (left);
		\draw[kernel1] (left2) to (left1);
		\draw[rho] (variable3) to (left2); 
	\end{tikzpicture}
	\; , \qquad \;
	\begin{tikzpicture}[scale=0.35,baseline=0.5cm]
		\node at (0,-1)  [root] (root) {};
		\node at (-2,1)  [dot] (left) {};
		\node at (-2,3)  [dot] (left1) {};
		\node at (-2,5)  [dot] (left2) {};
		\node at (0,3) [dot] (variable2) {};
		\node at (0,5) [var] (variable3) {};
		
		\draw[testfcn] (left) to  (root);
		
		\draw[kernel1] (left2) to (variable2);
		\draw[rho] (variable3) to (left2); 
		\draw[rho] (left) to (left1); 
		\draw[rho] (variable2) to (left1); 
		\draw[kernel] (variable2) to (root); 
	\end{tikzpicture}
	\; ,\qquad \;
	\begin{tikzpicture}[scale=0.35,baseline=0.5cm]
		\node at (0,-1)  [root] (root) {};
		\node at (-2,1)  [dot] (left) {};
		\node at (-2,3)  [dot] (jnct) {};
		\node at (-0.75,3)  [dot] (left1) {};
		\node at (-2,5)  [dot] (left2) {};
		\node at (0.5,3) [var] (variable2) {};
		
		\draw[testfcn] (left) to  (root);
		
		\draw[kernel1] (left1) to (left);
		\draw[kernel1] (left2) to (left1);
		\draw[rho] (jnct) to (left2); 
		\draw[rho] (variable2) to (left1); 
		\draw[rho] (jnct) to (left); 
	\end{tikzpicture}
	\; , \;
\end{equation}
The latter two can be bounded by  repeatedly using the analogue of \cite[Lem.~10.14]{hairer_14_RegStruct}, while for the first diagram we use additionally the natural analogue of \cite[Lem.~10.16]{hairer_14_RegStruct}. This concludes the proofs of \ref{it:Xi3_eps_bound} and therefore the first column. 
The bounds on the random variables in \eqref{eq:renormalised_model} are obtained analogously by replacing $\mfc_{\bar{\alpha}}* \rho_\eps$ by just $\mfc_{\bar{\alpha}}$ throughout (and appropriately interpreting the renormalised kernel, $Q_\eps$, for $\varepsilon=0$). 
The approximation bounds \ref{it:Xi2_diff_bound} and \ref{it:Xi3_diff_bound} are obtained by standard arguments, replacing each diagram with a finite telescopic sum of diagrams, where all dotted lines except one either encode convolution with $\mfc_{\bar{\alpha}}* \rho_\eps$ or convolution with $\mfc_{\bar{\alpha}}$ and the one remaining dotted line is replaced by 
$ \mfc_{\bar{\alpha}}- \mfc_{\bar{\alpha}}* \rho_\eps$ and then applying the analogue of \cite[Lem.~10.17]{hairer_14_RegStruct}.
\end{proof}

\begin{remark}
An alternative method to obtain the above estimates on the model would have been to leverage that our noise satisfies a spectral gap inequality, see \cite{hairer_steele_23_spectral}.
We believe it would be interesting to study how the Feynman diagrammatic arguments in \cite{chandra_hairer_16}
as well as the more functional analytic arguments in \cite{hairer_steele_23_spectral}
extend these non-commutative spaces.
\end{remark}

		\bibliographystyle{alpha}
		\bibliography{./HomGroups.bib}

	\end{document}